\documentclass{tac}
\usepackage{amssymb,amsmath,stmaryrd,txfonts,mathrsfs}
\usepackage{arydshln}

\usepackage{color}
\definecolor{myurlcolor}{rgb}{0.5,0,0}
\definecolor{mycitecolor}{rgb}{0,0,1}
\definecolor{myrefcolor}{rgb}{0,0,1}
\definecolor{hyperrefcolor}{rgb}{0,0,0.7}
\usepackage{hyperref}
\hypersetup{
	colorlinks,
	linkcolor={hyperrefcolor},
	citecolor={hyperrefcolor},
	urlcolor={hyperrefcolor}
}

\usepackage[all,2cell,cmtip]{xy}
\UseAllTwocells

\newcommand{\id}{\mathrm{id}}
\newcommand{\La}{\mathcal{L}}
\newcommand{\N}{\mathbb{N}}
\newcommand{\R}{\mathbb{R}}
\newcommand{\Ob}{\mathrm{Ob}}
\newcommand{\Mor}{\mathrm{Mor}}

\newcommand{\Set}{\mathsf{Set}}
\newcommand{\Graph}{\mathsf{Graph}}
\newcommand{\Dynam}{\mathsf{Dynam}}

\newcommand{\Petri}{\mathsf{Petri}}
\newcommand{\A}{\mathsf{A}}
\newcommand{\C}{\mathsf{C}}

\newcommand{\T}{\mathsf{T}}
\newcommand{\D}{\mathsf{D}}
\newcommand{\X}{\mathsf{X}}
\newcommand{\Z}{\mathsf{Z}}

\newcommand{\Fin}{\mathsf{Fin}}
\newcommand{\Rel}{\mathsf{Rel}}
\newcommand{\Cospan}{\mathsf{Cospan}}
\newcommand{\Csp}{\mathsf{Csp}}
\newcommand{\Circ}{\mathsf{Circ}}
\newcommand{\CMC}{\mathsf{CMC}}
\newcommand{\RxNet}{\mathsf{RxNet}}

\newcommand{\Dbl}{\mathbf{Dbl}}
\newcommand{\bCsp}{\mathbf{Csp}}
\newcommand{\Cat}{\mathbf{Cat}}
\newcommand{\Rex}{\mathbf{Rex}}
\newcommand{\SMC}{\mathbf{SymMonCat}}
\newcommand{\B}{\mathbf{B}}

\newcommand{\lCsp}{\mathbb{C}\mathbf{sp}}
\newcommand{\Open}{\mathbb{O}\mathbf{pen}}
\newcommand{\lCospan}{\mathbb{C}\mathbf{ospan}}

\newcommand{\define}[1]{{\bf \boldmath{#1}}}

\newcommand{\lA}{\ensuremath{\mathbb{A}}}
\newcommand{\lB}{\ensuremath{\mathbb{B}}}
\newcommand{\lC}{\ensuremath{\mathbb{C}}}
\newcommand{\lD}{\ensuremath{\mathbb{D}}}

\newcommand{\lF}{\ensuremath{\mathbb{F}}}
\newcommand{\lG}{\ensuremath{\mathbb{G}}}

\newcommand{\lX}{\ensuremath{\mathbb{X}}}

\newcommand{\bB}{\ensuremath{\mathbf{B}}}
\newcommand{\bC}{\ensuremath{\mathbf{C}}}
\newcommand{\bD}{\ensuremath{\mathbf{D}}}

\newcommand{\bX}{\ensuremath{\mathbf{X}}}

\newcommand{\To}{\Rightarrow}
\newcommand{\maps}{\colon}

\def\toiso{\xrightarrow{\smash{\raisebox{-.5mm}{$\;\scriptstyle\sim\;$}}}}

\theoremstyle{plain}

\newtheorem{thm}{Theorem}[section]
\newtheorem{lem}[thm]{Lemma}
\newtheorem{prop}[thm]{Proposition}
\newtheorem{cor}[thm]{Corollary}

\newtheorem{defn}[thm]{Definition}

\theoremstyle{remark}

\usepackage{tikz}
\usepackage{tikz-cd}
\usetikzlibrary{matrix,arrows,decorations.pathmorphing,positioning}
\usetikzlibrary{intersections,decorations.markings}
\usetikzlibrary{arrows,positioning,fit,matrix,shapes.geometric,external}
\usetikzlibrary{backgrounds,circuits,circuits.ee.IEC,shapes,fit,matrix}

\usetikzlibrary{matrix,arrows}
\definecolor{rewritecolor}{rgb}{0,.9,1}
\tikzset{rewritenode/.style={shape=circle,fill=rewritecolor,scale=0.25,font=\Huge}}
\tikzset{RWopen/.style={shape=circle,draw=black,fill=white,scale=0.5,font=\Huge}}
\tikzset{RWclosed/.style={shape=circle,fill=black,scale=0.5,font=\Huge}}
\tikzset{CDnode/.style={shape=circle,fill=white,scale=.5}}
\makeatletter
\let\ea\expandafter

\pgfdeclarelayer{edgelayer}
\pgfdeclarelayer{nodelayer}
\pgfsetlayers{edgelayer,nodelayer,main}

\newcommand\SWarrow{\mathrel{\rotatebox[origin=c]{-135}{$\Rightarrow$}}}
\newcommand\SEarrow{\mathrel{\rotatebox[origin=c]{-45}{$\Rightarrow$}}}

\newcommand\NEarrow{\mathrel{\rotatebox[origin=c]{45}{$\Rightarrow$}}}

\definecolor{lblue}{rgb}{0,250,255}
\tikzstyle{species}=[circle,fill=yellow,draw=black,scale=1.15]
\tikzstyle{transition}=[rectangle,fill=lblue,draw=black,scale=1.15]
\tikzstyle{inarrow}=[->, >=stealth, shorten >=.03cm,line width=1.5]
\tikzstyle{empty}=[circle,fill=none, draw=none]
\tikzstyle{inputdot}=[circle,fill=purple,draw=purple, scale=.25]
\tikzstyle{inputarrow}=[->,draw=purple, shorten >=.05cm]
\tikzstyle{simple}=[-,draw=purple,line width=1.000]
\tikzstyle{none}=[inner sep=0pt]

\tikzset{->-/.style={decoration={
  markings,
  mark=at position .5 with {\arrow{>}}},postaction={decorate}}}

\tikzstyle{none}=[inner sep=0pt]
\tikzstyle{connection}=[circuit symbol open,
    circuit symbol size=width .5 height .5,
    shape=circle ee,
    inner sep=0.25\pgflinewidth]
\tikzstyle{bdot}=[circle, fill=black, draw, inner sep=1.5pt, anchor=center]
\tikzstyle{circ}=[circle,fill=black,draw,inner sep=3pt]


\copyrightyear{2020}


\begin{document}
\sloppy

\title{Structured cospans}
\author{John\ C.\ Baez and Kenny Courser}
\maketitle

\address{Department of Mathematics, University of California, 
 Riverside, CA 92521, USA\\[5pt]}

\eaddress{%
baez@math.ucr.edu \\
\null \hspace{2.6em} courser@math.ucr.edu}


\vspace{-0.4em}
\begin{abstract}
\noindent
One goal of applied category theory is to better understand networks appearing throughout science and engineering.  Here we introduce `structured cospans' as a way to study networks with inputs and outputs.   Given a functor $L \maps \A \to \X$, a structured cospan is a diagram in $\X$ of the form $L(a) \rightarrow x \leftarrow L(b)$.   If $\A$ and $\X$ have finite colimits and $L$ is a left adjoint, we obtain a symmetric monoidal category whose objects are those of $\A$ and whose morphisms are isomorphism classes of structured cospans.     This is a hypergraph category.   However, it arises from a more fundamental structure: a symmetric monoidal double category where the horizontal 1-cells are structured cospans.   We show how structured cospans solve certain problems in the closely related formalism of `decorated cospans', and explain how they work in some examples: electrical circuits, Petri nets, and 
chemical reaction networks.
\end{abstract}

 \vspace{-0.4em}
\section{Introduction}
\label{sec:intro}

Structured cospans are a framework for dealing with open networks: that is, networks with inputs and outputs.  Networks arise in many areas of science and engineering and come in many kinds, but a companion paper illustrates the general framework developed here with the example of open Petri nets \cite{BM}, so let us consider those.

Petri nets are important in computer science, chemistry and other subjects.   For example, the chemical reaction that takes two atoms of hydrogen and one atom of oxygen and produces a molecule of water can be represented by this very simple Petri net:
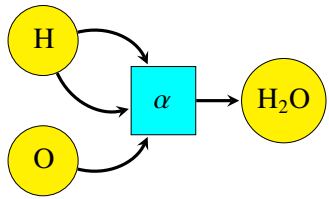
\[
\scalebox{0.8}{
\begin{tikzpicture}
	\begin{pgfonlayer}{nodelayer}
		\node [style=species] (I) at (0,1) {$\;$H$\phantom{|}$};
		\node [style=species] (T) at (0,-1) {$\;$O$\phantom{|}$};
		\node [style=transition] (W) at (2,0) {$\,\;\alpha\;\phantom{\Big|}$};
		\node [style=species] (Water) at (4,0) {$\textnormal{H}_2$O};
	\end{pgfonlayer}
	\begin{pgfonlayer}{edgelayer}
		\draw [style=inarrow, bend right=40, looseness=1.00] (I) to (W);
		\draw [style=inarrow, bend left=40, looseness=1.00] (I) to (W);
		\draw [style=inarrow, bend right=40, looseness=1.00] (T) to (W);
		\draw [style=inarrow] (W) to (Water);
	\end{pgfonlayer}
\end{tikzpicture}
}
\]
Here we have a set of `places' (or in chemistry, `species') drawn in yellow 
and a set of `transitions' (or `reactions') drawn in blue.   The disjoint union of these two sets then forms the vertex set of a directed bipartite graph, which is one description of a Petri net.

Networks can often be seen as pieces of larger networks.  This naturally leads to the idea of an \emph{open} Petri net, meaning that the set of places is equipped with `inputs' and `outputs'. We can do this by prescribing two functions into the set of places that pick out these inputs and outputs. For example:
\[
\scalebox{0.8}{
\begin{tikzpicture}
	\begin{pgfonlayer}{nodelayer}
		\node [style=species] (I) at (-2,1) {$\;$H$\phantom{|}$};
		\node [style=species] (T) at (-2,-1) {$\;$O$\phantom{|}$};
		\node [style=transition] (W) at (0,0) {$\,\;\alpha\;\phantom{\Big|}$};
		\node [style=species] (Water) at (1.5,0) {$\textnormal{H}_2$O};
		\node [style=none] (1) at (-3.25,1) {1};
		\node [style=none] (2'') at (-3.25,-0.3) {2};
		\node [style=none] (2) at (-3.25,-1) {3};
		\node [style=none] (ATL) at (-3.65,1.4) {};
		\node [style=none] (ATR) at (-2.85,1.4) {};
		\node [style=none] (ABR) at (-2.85,-1.4) {};
		\node [style=none] (ABL) at (-3.65,-1.4) {};
		\node [style=none] (X) at (-3.25,1.8) {$a$};
		\node [style=inputdot] (inI) at (-3.05,1) {};
		\node [style=inputdot] (inS) at (-3.05,-1) {};
		\node [style=inputdot] (inI') at (-3.05,-0.3) {};
		\node [style=none] (Z) at (3,1.8) {$b$};
		\node [style=none] (1'') at (3,0) {4};
		\node [style=none] (MTL) at (2.6,1.4) {};
		\node [style=none] (MTR) at (3.4,1.4) {};
		\node [style=none] (MBR) at (3.4,-1.4) {};
		\node [style=none] (MBL) at (2.6,-1.4) {};
		\node [style=inputdot] (outI') at (2.8,0) {};
	\end{pgfonlayer}
	\begin{pgfonlayer}{edgelayer}
		\draw [style=inarrow, bend right=40, looseness=1.00] (I) to (W);
		\draw [style=inarrow, bend left=40, looseness=1.00] (I) to (W);
		\draw [style=inarrow, bend right=40, looseness=1.00] (T) to (W);
		\draw [style=inarrow] (W) to (Water);
		\draw [style=simple] (ATL.center) to (ATR.center);
		\draw [style=simple] (ATR.center) to (ABR.center);
		\draw [style=simple] (ABR.center) to (ABL.center);
		\draw [style=simple] (ABL.center) to (ATL.center);
		\draw [style=simple] (MTL.center) to (MTR.center);
		\draw [style=simple] (MTR.center) to (MBR.center);
		\draw [style=simple] (MBR.center) to (MBL.center);
		\draw [style=simple] (MBL.center) to (MTL.center);
		\draw [style=inputarrow] (inI) to (I);
		\draw [style=inputarrow] (inI') to (T);
		\draw [style=inputarrow] (inS) to (T);
		\draw [style=inputarrow] (outI') to (Water);
	\end{pgfonlayer}
\end{tikzpicture}
}
\]
The inputs and outputs let us compose open Petri nets.  For example, suppose we have another open Petri net that represents the chemical reaction of two molecules of water turning into hydronium and hydroxide:
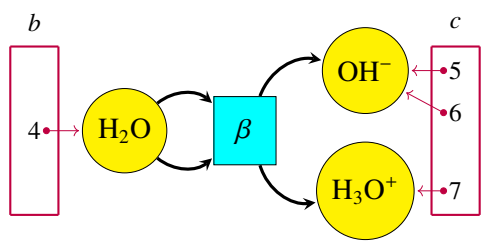
\[
\scalebox{0.8}{
\begin{tikzpicture}
	\begin{pgfonlayer}{nodelayer}
		\node [style=species] (Water2) at (4,0) {$\textnormal{H}_2$O};
		\node [style=transition] (Something) at (6,0) {$\,\;\beta\;\phantom{\Big|}$};
		\node [style=species] (A) at (8,1) {O$\textnormal{H}^{-}$};
		\node [style=species] (B) at (8,-1) {$\textnormal{H}_3 \textnormal{O}^{+}$};
		\node [style=none] (1') at (9.5,1) {5};
		\node [style=none] (3') at (9.5,0.3) {6};
		\node [style=none] (2') at (9.5,-1) {7};
		\node [style=none] (BTL) at (9.1,1.4) {};
		\node [style=none] (BTR) at (9.9,1.4) {};
		\node [style=none] (BBR) at (9.9,-1.4) {};
		\node [style=none] (BBL) at (9.1,-1.4) {};
		\node [style=none] (Y) at (9.5,1.8) {$c$};
		\node [style=inputdot] (outI) at (9.3,1) {};
		\node [style=inputdot] (outR) at (9.3,0.3) {};
		\node [style=inputdot] (outS) at (9.3,-1) {};
		\node [style=none] (Z) at (2.5,1.8) {$b$};
		\node [style=none] (1'') at (2.5,0) {4};
		\node [style=none] (MTL) at (2.1,1.4) {};
		\node [style=none] (MTR) at (2.9,1.4) {};
		\node [style=none] (MBR) at (2.9,-1.4) {};
		\node [style=none] (MBL) at (2.1,-1.4) {};
		\node [style=inputdot] (inI') at (2.7,0) {}; 
	\end{pgfonlayer}
	\begin{pgfonlayer}{edgelayer}
		\draw [style=inarrow, bend left=40, looseness=1.00] (Water2) to (Something);
		\draw [style=inarrow, bend right=40, looseness=1.00] (Water2) to (Something);
		\draw [style=inarrow, bend left=40, looseness=1.00] (Something) to (A);
		\draw [style=inarrow, bend right=40, looseness=1.00] (Something) to (B);
		\draw [style=simple] (BTL.center) to (BTR.center);
		\draw [style=simple] (BTR.center) to (BBR.center);
		\draw [style=simple] (BBR.center) to (BBL.center);
		\draw [style=simple] (BBL.center) to (BTL.center);
		\draw [style=simple] (MTL.center) to (MTR.center);
		\draw [style=simple] (MTR.center) to (MBR.center);
		\draw [style=simple] (MBR.center) to (MBL.center);
		\draw [style=simple] (MBL.center) to (MTL.center);
		\draw [style=inputarrow] (outI) to (A);
		\draw [style=inputarrow] (outR) to (A);
		\draw [style=inputarrow] (outS) to (B);
		\draw [style=inputarrow] (inI') to (Water2);
	\end{pgfonlayer}
\end{tikzpicture}
}
\]
Since the outputs of the first open Petri net coincide with the inputs of the second, we can compose them by identifying the outputs of the first with the inputs of the second:  
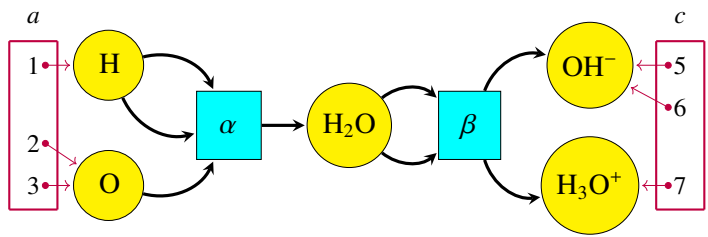
\[
\scalebox{0.8}{
\begin{tikzpicture}
	\begin{pgfonlayer}{nodelayer}
		\node [style=species] (I) at (0,1) {$\;$H$\phantom{|}$};
		\node [style=species] (T) at (0,-1) {$\;$O$\phantom{|}$};
		\node [style=transition] (W) at (2,0) {$\,\;\alpha\;\phantom{\Big|}$};
		\node [style=species] (Water) at (4,0) {$\textnormal{H}_2$O};
		\node [style=transition] (Something) at (6,0) {$\,\;\beta\;\phantom{\Big|}$};
		\node [style=species] (A) at (8,1) {O$\textnormal{H}^{-}$};
		\node [style=species] (B) at (8,-1) {$\textnormal{H}_3 \textnormal{O}^{+}$};
		\node [style=none] (1) at (-1.25,1) {1};
		\node [style=none] (2a) at (-1.25,-0.3) {2};
		\node [style=none] (2) at (-1.25,-1) {3};
		\node [style=none] (1') at (9.5,1) {5};
		\node [style=none] (1'') at (9.5,0.3) {6};
		\node [style=none] (2') at (9.5,-1) {7};
		\node [style=none] (ATL) at (-1.65,1.4) {};
		\node [style=none] (ATR) at (-0.85,1.4) {};
		\node [style=none] (ABR) at (-0.85,-1.4) {};
		\node [style=none] (ABL) at (-1.65,-1.4) {};
		\node [style=none] (BTL) at (9.1,1.4) {};
		\node [style=none] (BTR) at (9.9,1.4) {};
		\node [style=none] (BBR) at (9.9,-1.4) {};
		\node [style=none] (BBL) at (9.1,-1.4) {};
		\node [style=none] (X) at (-1.25,1.8) {$a$};
		\node [style=none] (Y) at (9.5,1.8) {$c$};
		\node [style=inputdot] (outI) at (9.3,1) {};
		\node [style=inputdot] (outR) at (9.3,0.3) {};
		\node [style=inputdot] (outS) at (9.3,-1) {};
		\node [style=inputdot] (inI) at (-1.05,1) {};
		\node [style=inputdot] (inI') at (-1.05,-0.3) {};
		\node [style=inputdot] (inS) at (-1.05,-1) {};
	\end{pgfonlayer}
	\begin{pgfonlayer}{edgelayer}
		\draw [style=inarrow, bend right=40, looseness=1.00] (I) to (W);
		\draw [style=inarrow, bend left=40, looseness=1.00] (I) to (W);
		\draw [style=inarrow, bend right=40, looseness=1.00] (T) to (W);
		\draw [style=inarrow] (W) to (Water);
		\draw [style=inarrow, bend left=40, looseness=1.00] (Water) to (Something);
		\draw [style=inarrow, bend right=40, looseness=1.00] (Water) to (Something);
		\draw [style=inarrow, bend left=40, looseness=1.00] (Something) to (A);
		\draw [style=inarrow, bend right=40, looseness=1.00] (Something) to (B);
		\draw [style=simple] (ATL.center) to (ATR.center);
		\draw [style=simple] (ATR.center) to (ABR.center);
		\draw [style=simple] (ABR.center) to (ABL.center);
		\draw [style=simple] (ABL.center) to (ATL.center);
		\draw [style=simple] (BTL.center) to (BTR.center);
		\draw [style=simple] (BTR.center) to (BBR.center);
		\draw [style=simple] (BBR.center) to (BBL.center);
		\draw [style=simple] (BBL.center) to (BTL.center);
		\draw [style=inputarrow] (outI) to (A);
		\draw [style=inputarrow] (outR) to (A);
		\draw [style=inputarrow] (outS) to (B);
		\draw [style=inputarrow] (inI) to (I);
		\draw [style=inputarrow] (inI') to (T);
		\draw [style=inputarrow] (inS) to (T);
	\end{pgfonlayer}
\end{tikzpicture}
}
\]
Similarly we can `tensor' two open Petri nets by placing them side by side:
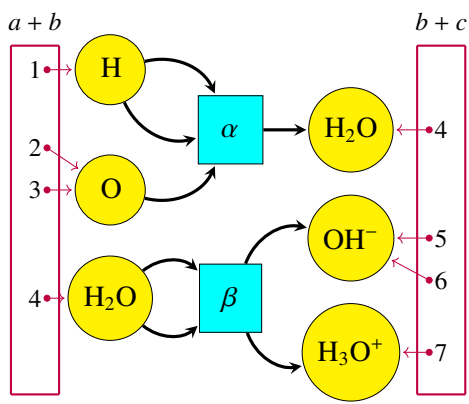
\[
\scalebox{0.8}{
\begin{tikzpicture}
	\begin{pgfonlayer}{nodelayer}
		\node [style=inputdot] (inI) at (2.95,2) {};
		\node [style=inputdot] (inS) at (2.95,4) {};
		\node [style=inputdot] (outI') at (9.3,3) {};
		\node [style=species] (I) at (4,4) {$\;$H$\phantom{|}$};
		\node [style=species] (T) at (4,2) {$\;$O$\phantom{|}$};
		\node [style=transition] (W) at (6,3) {$\,\;\alpha\;\phantom{\Big|}$};
		\node [style=species] (Water) at (8,3) {$\textnormal{H}_2$O};
		\node [style=none] (1) at (2.75,4) {1};
		\node [style=none] (2a) at (2.75,2.7) {2};
		\node [style=none] (2) at (2.75,2) {3};
		\node [style=none] (1''') at (9.5,3) {4};
		\node [style=species] (Water2) at (4,0.2) {$\textnormal{H}_2$O};
		\node [style=transition] (Something) at (6,0.2) {$\,\;\beta\;\phantom{\Big|}$};
		\node [style=species] (A) at (8,1.2) {O$\textnormal{H}^{-}$};
		\node [style=species] (B) at (8,-0.7) {$\textnormal{H}_3 \textnormal{O}^{+}$};
		\node [style=none] (1') at (9.5,1.2) {5};
		\node [style=none] (1'') at (9.5,0.5) {6};
		\node [style=none] (2') at (9.5,-0.7) {7};
		\node [style=none] (BTL) at (9.1,4.4) {};
		\node [style=none] (BTR) at (9.9,4.4) {};
		\node [style=none] (BBR) at (9.9,-1.4) {};
		\node [style=none] (BBL) at (9.1,-1.4) {};
		\node [style=none] (Y) at (9.5,4.8) {$b+c$};
		\node [style=inputdot] (outI) at (9.3,1.2) {};
		\node [style=inputdot] (outR) at (9.3,0.5) {};
		\node [style=none] (Z) at (2.75,4.8) {$a+b$};
		\node [style=none] (1'') at (2.75,0.2) {4};
		\node [style=none] (MTL) at (2.35,4.4) {};
		\node [style=none] (MTR) at (3.15,4.4) {};
		\node [style=none] (MBR) at (3.15,-1.4) {};
		\node [style=none] (MBL) at (2.35,-1.4) {};
		\node [style=inputdot] (inI') at (2.95,0.2) {}; 
		\node [style=inputdot] (inI'') at (2.95,2.7) {}; 
		\node [style=inputdot] (outS) at (9.3,-0.7) {};
	\end{pgfonlayer}
	\begin{pgfonlayer}{edgelayer}
		\draw [style=inarrow, bend right=40, looseness=1.00] (I) to (W);
		\draw [style=inarrow, bend left=40, looseness=1.00] (I) to (W);
		\draw [style=inarrow, bend right=40, looseness=1.00] (T) to (W);
		\draw [style=inarrow] (W) to (Water);
		\draw [style=inarrow, bend left=40, looseness=1.00] (Water2) to (Something);
		\draw [style=inarrow, bend right=40, looseness=1.00] (Water2) to (Something);
		\draw [style=inarrow, bend left=40, looseness=1.00] (Something) to (A);
		\draw [style=inarrow, bend right=40, looseness=1.00] (Something) to (B);
		\draw [style=simple] (BTL.center) to (BTR.center);
		\draw [style=simple] (BTR.center) to (BBR.center);
		\draw [style=simple] (BBR.center) to (BBL.center);
		\draw [style=simple] (BBL.center) to (BTL.center);
		\draw [style=simple] (MTL.center) to (MTR.center);
		\draw [style=simple] (MTR.center) to (MBR.center);
		\draw [style=simple] (MBR.center) to (MBL.center);
		\draw [style=simple] (MBL.center) to (MTL.center);
		\draw [style=inputarrow] (outI) to (A);
		\draw [style=inputarrow] (outR) to (A);
		\draw [style=inputarrow] (outS) to (B);
		\draw [style=inputarrow] (inI') to (Water2);
		\draw [style=inputarrow] (inI) to (T);
		\draw [style=inputarrow] (inI'') to (T);
		\draw [style=inputarrow] (inS) to (I);
		\draw [style=inputarrow] (outI') to (Water);
	\end{pgfonlayer}
\end{tikzpicture}
}
\]

We can formalize this example using `structured cospans'.
Given a functor $L \maps \A \to \X$, a \define{structured cospan} is a diagram in $\X$ of the form
\[ \xymatrix{  & x  &  \\ L(a) \ar[ur]^{i} & & L(b). \ar[ul]_{o} } \]
The objects $a$ and $b$ are called the \define{input} and \define{output}, respectively, while
$x$ is called the \define{apex}.  The morphisms $i$ and $o$ are called the \define{legs}
of the cospan.   

Typically the input and output of a structured cospan are simpler in nature than the apex.  For example, an open Petri net is a structured cospan where $a$ and $b$ are sets while $x$ is a Petri net.  As explained in Section \ref{subsec:Petri}, there is a category $\Petri$ with Petri nets as objects and a functor $L \maps \Set \to \Petri$ sending any set to the Petri net with that set of places and no transitions.    Furthermore, $L$ is a left adjoint, so it preserves colimits.   This occurs in many examples.

Given a functor $L \maps \A \to \X$, we can compose structured cospans whenever $\X$ has pushouts.  In Corollary \ref{_L Csp(X) category} we show this gives a category ${}_L \Csp(\X)$ with:
\begin{itemize}
\item objects of $\A$ as objects,
\item isomorphism classes of structured cospans as morphisms.
\end{itemize}
Here we say two structured cospans $L(a) \rightarrow x \leftarrow L(b)$ and $L(a) \rightarrow y \leftarrow L(b)$ are \define{isomorphic} if there is an isomorphism $f \maps x \to y$ such that the diagram
\[
\begin{tikzpicture}[scale=1.5]
\node (E) at (3,-0.5) {$L(a)$};
\node (F) at (5,-0.5) {$L(b)$};
\node (G) at (4,0) {$x$};
\node (G') at (4,-1) {$y$};
\path[->,font=\scriptsize,>=angle 90]
(F) edge node[above]{$$} (G)
(G) edge node[left]{$f$} (G')
(E) edge node[above]{$$} (G)
(E) edge node[above]{$$} (G')
(F) edge node[below]{$$} (G');
\end{tikzpicture}
\]
commutes.    In Corollary \ref{cor:LCsp(X) symmetric category} we show this category $_L\Csp(\X)$ becomes symmetric monoidal when $\A$ and $\X$ have finite colimits and $L$ preserves them.  
Under these assumptions, in Theorem \ref{thm:hypergraph} we prove that $_L\Csp(\X)$ is actually a special sort of symmetric monoidal category called a `hypergraph category' \cite{FongSpivak}.   These are important in the theory of networks \cite{Fong2015,Fong2016}.

Sometimes it is inconvenient to work with isomorphism classes of structured cospans.  For example, in an open Petri net we can refer to a particular place or transition; in an isomorphism class of open Petri nets we cannot.   To use actual structured cospans as morphisms we need a higher categorical structure, because composing them is associative only up to isomorphism.  Indeed, in Corollary \ref{cor:LCsp(X) bicategory} we show that for any functor $L \maps \A \to \X$, if $\X$ has pushouts there is a bicategory ${}_L \bCsp(\X)$ with:
\begin{itemize}
\item objects of $\A$ as objects,
\item structured cospans as 1-morphisms,
\item commutative diagrams
\[
\begin{tikzpicture}[scale=1.5]
\node (E) at (3,-0.5) {$L(a)$};
\node (F) at (5,-0.5) {$L(b)$};
\node (G) at (4,0) {$x$};
\node (G') at (4,-1) {$y$};

\path[->,font=\scriptsize,>=angle 90]
(F) edge node[above]{$$} (G)
(G) edge node[left]{$f$} (G')
(E) edge node[above]{$$} (G)
(E) edge node[above]{$$} (G')
(F) edge node[below]{$$} (G');
\end{tikzpicture}
\]
as 2-morphisms.
\end{itemize}

In Corollary \ref{cor:LCsp(X) symmetric bicategory} we show that the bicategory $_L \bCsp(\X)$ is symmetric monoidal when $\A$ and $\X$ have finite colimits and $L$ preserves them.    
However, the coherence laws for a symmetric monoidal bicategory are rather complicated \cite{Stay}.   As noted by Ehresmann \cite{Ehresmann}, and then by Grandis and Par\'e \cite{GP1,GP2}, double categories are sometimes more convenient than bicategories.   This is especially true in the symmetric monoidal case \cite{HS, Shul2010}.  Thus we show in Theorem \ref{thm:LCsp(X)} that for any functor $L \maps \A \to \X$, if $\X$ has pushouts there is a double category $_L \lCsp(\X)$ with:
\begin{itemize}
\item objects of $\A$ as objects,
\item morphisms of $\A$ as vertical 1-morphisms,
\item structured cospans as horizontal 1-cells,
\item commutative diagrams
\[
\begin{tikzpicture}[scale=1.5]
\node (E) at (3,0) {$L(a)$};
\node (F) at (5,0) {$L(b)$};
\node (G) at (4,0) {$x$};
\node (E') at (3,-1) {$L(a')$};
\node (F') at (5,-1) {$L(b')$};
\node (G') at (4,-1) {$x'$};
\path[->,font=\scriptsize,>=angle 90]
(F) edge node[above]{$o$} (G)
(E) edge node[left]{$L(\alpha)$} (E')
(F) edge node[right]{$L(\beta)$} (F')
(G) edge node[left]{$f$} (G')
(E) edge node[above]{$i$} (G)
(E') edge node[below]{$i'$} (G')
(F') edge node[below]{$o'$} (G');
\end{tikzpicture}
\]
as 2-morphisms.
\end{itemize}
Note that vertical composition in this double category is strictly associative, while horizontal composition is not.  In Theorem \ref{thm:LCsp(X) symmetric} we show that  $_L \lCsp(\X)$ is a symmetric monoidal double category when $\A$ and $\X$ have finite colimits and $L$ preserves them.  Using Shulman's work \cite{Shul2010}, we conclude in Corollary \ref{cor:LCsp(X) symmetric bicategory} that the bicategory $_L \bCsp(\X)$ is a symmetric monoidal bicategory under the same conditions.

The reader familiar with decorated cospans may wonder why we need structured
cospans.  Recall that Fong \cite{Fong2015} constructed a category of `decorated cospans' $F\Cospan$ from any category $\A$ with finite colimits together with a symmetric lax monoidal functor $F \maps (\A,+) \to (\Set,\times)$.   The objects of $F\Cospan$ are those of $\A$, while the morphisms are equivalence classes of $F$-decorated cospans.   Here an \define{$F$-decorated cospan} is a pair
\[
\begin{tikzpicture}[scale=1.5]
\node (A) at (0,0) {$(a$};
\node (B) at (1,0) {$s$};
\node (C) at (2,0) {$b$,};
\node (D) at (2.8,0) {$d \in F(s))$.};
\path[->,font=\scriptsize,>=angle 90]
(A) edge node[above]{$i$} (B)
(C)edge node[above]{$o$}(B);
\end{tikzpicture}
\]
The element $d$, called the \define{decoration}, serves as a way to equip the apex $s$ with extra structure.   The above decorated cospan is equivalent to
\[
\begin{tikzpicture}[scale=1.5]
\node (A) at (0,0) {$(a$};
\node (B) at (1,0) {$s'$};
\node (C) at (2,0) {$b$,};
\node (D) at (2.8,0) {$d' \in F(s'))$};
\path[->,font=\scriptsize,>=angle 90]
(A) edge node[above]{$i'$} (B)
(C)edge node[above]{$o'$}(B);
\end{tikzpicture}
\]
iff there an isomorphism $f \maps s \to s'$ in $\A$ making this diagram commute:
\[
\begin{tikzpicture}[scale=1.5]
\node (E) at (3,-0.5) {$a$};
\node (F) at (5,-0.5) {$b$};
\node (G) at (4,0) {$s$};
\node (G') at (4,-1) {$s'$};
\path[->,font=\scriptsize,>=angle 90]
(F) edge node[above]{$o$} (G)
(G) edge node[left]{$f$} (G')
(E) edge node[above]{$i$} (G)
(E) edge node[below]{$i'$} (G')
(F) edge node[below]{$o'$} (G');
\end{tikzpicture}
\]
and such that $F(f)(d) = d'$.

Both decorated and structured cospans are ways to describe
a cospan whose apex is equipped with extra structure.   Since the theory of decorated
cospans is already well-developed, what is the point of another formalism?   One reason is that structured cospans are a bit simpler: instead of a symmetric lax monoidal functor $F \maps \A \to \Set$ assigning to each object of $\A$ the set of possible structures we can put on it, we can simply use a left adjoint $L$ from $\A$ to any category $\X$.  Another reason is that
structured cospans solve some problems that prevent decorated cospans from being applied as
originally intended, and indeed led to errors in a number of published papers.   
We discuss these problems, and how structured cospans get around them, in Section \ref{sec:decorated}.  In Section \ref{sec:applications} we study applications of structured cospans to electrical circuits, Petri nets and chemical reaction networks.  

\subsection*{Conventions}

In this paper, `double category' means `pseudo double category', as in Definition \ref{defn:double_category}.   Following Shulman \cite{Shul2010}, vertical composition 
in our double categories is strictly associative, while horizontal composition need not be.
We use sans-serif font like $\C$ for categories, boldface like $\B$ for bicategories or
2-categories, and blackboard bold like $\lD$ for double categories.   We also use blackboard
bold for weak category objects in any 2-category.  For double categories with names having more than one letter, like $\lCsp(\X)$, only the first letter is in blackboard bold.  A double category $\lD$ has a category of objects and a category of arrows, and we call these $\lD_0$ and $\lD_1$ despite the fact that they are categories.   

\subsection*{Acknowledgements}

The authors would like to thank Christina Vasilakopoulou for the clever idea of replacing the category of objects of some double category by some other category.   We would also like to thank Marco Grandis and Robert Par\'e for pointing out the importance of double categories with double colimits, and Joachim Kock and Mike Shulman for catching errors.


\section{Structured cospans}
\label{sec:structured}

Given a functor $L \maps \A \to \X$, a \define{structured cospan} is a cospan in $\X$ whose feet come from a pair of objects in $\A$:
\[
\begin{tikzpicture}[scale=1.2]
\node (A) at (0,0) {$L(a)$};
\node (B) at (1,1) {$x$};
\node (C) at (2,0) {$L(b).$};
\path[->,font=\scriptsize,>=angle 90]
(A) edge node[above]{$$} (B)
(C)edge node[above]{$$}(B);
\end{tikzpicture}
\]
When $L$ has a right adjoint $R \maps \X \to \A$ we can also think of this as a cospan in $\A$,
\[
\begin{tikzpicture}[scale=1.2]
\node (A) at (0,0) {$a$};
\node (B) at (1,1) {$R(x)$};
\node (C) at (2,0) {$b$,};
\path[->,font=\scriptsize,>=angle 90]
(A) edge node[above]{$$} (B)
(C)edge node[above]{$$}(B);
\end{tikzpicture}
\]
where the apex is equipped with extra structure, namely an object $x \in \X$ that it comes from.    However, treating structured cospans as living in $\X$ is technically more convenient, since then we only need $\X$ to have pushouts to compose them.
In Theorem \ref{thm:LCsp(X)} we show that when $\X$ has pushouts, structured cospans are the horizontal 1-cells of a double category ${}_L \lCsp(\X)$.   To prove this we begin by recalling the double category of cospans in $\X$.  For the definition of double category see Appendix \ref{appendix}.

\begin{lem} \label{lem:dubcsp}
Given a category $\X$ with chosen pushouts, there is a double category $\lCsp(\X)$ in which:
\begin{itemize}
\item an object is an object of $\X$,
\item a vertical 1-morphism is a morphism of $\X$,
\item a horizontal 1-cell from $x_1$ to $x_2$ is a cospan in $\X$:
\[
\begin{tikzpicture}[scale=1.5]
\node (A) at (0,0) {$x_1$};
\node (B) at (1,0) {$y$};
\node (C) at (2,0) {$x_2$};
\path[->,font=\scriptsize,>=angle 90]
(A) edge node[above]{$i$} (B)
(C)edge node[above]{$o$}(B);
\end{tikzpicture}
\]
\item a 2-morphism is a commutative diagram in $\X$ of this form:
\[
\begin{tikzpicture}[scale=1.5]
\node (E) at (3,0) {$x_1$};
\node (G) at (4,0) {$y$};
\node (F) at (5,0) {$x_2$};
\node (E') at (3,-1) {$x_1'$};
\node (G') at (4,-1) {$y'$};
\node (F') at (5,-1) {$x_2'$,};
\path[->,font=\scriptsize,>=angle 90]
(E) edge node[left]{$f_1$} (E')
(F) edge node[right]{$f_2$} (F')
(G) edge node[left]{$g$} (G')
(E) edge node[above]{$i$} (G)
(F) edge node[above]{$o$} (G)
(E') edge node[below]{$i'$} (G')
(F') edge node[below]{$o'$} (G');
\end{tikzpicture}
\]
\item composition of vertical 1-morphisms is composition in $\X$,
\item composition of horizontal 1-cells is done using the chosen pushouts
in $\X$:
\[
\begin{tikzpicture}[scale=1.5]
\node (A) at (0,0) {$x_1$};
\node (B) at (1,0.75) {$y$};
\node (C) at (2,0) {$x_2$};
\node (D) at (3,0.75) {$z$};
\node (E) at (4,0) {$x_3$};
\node (F) at (2,1.5) {$y+_{x_2} z$};
\path[->,font=\scriptsize,>=angle 90]
(A) edge node[above left=-1mm]{$i_1$} (B)
(C) edge node[above right=-1mm]{$o_1$} (B)
(C) edge node [above left=-1mm] {$i_2$} (D)
(E) edge node [above right=-1mm] {$o_2$} (D)
(B) edge node [above] {$j_y$} (F)
(D) edge node [above] {$j_z$} (F);
\end{tikzpicture}
\]
where $j_y$ and $j_z$ are the canonical morphisms from $y$ and $z$ to the pushout object, 
\item the horizontal composite of two 2-morphisms:
\[
\begin{tikzpicture}[scale=1.5]
\node (D) at (-3,1) {$x_1$};
\node (E) at (-2,1) {$y$};
\node (F) at (-1,1) {$x_2$};
\node (A) at (-3,0) {$x_1'$};
\node (G) at (-2,0) {$y'$};
\node (B) at (-1,0) {$x_2'$};

\node (D') at (0.5,1) {$x_2$};
\node (E') at (1.5,1) {$z$};
\node (F') at (2.5,1) {$x_3$};
\node (C) at (0.5,0) {$x_2'$};
\node (G') at (1.5,0) {$z'$};
\node (H) at (2.5,0) {$x_3'$};

\path[->,font=\scriptsize,>=angle 90]
(D) edge node[above] {$i_1$}(E)
(A) edge node[above] {$i'_1$} (G)
(B) edge node [above]{$o'_1$} (G)
(F) edge node [above]{$o_1$}(E)
(D) edge node [left]{$f_1$}(A)
(F) edge node [right]{$f_2$}(B)
(E) edge node[left] {$g$}(G)
(D') edge node [above]{$i_2$}(E')
(F') edge node [above]{$o_2$}(E')
(D') edge node [left]{$f_2$}(C)
(C) edge node [above]{$i'_2$} (G')
(H) edge node [above]{$o'_2$} (G')
(F') edge node[right] {$f_3$}(H)
(E') edge node[left] {$h$}(G');
\end{tikzpicture}
\]
is given by
\[
\begin{tikzpicture}[scale=1.5]
\node (E) at (3,0) {$x_1$};
\node (G) at (4.5,0) {$y+_{x_2} z$};
\node (F) at (6,0) {$x_3$};
\node (E') at (3,-1) {$x_1'$};
\node (G') at (4.5,-1) {$y'+_{x_2'}z'$};
\node (F') at (6,-1) {$x_3'$.};
\path[->,font=\scriptsize,>=angle 90]
(E) edge node[left]{$f_1$} (E')
(F) edge node[right]{$f_3$} (F')
(G) edge node[left]{$g+_{f_2} h$} (G')
(E) edge node[above]{$j_y i_1$} (G)
(F) edge node[above]{$j_z o_2$} (G)
(E') edge node[below]{$j_{y'} i_1'$} (G')
(F') edge node[below]{$j_{z'} o_2'$} (G');
\end{tikzpicture}
\]
\item the vertical composite of two 2-morphisms:
\[
\begin{tikzpicture}[scale=1.5]
\node (E) at (3,0) {$x_1$};
\node (G) at (4,0) {$y$};
\node (F) at (5,0) {$x_2$};
\node (E') at (3,-1) {$x_1'$};
\node (G') at (4,-1) {$y'$};
\node (F') at (5,-1) {$x_2'$};
\path[->,font=\scriptsize,>=angle 90]
(E) edge node[left]{$f_1$} (E')
(F) edge node[right]{$f_2$} (F')
(G) edge node[left]{$g$} (G')
(E) edge node[above]{$i$} (G)
(F) edge node[above]{$o$} (G)
(E') edge node[below]{$i'$} (G')
(F') edge node[below]{$o'$} (G');
\end{tikzpicture}
\]
\[
\begin{tikzpicture}[scale=1.5]
\node (E) at (3,0) {$x_1'$};
\node (G) at (4,0) {$y'$};
\node (F) at (5,0) {$x_2'$};
\node (E') at (3,-1) {$x_1''$};
\node (G') at (4,-1) {$y''$};
\node (F') at (5,-1) {$x_2''$};
\path[->,font=\scriptsize,>=angle 90]
(E) edge node[left]{$f_1'$} (E')
(F) edge node[right]{$f_2'$} (F')
(G) edge node[left]{$g'$} (G')
(E) edge node[above]{$i'$} (G)
(F) edge node[above]{$o'$} (G)
(E') edge node[below]{$i''$} (G')
(F') edge node[below]{$o''$} (G');
\end{tikzpicture}
\]
is given by
\[
\begin{tikzpicture}[scale=1.5]
\node (E) at (3,0) {$x_1$};
\node (G) at (4,0) {$y$};
\node (F) at (5,0) {$x_2$};
\node (E') at (3,-1) {$x_1''$};
\node (G') at (4,-1) {$y''$};
\node (F') at (5,-1) {$x_2''$};
\path[->,font=\scriptsize,>=angle 90]
(E) edge node[left]{$f_1' f_1$} (E')
(F) edge node[right]{$f_2' f_2$} (F')
(G) edge node[left]{$g' g$} (G')
(E) edge node[above]{$i$} (G)
(F) edge node[above]{$o$} (G)
(E') edge node[below]{$i''$} (G')
(F') edge node[below]{$o''$} (G');
\end{tikzpicture}
\]
\item the associator and unitors are defined using the universal property of pushouts.
\end{itemize}
\end{lem}

\begin{proof}
This is well known \cite{Cour,Nie}. 
\end{proof}

We expect that a different choice of pushouts in $\X$ will give an equivalent double category
$\lCsp(\X)$, since pushouts are unique up to canonical isomorphism.

To build structured cospan double categories, we use a method we learned from Christina Vasilakopoulou for taking a double category $\lX$ and replacing its objects and vertical 1-morphisms with the objects and morphisms of some category $\A$.   In Appendix \ref{appendix}, we recall that any double category $\lX$ has a category $\lX_0$ called its \define{category of objects}, whose objects are those of $\lX$ and whose morphisms are the vertical 1-morphisms of $\lX$.   We can replace the category of objects by $\A$ using a functor $L \maps \A \to \lX_0$. 

\begin{lem} \label{trick1}
Given a double category $\lX$, a category $\A$ and a functor $L \maps \A \to \lX_0$, there is a double category $_{L} \lX$ in which:
\begin{itemize}
\item an object is an object of $\A$,
\item a vertical 1-morphism is a morphism of $\A$, 
\item a horizontal 1-cell from $a$ to $b$ is a horizontal 1-cell $L(a) \xrightarrow{M} L(b)$ of $\lX$, 
\item a 2-morphism is a 2-morphism in $\lX$ of the form:
\[
\begin{tikzpicture}[scale=1.5]
\node (A) at (0,0) {$L(a)$};
\node (C) at (1,0) {$L(b)$};
\node (A') at (0,-1) {$L(a')$};
\node (C') at (1,-1) {$L(b')$,};
\node (B) at (0.5,-0.5) {$\Downarrow \alpha$};
\path[->,font=\scriptsize,>=angle 90]
(A) edge node[above]{$M$} (C)
(A) edge node [left]{$L(f)$} (A')
(C)edge node[right]{$L(g)$}(C')
(A')edge node [above] {$N$}(C');
\end{tikzpicture}
\]
\item composition of vertical 1-morphisms is composition in $\A$
\item composition of horizontal 1-morphisms is defined as in $\lX$,
\item vertical and horizontal composition of 2-morphisms are defined as in $\lX$,
\item the associator and unitors are defined as in $\lX$.
\end{itemize}
\end{lem}

\begin{proof}
It is easy to check the double category axioms using the fact that $\lX$ is a double category and $L$ is a functor.   \end{proof}

Putting the above lemmas together, we obtain our double category of structured cospans.  We describe it quite explicitly for reference purposes:

\begin{thm} 
\label{thm:LCsp(X)}
Let $L \maps \A \to \X$ be a functor where $\X$ is a category with chosen pushouts. Then there is a double category $_L \lCsp(\X)$ in which:
\begin{itemize}
\item an object is an object of $\A$,
\item a vertical 1-morphism is a morphism of $\A$,
\item a horizontal 1-cell from $a$ to $b$ is a diagram in $\X$ of this form:
\[
\begin{tikzpicture}[scale=1.5]
\node (A) at (0,0) {$L(a)$};
\node (B) at (1,0) {$x$};
\node (C) at (2,0) {$L(b)$};
\path[->,font=\scriptsize,>=angle 90]
(A) edge node[above]{$i$} (B)
(C)edge node[above]{$o$}(B);
\end{tikzpicture}
\]
\item a 2-morphism is a commutative diagram in $\X$ of this form:
\[
\begin{tikzpicture}[scale=1.5]
\node (E) at (3,0) {$L(a)$};
\node (F) at (5,0) {$L(b)$};
\node (G) at (4,0) {$x$};
\node (E') at (3,-1) {$L(a')$};
\node (F') at (5,-1) {$L(b')$};
\node (G') at (4,-1) {$x'$};
\path[->,font=\scriptsize,>=angle 90]
(F) edge node[above]{$o$} (G)
(E) edge node[left]{$L(\alpha)$} (E')
(F) edge node[right]{$L(\beta)$} (F')
(G) edge node[left]{$f$} (G')
(E) edge node[above]{$i$} (G)
(E') edge node[below]{$i'$} (G')
(F') edge node[below]{$o'$} (G');
\end{tikzpicture}
\]
\item composition of horizontal 1-cells is done using the chosen pushouts in $\X$: 
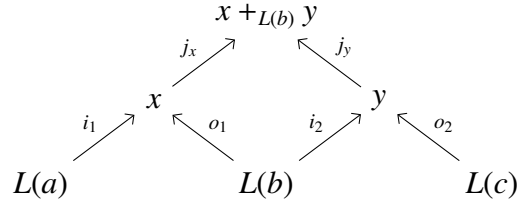
\[
\begin{tikzpicture}[scale=1.5]
\node (A) at (0,0) {$L(a)$};
\node (B) at (1,0.75) {$x$};
\node (C) at (2,0) {$L(b)$};
\node (D) at (3,0.75) {$y$};
\node (E) at (4,0) {$L(c)$};
\node (F) at (2,1.5) {$x +_{L(b)} y$};
\path[->,font=\scriptsize,>=angle 90]
(A) edge node [above left=-1mm]{$i_1$} (B)
(C) edge node [above right=-1mm]{$o_1$} (B)
(C) edge node [above left=-1mm] {$i_2$} (D)
(E) edge node [above right=-1mm] {$o_2$} (D)
(B) edge node [above left=-1mm] {$j_x$} (F)
(D) edge node [above right=-1mm] {$j_y$} (F);
\end{tikzpicture}
\]
where $j_x$ and $j_y$ are the canonical morphisms from $x$ and $y$ to the pushout object, 
\item identity horizontal 1-cells are diagrams of this form:
\[
\begin{tikzpicture}[scale=1.5]
\node (A) at (0,0) {$L(a)$};
\node (B) at (1,0) {$L(a)$};
\node (C) at (2,0) {$L(a)$};
\path[->,font=\scriptsize,>=angle 90]
(A) edge node[above]{$1$} (B)
(C)edge node[above]{$1$}(B);
\end{tikzpicture}
\]
\item 
the horizontal composite of two 2-morphisms:
\[
\begin{tikzpicture}[scale=1.5]
\node (D) at (-3,1) {$L(a)$};
\node (E) at (-2,1) {$x$};
\node (F) at (-1,1) {$L(b)$};
\node (A) at (-3,0) {$L(a')$};
\node (G) at (-2,0) {$x'$};
\node (B) at (-1,0) {$L(b')$};

\node (D') at (0.5,1) {$L(b)$};
\node (E') at (1.5,1) {$y$};
\node (F') at (2.5,1) {$L(c)$};
\node (C) at (0.5,0) {$L(b')$};
\node (G') at (1.5,0) {$y'$};
\node (H) at (2.5,0) {$L(c')$};
\path[->,font=\scriptsize,>=angle 90]
(D) edge node[above] {$i_1$}(E)
(A) edge node[below] {$i'_1$} (G)
(B) edge node [below]{$o'_1$} (G)
(F) edge node [above]{$o_1$}(E)
(D) edge node [left]{$L(\alpha)$}(A)
(F) edge node [right]{$L(\beta)$}(B)
(E) edge node[left] {$f$}(G)
(D') edge node [above]{$i_2$}(E')
(F') edge node [above]{$o_2$}(E')
(D') edge node [left]{$L(\beta)$}(C)
(C) edge node [below]{$i'_2$} (G')
(H) edge node [below]{$o'_2$} (G')
(F') edge node[right] {$L(\gamma)$}(H)
(E') edge node[left] {$g$}(G');
\end{tikzpicture}
\]
is given by
\[
\begin{tikzpicture}[scale=1.5]
\node (E) at (3,0) {$L(a)$};
\node (G) at (4.5,0) {$x+_{L(b)} y$};
\node (F) at (6,0) {$L(c)$};
\node (E') at (3,-1) {$L(a')$};
\node (G') at (4.5,-1) {$x'+_{L(b')} y'$};
\node (F') at (6,-1) {$L(c')$};
\path[->,font=\scriptsize,>=angle 90]
(E) edge node[left]{$L(\alpha)$} (E')
(F) edge node[right]{$L(\gamma)$} (F')
(G) edge node[left]{$f +_{L(\beta)} g$} (G')
(E) edge node[above]{$j_x i_1$} (G)
(F) edge node[above]{$j_y o_2$} (G)
(E') edge node[below]{$j_{x'} i_1'$} (G')
(F') edge node[below]{$j_{y'} o_2'$} (G');
\end{tikzpicture}
\]
\item the identities for horizontal composition of 2-morphisms are diagrams of this form:
\[
\begin{tikzpicture}[scale=1.5]
\node (E) at (3,0) {$L(a)$};
\node (G) at (4,0) {$L(a)$};
\node (F) at (5,0) {$L(a)$};
\node (E') at (3,-1) {$L(a')$};
\node (G') at (4,-1) {$L(a')$};
\node (F') at (5,-1) {$L(a')$};
\path[->,font=\scriptsize,>=angle 90]
(E) edge node[left]{$L(\alpha)$} (E')
(F) edge node[right]{$L(\alpha)$} (F')
(G) edge node[left]{$L(\alpha)$} (G')
(E) edge node[above]{$1$} (G)
(F) edge node[above]{$1$} (G)
(E') edge node[below]{$1$} (G')
(F') edge node[below]{$1$} (G');
\end{tikzpicture}
\]
\item the vertical composite of two 2-morphisms:
\[
\begin{tikzpicture}[scale=1.5]
\node (E) at (3,0) {$L(a)$};
\node (G) at (4,0) {$y$};
\node (F) at (5,0) {$L(b)$};
\node (E') at (3,-1) {$L(a')$};
\node (G') at (4,-1) {$y'$};
\node (F') at (5,-1) {$L(b')$};
\path[->,font=\scriptsize,>=angle 90]
(E) edge node[left]{$L(\alpha)$} (E')
(F) edge node[right]{$L(\beta)$} (F')
(G) edge node[left]{$f$} (G')
(E) edge node[above]{$i$} (G)
(F) edge node[above]{$o$} (G)
(E') edge node[below]{$i'$} (G')
(F') edge node[below]{$o'$} (G');
\end{tikzpicture}
\]
\[
\begin{tikzpicture}[scale=1.5]
\node (E) at (3,0) {$L(a')$};
\node (G) at (4,0) {$y'$};
\node (F) at (5,0) {$L(b')$};
\node (E') at (3,-1) {$L(a'')$};
\node (G') at (4,-1) {$y''$};
\node (F') at (5,-1) {$L(b'')$};
\path[->,font=\scriptsize,>=angle 90]
(E) edge node[left]{$L(\alpha')$} (E')
(F) edge node[right]{$L(\beta')$} (F')
(G) edge node[left]{$f'$} (G')
(E) edge node[above]{$i'$} (G)
(F) edge node[above]{$o'$} (G)
(E') edge node[below]{$i''$} (G')
(F') edge node[below]{$o''$} (G');
\end{tikzpicture}
\]
is given by
\[
\begin{tikzpicture}[scale=1.5]
\node (E) at (3,0) {$L(a)$};
\node (G) at (4,0) {$y$};
\node (F) at (5,0) {$L(b)$};
\node (E') at (3,-1) {$L(a'')$};
\node (G') at (4,-1) {$y''$};
\node (F') at (5,-1) {$L(b'')$};
\path[->,font=\scriptsize,>=angle 90]
(E) edge node[left]{$L(\alpha'\alpha)$} (E')
(F) edge node[right]{$L(\beta'\beta)$} (F')
(G) edge node[left]{$f' f$} (G')
(E) edge node[above]{$i$} (G)
(F) edge node[above]{$o$} (G)
(E') edge node[below]{$i''$} (G')
(F') edge node[below]{$o''$} (G');
\end{tikzpicture}
\]
\item the associator and unitors are defined using the universal property of pushouts.
\end{itemize}
\end{thm}

\begin{proof}
We apply Lemma \ref{trick1} to the double category $\lCsp(\X)$ of Lemma \ref{lem:dubcsp}.
\end{proof}

 From the double category ${}_L \lCsp(\X)$ we can extract a bicategory ${}_L \bCsp(\X)$ and then a category ${}_L \Csp(\X)$.  In many applications all we need is a bicategory or even a mere category of structured cospans, so the reader should not get the misimpression that working with structured cospans \emph{requires} using double categories.  We begin with the bicategory:

\begin{cor} \label{cor:LCsp(X) bicategory}
Let $L \maps \A \to \X$ be a functor where $\X$ is a category with chosen pushouts. Then there is a bicategory $_L \bCsp(\X)$ in which:
\begin{itemize}
\item an object is an object of $\A$,
\item a morphism from $a$ to $b$ is a diagram in $\X$ of this form:
\[
\begin{tikzpicture}[scale=1.5]
\node (A) at (0,0) {$L(a)$};
\node (B) at (1,0) {$x$};
\node (C) at (2,0) {$L(b)$};
\path[->,font=\scriptsize,>=angle 90]
(A) edge node[above]{$i$} (B)
(C)edge node[above]{$o$}(B);
\end{tikzpicture}
\]
\item a 2-morphism is a commutative diagram in $\X$ of this form:
\[
\begin{tikzpicture}[scale=1.5]
\node (E) at (3,-0.5) {$L(a)$};
\node (F) at (5,-0.5) {$L(b)$};
\node (G) at (4,0) {$x$};
\node (G') at (4,-1) {$x'$};
\path[->,font=\scriptsize,>=angle 90]
(F) edge node[above]{$o$} (G)
(G) edge node[left]{$f$} (G')
(E) edge node[above]{$i$} (G)
(E) edge node[below]{$i'$} (G')
(F) edge node[below]{$o'$} (G');
\end{tikzpicture}
\]
\item composition of morphisms is done using the chosen pushouts in $\X$, 
\item identity morphisms are of this form:
\[
\begin{tikzpicture}[scale=1.5]
\node (A) at (0,0) {$L(a)$};
\node (B) at (1,0) {$L(a)$};
\node (C) at (2,0) {$L(a)$};
\path[->,font=\scriptsize,>=angle 90]
(A) edge node[above]{$1$} (B)
(C)edge node[above]{$1$}(B);
\end{tikzpicture}
\]
\item the horizontal composite of 2-morphisms:
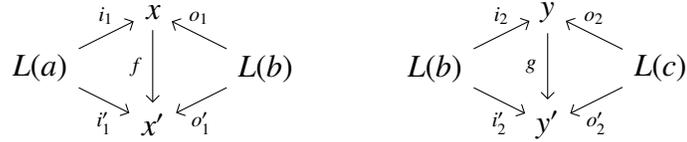
\[
\begin{tikzpicture}[scale=1.5]
\node (D) at (-3,0.5) {$L(a)$};
\node (E) at (-2,1) {$x$};
\node (F) at (-1,0.5) {$L(b)$};
\node (G) at (-2,0) {$x'$};

\node (D') at (0.5,0.5) {$L(b)$};
\node (E') at (1.5,1) {$y$};
\node (F') at (2.5,0.5) {$L(c)$};
\node (G') at (1.5,0) {$y'$};
\path[->,font=\scriptsize,>=angle 90]
(D) edge node[above] {$i_1$}(E)
(D) edge node[below] {$i'_1$} (G)
(F) edge node [below]{$o'_1$} (G)
(F) edge node [above]{$o_1$}(E)
(E) edge node[left] {$f$}(G)
(D') edge node [above]{$i_2$}(E')
(F') edge node [above]{$o_2$}(E')
(D') edge node [below]{$i'_2$} (G')
(F') edge node [below]{$o'_2$} (G')
(E') edge node[left] {$g$}(G');
\end{tikzpicture}
\]
is given by
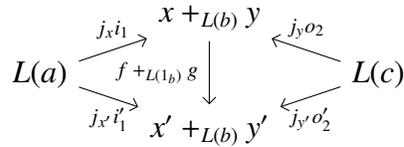
\[
\begin{tikzpicture}[scale=1.5]
\node (E) at (3,-0.5) {$L(a)$};
\node (G) at (4.5,0) {$x+_{L(b)} y$};
\node (F) at (6,-0.5) {$L(c)$};
\node (G') at (4.5,-1) {$x'+_{L(b)} y'$};
\path[->,font=\scriptsize,>=angle 90]
(G) edge node[left]{$f +_{L(1_b)} g$} (G')
(E) edge node[above]{$j_x i_1$} (G)
(F) edge node[above]{$j_y o_2$} (G)
(E) edge node[below]{$j_{x'} i_1'$} (G')
(F) edge node[below]{$j_{y'} o_2'$} (G');
\end{tikzpicture}
\]
where $j_x$ and $j_y$ are the canonical morphisms from $x$ and $y$ to the pushout object \hfill \break $x +_{L(b)} y$, and similarly for $j_{x'}$ and $j_{y'}$,
\item the vertical composite of 2-morphisms:
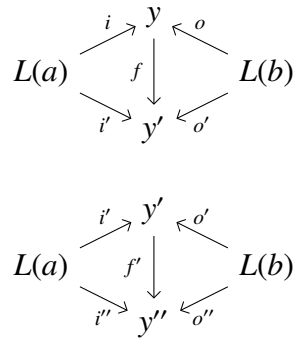
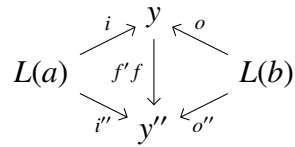
\[
\begin{tikzpicture}[scale=1.5]
\node (E) at (3,-0.5) {$L(a)$};
\node (G) at (4,0) {$y$};
\node (F) at (5,-0.5) {$L(b)$};
\node (G') at (4,-1) {$y'$};
\path[->,font=\scriptsize,>=angle 90]
(G) edge node[left]{$f$} (G')
(E) edge node[above]{$i$} (G)
(F) edge node[above]{$o$} (G)
(E) edge node[below]{$i'$} (G')
(F) edge node[below]{$o'$} (G');
\end{tikzpicture}
\]
\[
\begin{tikzpicture}[scale=1.5]
\node (E) at (3,-0.5) {$L(a)$};
\node (G) at (4,0) {$y'$};
\node (F) at (5,-0.5) {$L(b)$};
\node (G') at (4,-1) {$y''$};
\path[->,font=\scriptsize,>=angle 90]
(G) edge node[left]{$f'$} (G')
(E) edge node[above]{$i'$} (G)
(F) edge node[above]{$o'$} (G)
(E) edge node[below]{$i''$} (G')
(F) edge node[below]{$o''$} (G');
\end{tikzpicture}
\]
is given by
\[
\begin{tikzpicture}[scale=1.5]
\node (E) at (3,-0.5) {$L(a)$};
\node (G) at (4,0) {$y$};
\node (F) at (5,-0.5) {$L(b)$};
\node (G') at (4,-1) {$y''$};
\path[->,font=\scriptsize,>=angle 90]
(G) edge node[left]{$f' f$} (G')
(E) edge node[above]{$i$} (G)
(F) edge node[above]{$o$} (G)
(E) edge node[below]{$i''$} (G')
(F) edge node[below]{$o''$} (G');
\end{tikzpicture}
\]
\item the associator and unitors are defined using the universal property of pushouts.
\end{itemize}
\end{cor}
 
\begin{proof} 
As noted for example by Shulman \cite{Shul2010}, any double category $\lX$ gives rise to a bicategory $\bX$ with
\begin{itemize}
\item objects given by objects of $\lX$,
\item morphisms given by horizontal 1-cells of $\lX$,
\item 2-morphisms given by \define{globular} 2-morphisms of $\lX$, meaning 2-morphisms whose source and target vertical 1-morphisms are identities,
\item composition of morphisms given by horizontal composition of horizontal 1-cells in $\lX$,
\item vertical and horizontal composition of 2-morphisms given by vertical and horizontal
composition of 2-morphisms in $\lX$.
\end{itemize}
Applying this to ${}_L \lCsp(\X)$ we obtain ${}_L \bCsp(\X)$. 
\end{proof}

\begin{cor} \label{_L Csp(X) category}
Let $L \maps \A \to \X$ be a functor where $\X$ is a category with pushouts. Then there is a category ${}_L \Csp(\X)$ in which:
\begin{itemize}
\item an object is an object of $\A$,
\item a morphism from $a$ to $b$ is an isomorphism class of diagrams in $\X$ of this form:
\[
\begin{tikzpicture}[scale=1.5]
\node (A) at (0,0) {$L(a)$};
\node (B) at (1,0) {$x$};
\node (C) at (2,0) {$L(b)$};
\path[->,font=\scriptsize,>=angle 90]
(A) edge node[above]{$i$} (B)
(C)edge node[above]{$o$}(B);
\end{tikzpicture}
\]
where
\raisebox{-0.5em}{
$\begin{tikzpicture}[scale=1]
\node (A) at (0,0) {$L(a)$};
\node (B) at (1,0) {$x$};
\node (C) at (2,0) {$L(b)$};
\path[->,font=\scriptsize,>=angle 90]
(A) edge node[above]{$i$} (B)
(C)edge node[above]{$o$}(B);
\end{tikzpicture}$
}
and 
\raisebox{-0.5em}{
$\begin{tikzpicture}[scale=1]
\node (A) at (0,0) {$L(a)$};
\node (B) at (1,0) {$x'$};
\node (C) at (2,0) {$L(b)$};
\path[->,font=\scriptsize,>=angle 90]
(A) edge node[above]{$i'$} (B)
(C)edge node[above]{$o'$}(B);
\end{tikzpicture}$
}
are isomorphic iff there is an isomorphism $f \maps x \to x'$ making this diagram commute:
\[
\begin{tikzpicture}[scale=1.5]
\node (E) at (3,-0.5) {$L(a)$};
\node (F) at (5,-0.5) {$L(b)$};
\node (G) at (4,0) {$x$};
\node (G') at (4,-1) {$x'$};
\path[->,font=\scriptsize,>=angle 90]
(F) edge node[above]{$o$} (G)
(G) edge node[left]{$f$} (G')
(E) edge node[above]{$i$} (G)
(E) edge node[below]{$i'$} (G')
(F) edge node[below]{$o'$} (G');
\end{tikzpicture}
\]
\item composition of morphisms is done using pushouts in $\X$.
\end{itemize}
\end{cor}

\begin{proof} 
By decategorifying a bicategory $\bB$ we obtain a category $\mathsf{B}$ with the same objects, whose morphisms are isomorphism classes of 1-morphisms in $\bB$.  Applying this to ${}_L \bCsp(\X)$ we obtain ${}_L \Csp(\X)$.   Note that this category is independent of our
choice of pushouts in $\X$, since pushouts are unique up to isomorphism.
\end{proof}

\section{Symmetric monoidal double categories of structured cospans}

Now we give simple conditions under which the double category ${}_L \lCsp(\X)$, 
the bicategory ${}_L \bCsp(\X)$ and the category ${}_L \Csp(\X)$ all become symmetric monoidal.
We have seen that if $\X$ has pushouts and $L \maps \A \to \X$ is any functor then there is a double category of structured cospans ${}_L \lCsp(\X)$.   In Theorem \ref{thm:LCsp(X) symmetric} we show that ${}_L \lCsp(\X)$ becomes symmetric monoidal when $\A$ and $\X$ have finite colimits and $L$ preserves these.  The monoidal structure describes our ability to take two structured cospans:
\[
\begin{tikzpicture}[scale=1.2]
\node (A) at (0,0) {$L(a)$};
\node (B) at (1,1) {$x$};
\node (C) at (2,0) {$L(b)$};
\node (E) at (4,0) {$L(a')$};
\node (F) at (5,1) {$x'$};
\node (G) at (6,0) {$L(b')$};
\path[->,font=\scriptsize,>=angle 90]
(A) edge node[above,left]{$i$} (B)
(C)edge node[above,right]{$o$}(B)
(E) edge node[above,left]{$i'$} (F)
(G)edge node[above,right]{$o'$}(F);
\end{tikzpicture}
\]
and form a new one via coproduct:
\[
\begin{tikzpicture}[scale=1.2]
\node (I) at (1,0) {$L(a) + L(a')$};
\node (J) at (2,1) {$x + x'$};
\node (K) at (3,0) {$L(b) + L(b')$};
\node (L) at (0,-1) {$L(a + a')$};
\node (M) at (4,-1) {$L(b + b')$};
\path[->,font=\scriptsize,>=angle 90]
(I) edge node [above,left] {$i + i'$} (J)
(K) edge node [above,right] {$o + o'$} (J)
(L) edge node [above,left] {$\cong$} (I)
(M) edge node [above,right] {$\cong$} (K); 
\end{tikzpicture}
\]
One can check that this operation makes ${}_L \lCsp(\X)$ into a monoidal double category simply by verifying that a rather large number of diagrams commute.   This is the approach taken in \cite{CourThesis}.  There is nothing tricky about it.  Indeed, requiring that $L$ preserve finite colimits is overkill: it suffices for $L$ to preserve finite coproducts.  Thus, for most readers the best thing to do at this point would be to review the definition of `symmetric monoidal double category' in Appendix \ref{appendix}, look at the statement of Theorem \ref{thm:LCsp(X) symmetric}, and move on to the next section.

However, it is a bit irksome to check that all the necessary diagrams commute, especially
since one gets the feeling that there must be a simple underlying reason.  So, we decided to
give a more conceptual proof.  While perhaps harder to digest, this gives us more---at least 
when $F$ preserves finite colimits.  In this case we can do much more than take binary coproducts of structured cospans: we can take finite colimits of them!   This means that we can 
glue together structured cospans in more interesting ways than merely composing them 
end to end or setting them side by side.  Thus, we prove Theorem \ref{thm:LCsp(X) symmetric} 
as a consequence of a stronger result, Theorem \ref{thm:rex}, which captures the full range of ways we can take finite colimits of structured cospans.

The key concept we need is that of a `weak category' or `pseudocategory' \cite{NMF} in a 2-category. This is a slight generalization of the concept of double category.

\begin{defn}
Given a 2-category $\bC$, a \define{weak category} $\lD$ in $\bC$ consists of:
\begin{itemize}
\item an \define{object of objects} $\lD_0 \in \bC$ and an \define{object of arrows} $\lD_1 \in \bC$,
\item  \define{source} and \define{target} morphisms
\[  S,T \colon \lD_1 \to \lD_0 ,\]
\item 
an \define{identity-assigning} morphism
\[  U\colon \lD_0 \to \lD_1 ,\]
\item 
and a \define{composition} morphism
\[ \odot \colon \lD_1 \times_{\lD_0} \lD_1 \to \lD_1 \]
where the pullback is taken over $\lD_1 \xrightarrow[]{T} \lD_0 \xleftarrow[]{S} \lD_1$,
\end{itemize}
such that:
\begin{itemize}
\item the source and target morphisms behave as expected for identities:
\[ S \circ U = 1_{\lD_0} = T \circ U \]
and for composition:
\[ 	S \circ \odot = S \circ p_1, \qquad T \circ \odot = T \circ p_2 \]
where $p_1, p_2 \maps  \lD_1 \times_{\lD_0} \lD_1 \to \lD_1$ are projections to the 
two factors;
\item composition is associative up to a 2-isomorphism called the \define{associator}:
\[
\begin{tikzpicture}[scale=1.5]
\node (A) at (0,0) {$\lD_1 \times_{\lD_0} \lD_1 \times_{\lD_0} \lD_1$};
\node (C) at (2.5,0) {$\lD_1 \times_{\lD_0} \lD_1$};
\node (A') at (0,-1) {$\lD_1 \times_{\lD_0} \lD_1$};
\node (C') at (2.5,-1) {$\lD_1$};
\node (B) at (1.25,-0.5) {$\alpha \NEarrow$};

\path[->,font=\scriptsize,>=angle 90]
(A) edge node[above]{$1 \times \odot$} (C)
(A) edge node [left]{$\odot \times 1$} (A')
(C)edge node[right]{$\odot$}(C')
(A')edge node [below] {$\odot$}(C');

\end{tikzpicture}
\]
\item composition obeys the left and right unit laws up to 2-isomorphisms called the 
\define{left} and \define{right unitors}:
\[
\begin{tikzpicture}[scale=1.5]

\node (A'') at (4,0) {$\lD_0 \times_{\lD_0} \lD_1$};
\node (C'') at (6,0) {$\lD_1 \times_{\lD_0} \lD_1$};
\node (A''') at (8,0) {$\lD_1 \times_{\lD_0} \lD_0$};
\node (C''') at (6,-1) {$\lD_1$};
\node (B'') at (5.5,-0.35) {$\lambda \SWarrow$};
\node (B''') at (6.5,-0.35) {$\rho \SEarrow$};
\path[->,font=\scriptsize,>=angle 90]

(A'') edge node[above]{$U \times_{\lD_0} 1$} (C'')
(A''') edge node [above]{$1 \times_{\lD_0} U$} (C'')
(C'')edge node[left]{$\odot$}(C''')
(A'')edge node [below] {$p_2$}(C''')
(A''')edge node [below] {$p_1$}(C''');
\end{tikzpicture}
\]
\item $\alpha, \lambda$ and $\rho$ obey the pentagon identity and triangle identity.
\end{itemize}
\end{defn}
\noindent
In this definition we assume that the necessary pullbacks exist; if $\bC$ has pullbacks this is automatic.  

Consulting Appendix \ref{appendix}, the reader can check that a weak category in $\Cat$ is the same as a double category.  We need weak categories in the following 2-categories as well:
\begin{defn}
Let $\Rex$ be the 2-category with:
\begin{itemize} 
\item categories with chosen finite colimits as objects,
\item right exact functors as morphisms,
\item natural transformations as 2-morphisms.
\end{itemize}
\end{defn}
\begin{defn}
Let $\SMC$ be the 2-category with:
\begin{itemize}
\item symmetric monoidal categories as objects,
\item (strong) symmetric monoidal functors as morphisms,
\item monoidal natural transformations as 2-morphisms.
\end{itemize}
\end{defn}
\noindent
The word `rex' is an abbreviation of `right exact', which is another term for `preserving finite colimits'.    Note that a right exact functor need not preserve a given \emph{choice} of finite
colimits.   Thus, our 2-category $\Rex$ is 2-equivalent to one where no choices of finite colimits were made.   One reason for making these choices is that they give us an unambiguously 
defined 2-functor
\[  \Phi \maps \Rex \to \SMC \]  
as follows.  Given an object $\C \in \Rex$, $\Phi(\C)$ is the symmetric monoidal category $(\C, +, 0)$ where $+$ is the chosen binary coproduct and $0$ is the chosen initial object.  Each right exact functor $F \maps \C \to \C'$ between categories $\C,\C' \in \Rex$ then becomes symmetric monoidal in a canonical way, and each natural transformation between right exact functors becomes monoidal.   

Our plan now proceeds as follows.  First, in Theorem \ref{thm:rex}, we show that when $L \maps \A \to \X$ is a morphism in $\Rex$, the double category ${}_L \lCsp(\X)$  is not merely a weak category in $\Cat$, but actually a weak category in $\Rex$.   In Theorem \ref{thm:SMC} we use the 2-functor $\Phi$ to convert ${}_L \lCsp(\X)$ into a weak category in $\SMC$.  

Finally, from this weak category in $\SMC$, we wish to get a symmetric monoidal double category.
Here we need the concept of a `symmetric pseudomonoid' \cite{Verdon}. To understand the following definitions the reader should keep in mind the example where $\B$ is $\Cat$ made into a symmetric monoidal bicategory using cartesian products.   Then a pseudomonoid in $\B$ is a monoidal category, a braided pseudomonoid is a braided monoidal category, and a symmetric pseudomonoid is a symmetric monoidal category.  

\begin{defn} A \define{pseudomonoid} in a monoidal bicategory $\B$ is an object $M \in \B$ equipped with 1-morphisms called the \define{multiplication} $m \maps M \otimes M \to M$ and \define{unit} $i \maps I \to M$ that obey associativity and the left and right unit laws up to 2-isomorphisms called the \define{associator} and left and right \define{unitors}, that in turn obey the pentagon identity and triangle identity.
\end{defn}  

\begin{defn} A pseudomonoid $M$ in a braided monoidal bicategory $\B$ is 
\define{braided} if it is equipped with a 2-isomorphism 
\[       b \maps m \circ \beta \stackrel{\sim}{\To} m \]
where $\beta \maps M \otimes M \to M \otimes M$ is the braiding in $\B$, 
and $b$ obeys the hexagon identities \cite{MacLane}.
\end{defn}

\begin{defn} A braided pseudomonoid $M$ in a symmetric monoidal bicategory $\B$ is called
\define{symmetric} if
\[
\begin{tikzpicture}[scale=1.5]
\node (A) at (0,0) {$M \otimes M$};
\node (D) at (6,0) {$M$};
\node (B) at (2,0) {$M \otimes M$};
\node (C) at (4,0) {$M \otimes M$};
\node at (3,1) {\scriptsize $\lambda^{-1} \Downarrow$};
\node at (1.8,0.4) {\scriptsize $\sigma^{-1} \Downarrow$};
\node at (4,-0.4) {\scriptsize $b \Downarrow$};
\node at (3,-1) {\scriptsize $b \Downarrow$};
\path[->,font=\scriptsize,>=angle 90]
(A) edge[bend left=50, looseness =1] node[above]{$m$} (D)
(A) edge [bend left] node[above]{$1$} (C)
(A) edge  node[above]{$\beta$} (B)
(B) edge  node[above]{$\beta$} (C)
(C) edge  node[above]{$m$} (D)
(B) edge[bend right] node[below]{$m$} (D)
(A) edge[bend right=50, looseness=1] node[below]{$m$} (D);
\end{tikzpicture}
\]
is the identity 2-morphism from $m$ to $m$.  Here $\lambda$ is the left unitor for composition of 1-morphisms in $\B$ and $\sigma \maps \beta^2 \To 1$ is the syllepsis for $\B$.
\end{defn}
\noindent Readers unfamiliar with these concepts may be relieved to learn that the syllepsis in $\Cat$ is the identity; in a general symmetric monoidal bicategory the square of the braiding may be only \emph{isomorphic} to the identity, and this isomorphism is called the syllepsis \cite{DayStreet}.

The plan continues as follows.   Having shown that ${}_L \lCsp(\X)$ is a weak category in $\SMC$, 
we notice that such a thing is 
\begin{center}
a weak category in [symmetric pseudomonoids in $\Cat$].  
\end{center}
By `commutativity of internalization' we could hope that this is the same as
\begin{center}
a symmetric pseudomonoid in [weak categories in $\Cat$].
\end{center}
But the latter is precisely a symmetric double category.  So, ${}_L \lCsp(\X)$ should be a symmetric
monoidal double category.

Unfortunately, this hope is a bit naive.  Shulman explains the reason \cite{Shul2010}:

\begin{quote}
The general yoga of internalization says that an $X$ internal to $Y$s internal to $Z$s is equivalent to a $Y$ internal to $X$s internal to $Z$s, but this is only strictly true when the internalizations are all strict. We have defined a symmetric monoidal double category to be a (pseudo) symmetric monoid internal to (pseudo) categories internal to categories, but one could also consider a (pseudo) category internal to (pseudo) symmetric monoids internal to categories, i.e.\ a pseudo internal category in the 2-category $\SMC$ of symmetric monoidal categories and strong symmetric monoidal functors. This would give \textit{almost} the same definition, except that $S$ and $T$ would only be strong monoidal (preserving $\otimes$ up to isomorphism) rather than strict monoidal.
 \end{quote}
Luckily, the difference between the two definitions is quite small, so with a bit of care we
can arrange for ${}_L \lCsp(\X)$ to be a symmetric monoidal double category. 

We begin as follows:

\begin{thm} \label{thm:rex}
Given a morphism $L \maps \A \to \X$ in $\Rex$, the double category ${}_L \lCsp(\X)$ is a weak category object in $\Rex$.
\end{thm}
\begin{proof}
In the double category ${}_L \lCsp(\X)$,
\begin{itemize}
\item the category of objects ${}_L \lCsp(\X)_0$ is $\A$, while
\item the category of arrows ${}_L \lCsp(\X)_1$ has structured cospans
\[
\begin{tikzpicture}[scale=1.5]
\node (A) at (0,0) {$L(a)$};
\node (B) at (1,0) {$x$};
\node (C) at (2,0) {$L(b)$};
\path[->,font=\scriptsize,>=angle 90]
(A) edge node[above]{$i$} (B)
(C)edge node[above]{$o$}(B);
\end{tikzpicture}
\]
as objects and commutative diagrams of this form:
\[
\begin{tikzpicture}[scale=1.5]
\node (E) at (3,0) {$L(a)$};
\node (F) at (5,0) {$L(b)$};
\node (G) at (4,0) {$x$};
\node (E') at (3,-1) {$L(a')$};
\node (F') at (5,-1) {$L(b')$};
\node (G') at (4,-1) {$x'$};
\path[->,font=\scriptsize,>=angle 90]
(F) edge node[above]{$o$} (G)
(E) edge node[left]{$L(\alpha)$} (E')
(F) edge node[right]{$L(\beta)$} (F')
(G) edge node[left]{$f$} (G')
(E) edge node[above]{$i$} (G)
(E') edge node[below]{$i'$} (G')
(F') edge node[below]{$o'$} (G');
\end{tikzpicture}
\]
as morphisms.
\end{itemize}
We need to choose finite colimits for ${_L \lCsp(\X)_0}$ and ${_L \lCsp(\X)_1}$ and show the source and target functors
\[   S,T \maps {_L \lCsp(\X)_1} \to {_L \lCsp(\X)_0}, \]
the identity-assigning functor
\[    U \maps {_L \lCsp(\X)_0} \to {_L \lCsp(\X)_1}, \]
and the composition functor
\[   \circ \maps {}_L \lCsp(\X)_1 \times_{ _L\lCsp(\X)_0} {}_L \lCsp(\X)_1 \to {}_L \lCsp(\X)_1 \] 
are right exact.  We also need to check that all the pullbacks in $\Cat$ used to define the double category ${_L \lCsp(\X)}$ are also pullbacks in $\Rex$.

The category of objects ${}_L \lCsp(\X)_0 = \A$ has chosen finite colimits by hypothesis. The category of arrows ${}_L \lCsp(\X)_1$ has finite colimits because $L$ preserves finite colimits and these colimits are computed pointwise in $\X$.    We give ${}_L \lCsp(\X)_1$ \emph{chosen}
finite colimits using the chosen finite colimits in $\A$ and $\X$. The functors $S, T$ and $U$ are right exact, again because colimits in ${}_L \lCsp(\X)_1$ are computed pointwise in $\X$.    The functor $\circ$ sends a composable pair of structured cospans to their composite, which is defined using a pushout.    This functor is right exact as a consequence of colimits commuting with other colimits.

We also need to check that the category 
\[  \Z = {}_L \lCsp(\X)_1 \times_{{_L \lCsp(\X)_0}} {}_L\lCsp(\X)_1, \]
defined as a pullback in $\Cat$, is also a pullback in $\Rex$.   Note that objects of $\Z$ are
composable pairs of structured cospans:
\[   
\begin{tikzpicture}[scale=1]
\node (A) at (0,0) {$L(a)$};
\node (B) at (1,0) {$x$};
\node (C) at (2,0) {$L(b)$};
\node (D) at (3,0) {$y$};
\node (E) at (4,0) {$L(c),$};
\path[->,font=\scriptsize,>=angle 90]
(A) edge node[above]{$$} (B)
(C) edge node[above]{$$} (B)
(C) edge node[above]{$$} (D)
(E) edge node[above]{$$} (D);
\end{tikzpicture}
\]
while morphisms are commuting diagrams of the form
\[
\begin{tikzpicture}[scale=1.5]
\node (E) at (3,0) {$L(a)$};
\node (F) at (5,0) {$L(b)$};
\node (G) at (4,0) {$x$};
\node (H) at (7,0) {$L(c)$};
\node (I) at (6,0) {$y$};
\node (E') at (3,-1) {$L(a')$};
\node (F') at (5,-1) {$L(b')$};
\node (G') at (4,-1) {$x'$};
\node (H') at (7,-1) {$L(c').$};
\node (I') at (6,-1) {$y'$};
\path[->,font=\scriptsize,>=angle 90]
(F) edge node[above]{$$} (G)
(E) edge node[above]{$$} (G)
(E') edge node[below]{$$} (G')
(F') edge node[below]{$$} (G')
(F) edge node[above]{$$} (I)
(H) edge node[above]{$$} (I)
(F') edge node[above]{$$} (I')
(H') edge node[above]{$$} (I')
(G) edge node[left]{$f$} (G')
(I) edge node[left]{$g$} (I')
(E) edge node[left]{$L(\alpha)$} (E')
(F) edge node[left]{$L(\beta)$} (F')
(H) edge node[right]{$L(\gamma)$} (H');
\end{tikzpicture}
\]
Because $\A$ and $\X$ have finite colimits and $L$ preserves them, $\Z$ has finite colimits
computed pointwise.    Consider the pullback square in $\Cat$ defining $\Z$:
\[
\begin{tikzpicture}[scale=1.5]
\node (A) at (0,0) {$\Z$};
\node (C) at (2,0) {${_L\lCsp(\X)_1}$};
\node (A') at (0,-1) {${_L\lCsp(\X)_1}$};
\node (C') at (2,-1) {${_L\lCsp(\X)_0}$};
\path[->,font=\scriptsize,>=angle 90]
(A) edge node[above]{$P_2$} (C)
(A) edge node [left]{$P_1$} (A')
(C)edge node[right]{$T$}(C')
(A')edge node [below] {$S$}(C');
\end{tikzpicture}
\]
where $P_1$ projects to the first structured cospan of an object in $\Z$, and $P_2$ projects to the second.  All the arrows here are right exact because colimits are computed pointwise.
Suppose next that $F$ and $G$ below are right exact:
\[
\begin{tikzpicture}[scale=1.5]
\node (Q) at (-1,1) {$\mathsf{Q}$};
\node (A) at (0,0) {$\Z$};
\node (C) at (2,0) {${_L \lCsp(\X)_1}$};
\node (A') at (0,-1) {${_L\lCsp(\X)_1}$};
\node (C') at (2,-1) {${_L \lCsp(\X)_0}.$};
\path[->,font=\scriptsize,>=angle 90]
(Q) edge[out=210,in=180] node[left]{$F$} (A')
(Q) edge[bend left] node[above]{$G$} (C)
(Q) edge[dashed] node[above]{$Q$} (A)
(A) edge node[above]{$P_2$} (C)
(A) edge node [left]{$P_1$} (A')
(C)edge node[right]{$T$}(C')
(A')edge node [below] {$S$}(C');
\end{tikzpicture}
\]
Then there exists a unique functor $Q$ making the diagram commute.  This functor $Q$ is right exact because its composites with $P_1$ and $P_2$ are: since colimits in a diagram category are computed pointwise, a cocone in $\Z$ is a colimit of $F \maps \mathsf{D} \to \Z$ if and only if the `pieces' obtained by applying $P_1$ and $P_2$ to this cocone are colimits of $P_1 \circ F$ and $P_2 \circ F$, respectively.   

The other pullbacks used in defining the double category ${_L\lCsp(\X)}$, such as the pullback ${{}_L\lCsp(\X)_1} \times_{{_L\lCsp(\X)_0}} {_L\lCsp(\X)_1}  \times_{{_L\lCsp(\X)_0}} {_L\lCsp(\X)_1}$ used in defining the associator, are also pullbacks in $\Rex$ for the same sort of reason.
\end{proof}

Next we make ${}_L\lCsp(\X)$ into a weak category in $\SMC$.  We do this by applying the
2-functor $\Phi \maps \Rex \to \SMC$.

\begin{thm}
\label{thm:SMC}
Given a morphism $L \maps \A \to \X$ in $\Rex$, the functor $\Phi \maps \Rex \to \SMC$ maps the weak category ${}_L \lCsp(\X)$ in $\Rex$ to a weak category in $\SMC$.
\end{thm}

\begin{proof}
We need to show that the various pullbacks in $\Rex$ used to make ${}_L \lCsp(\X)$ into a weak category in $\Rex$ are mapped by $\Phi$ to pullbacks in $\SMC$.   We do this only for 
the pullback $\Z = {}_L \lCsp(\X)_1 \times_{{}_L\lCsp(\X)_0} {}_L\lCsp(\X)_1$, since the others are similar.  To show that $\Phi(\Z)$ is the pullback of the following square in $\SMC$:
\[
\begin{tikzpicture}[scale=1.5]
\node (A) at (0,0) {$\Phi(\Z)$};
\node (C) at (2,0) {$\Phi({}_L \lCsp(\X)_1)$};
\node (A') at (0,-1) {$\Phi({}_L \lCsp(\X)_1)$};
\node (C') at (2,-1) {$\Phi({}_L \lCsp(\X)_0)$};
\path[->,font=\scriptsize,>=angle 90]
(A) edge node[above]{$\Phi(P_2)$} (C)
(A) edge node [left]{$\Phi(P_1)$} (A')
(C)edge node[right]{$\Phi(T)$}(C')
(A')edge node [above] {$\Phi(S)$}(C');
\end{tikzpicture}
\]
we need to show that for any symmetric monoidal category $\mathsf{Q}$ and symmetric
monoidal functors $F, G \maps \mathsf{Q} \to \Phi({}_L \lCsp(\X)_1)$ with $\Phi(S) F = \Phi(T) G$, there exists a unique symmetric monoidal functor $Q$ making this diagram commute:
\[
\begin{tikzpicture}[scale=1.5]
\node (Q) at (-1,1) {$\mathsf{Q}$};
\node (A) at (0,0) {$\Phi(\Z)$};
\node (C) at (2,0) {$\Phi({}_L\lCsp(\X)_1)$};
\node (A') at (0,-1) {$\Phi({}_L\lCsp(\X)_1)$};
\node (C') at (2,-1) {$\Phi({}_L\lCsp(\X)_0).$};
\path[->,font=\scriptsize,>=angle 90]
(Q) edge[out=210,in=180] node[left]{$F$} (A')
(Q) edge[bend left] node[above]{$G$} (C)
(Q) edge[dashed] node[above]{$Q$} (A)
(A) edge node[above]{$\Phi(P_2)$} (C)
(A) edge node [left]{$\Phi(P_1)$} (A')
(C)edge node[right]{$\Phi(T)$}(C')
(A')edge node [above] {$\Phi(S)$}(C');
\end{tikzpicture}
\]
By Theorem \ref{thm:rex} there exists a unique right exact functor $Q$ making the underlying diagram of functors commute.  We now show that this $Q$ can be made symmetric monoidal in such a way that the diagram commutes in $\SMC$.

First, let $0_\mathsf{Q}$ be the monoidal unit of $\mathsf{Q}$. Since $F \maps \mathsf{Q} \to \Phi({}_L\lCsp(\X)_1)$ is symmetric monoidal, we have an isomorphism between monoidal units: 
\[  F_0 \maps 0_{\Phi({}_L\lCsp(\X)_1)} \toiso F(0_\mathsf{Q}) \]
where $0_{\Phi({}_L\lCsp(\X)_1)}$ is initial in $\Phi({}_L\lCsp(\X)_1)$.  Similarly we have an
isomorphism 
\[  G_0 \maps 0_{\Phi({}_L\lCsp(\X)_1)} \toiso G(0_\mathsf{Q}) .\]
It follows that $Q(0_\mathsf{Q})$ is a pair of composable initial cospans in $\X$ so there
is a unique isomorphism
\[   Q_0 \colon  0_\Z \toiso Q(0_\mathsf{Q}). \]

Next, given two objects $a_1$ and $a_2$ in $\mathsf{Q}$, we have a natural isomorphism 
\[  F_{a_1,a_2} \colon F(a_1) + F(a_2) \toiso F(a_1 \otimes a_2) \] 
as $F$ is symmetric monoidal, and similarly for $G$. We know that as objects, $F(a_1)$ and $F(a_2)$ are simply cospans in $\X$ with $F(a_1) + F(a_2)$ their chosen coproduct. We also know that $Q(a_1)$ is a pair of composable cospans $(F(a_1),G(a_1))$ and likewise $Q(a_2)$ is a pair of composable cospans $(F(a_2),G(a_2))$. This results in a natural isomorphism 
\[   Q_{a_1,a_2} \maps  Q(a_1) + Q(a_2) \to Q(a_1 \otimes a_2) \]
given by the composite
\[   (F(a_1),G(a_1))+(F(a_2),G(a_2)) \toiso (F(a_1)+F(a_2),G(a_1)+G(a_2)) \toiso (F(a_1 \otimes a_2),G(a_1 \otimes a_2)) .  \]
One can check that this family of natural isomorphisms $Q_{a_1,a_2}$ together with the natural isomorphism $Q_0$ give $Q$ the structure of a symmetric monoidal functor, and that the above diagram then commutes in $\SMC$.
It follows that $\Phi(\Z)$ is a pullback square in $\SMC$, as was to be shown.
\end{proof}

In Theorem \ref{thm:SMC} we made ${}_L\lCsp(\X)$ into a weak category in $\SMC$.
Now we make ${}_L \lCsp(\X)$ into a symmetric monoidal double category.   

\begin{thm}
\label{thm:LCsp(X) symmetric}
Suppose $\A$ and $\X$ have finite colimits and $L \maps \A \to \X$ preserves them.   Choose finite colimits in $\A$ and $\X$.   Then the double category 
${}_L\lCsp(\X)$ becomes symmetric monoidal where:
\begin{itemize}
\item the tensor product of objects $a_1, a_2$ is their chosen coproduct $a_1 + a_2$ in $\A$,
\item the unit object is the chosen initial object $0_\A$ in $\A$,
\item the tensor product of two vertical 1-morphisms is given by
\[
\begin{tikzpicture}[scale=1.5]
\node (E) at (3,0) {$a_1$};
\node (F) at (3,-1) {$b_1$};
\node (G) at (4,0) {$a_2$};
\node (G') at (4,-1) {$b_2$};
\node (A) at (5.5,0) {$a_1 + a_2$};
\node (B) at (5.5,-1) {$b_1 +  b_2$};
\node (D) at (3.45,-0.5) {$\otimes$};
\node (C) at (4.6,-0.5) {$=$};
\path[->,font=\scriptsize,>=angle 90]
(E) edge node[left]{$f_1$} (F)
(G) edge node[left]{$f_2$} (G')
(A) edge node[left]{$f_1 + f_2$} (B);
\end{tikzpicture}
\]
\item the tensor product of horizontal 1-cells is given by
\[
\begin{tikzpicture}[scale=1.2]
\node (A) at (0,0) {$L(a)$};
\node (B) at (1,1) {$x$};
\node (C) at (2,0) {$L(b)$};
\node (D) at (2.5,0.5) {$\otimes$};
\node (E) at (3,0) {$L(a')$};
\node (F) at (4,1) {$x'$};
\node (G) at (5,0) {$L(b')$};
\node (H) at (5.4,0.5) {$=$};
\node (I) at (6,0) {$L(a + a')$};
\node (J) at (7,1) {$x + x'$};
\node (K) at (8,0) {$L(b + b')$};
\path[->,font=\scriptsize,>=angle 90]
(A) edge node[above,left]{$i$} (B)
(C)edge node[above,right]{$o$}(B)
(E) edge node[above,left]{$i'$} (F)
(G)edge node[above,right]{$o'$}(F)
(I) edge node [above,left] {$i + i'$} (J)
(K) edge node [above,right] {$o + o'$} (J);
\end{tikzpicture}
\]
where $i + i'$ and $o + o'$ are defined using the fact that $L$ preserves coproducts,
\item the unit horizontal 1-cell is given by
\[
\begin{tikzpicture}[scale=1.5]
\node (A) at (0,0) {$L(0_\A)$};
\node (B) at (1,0) {$0_\X$};
\node (C) at (2,0) {$L(0_\A)$};
\path[->,font=\scriptsize,>=angle 90]
(A) edge node[above]{$i$} (B)
(C)edge node[above]{$o$}(B);
\end{tikzpicture}
\]
where $0_\X$ is the chosen initial object in $\X$,
\item the tensor product of two 2-morphisms is given by:
\[
\begin{tikzpicture}[scale=1.5]
\node (E) at (3,0) {$L(a_1)$};
\node (F) at (5,0) {$L(b_1)$};
\node (G) at (4,0) {$x_1$};
\node (E') at (3,-1) {$L(a_2)$};
\node (F') at (5,-1) {$L(b_2)$};
\node (G') at (4,-1) {$x_2$};
\node (E'') at (6.5,0) {$L(a_1')$};
\node (F'') at (8.5,0) {$L(b_1')$};
\node (G'') at (7.5,0) {$x_1'$};
\node (E''') at (6.5,-1) {$L(a_2')$};
\node (F''') at (8.5,-1) {$L(b_2')$};
\node (G''') at (7.5,-1) {$x_2'$};
\node (X) at (5.75,-0.5) {$\otimes$};
\node (E'''') at (4,-2) {$L(a_1 + a_1')$};
\node (F'''') at (7.5,-2) {$L(b_1 + b_1')$};
\node (G'''') at (5.75,-2) {$x_1 + x_1'$};
\node (E''''') at (4,-3) {$L(a_2 + a_2')$};
\node (F''''') at (7.5,-3) {$L(b_2 + b_2'),$};
\node (G''''') at (5.75,-3) {$x_2 + x_2'$};
\node (Y) at (3,-2.5) {$=$};
\path[->,font=\scriptsize,>=angle 90]
(F) edge node[above]{$o_1$} (G)
(E) edge node[left]{$L(f)$} (E')
(F) edge node[right]{$L(g)$} (F')
(G) edge node[left]{$\alpha$} (G')
(E) edge node[above]{$i_1$} (G)
(E') edge node[below]{$i_2$} (G')
(F') edge node[below]{$o_2$} (G')
(F'') edge node[above]{$o_1'$} (G'')
(E'') edge node[left]{$L(f')$} (E''')
(F'') edge node[right]{$L(g')$} (F''')
(G'') edge node[left]{$\alpha'$} (G''')
(E'') edge node[above]{$i_1'$} (G'')
(E''') edge node[below]{$i_2'$} (G''')
(F''') edge node[below]{$o_2'$} (G''')
(F'''') edge node[above]{$o_1 + o_1'$} (G'''')
(E'''') edge node[left]{$L(f + f')$} (E''''')
(F'''') edge node[right]{$L(g + g')$} (F''''')
(G'''') edge node[left]{$\alpha + \alpha'$} (G''''')
(E'''') edge node[above]{$i_1 + i_1'$} (G'''')
(E''''') edge node[below]{$i_2 + i_2'$} (G''''')
(F''''') edge node[below]{$o_2 + o_2'$} (G''''');
\end{tikzpicture}
\]
\end{itemize}
and the associators, left and right unitors, and braidings are defined using the universal
properties of binary coproducts and unit objects.
\end{thm}

\begin{proof}
By Theorem \ref{thm:SMC}, ${_L \lCsp(\X)}$ is a weak category object in $\SMC$, 
so both its category of objects and category of arrows are symmetric monoidal.  To show that
it is a symmetric monoidal double category, we need only show that the source and target
functors 
\[   S,T \maps {}_L \lCsp(\X)_1 \to {}_L \lCsp(\X)_0 \]
are \emph{strict} symmetric monoidal \cite[Remark 2.12]{Shul2010}.  This follows because $S$ and $T$ simply pick out the input and output of a structured cospan, and we are using the same chosen binary coproducts and initial object in $\A$ in defining the monoidal structures on both ${}_L \lCsp(\X)_0$ and ${}_L \lCsp(\X)_1$.
\end{proof}

In fact, to make ${}_L \lCsp(\X)$ into a symmetric monoidal double category it suffices for $\A$ to have finite coproducts, $\X$ to have finite colimits, and $L$ to preserve finite coproducts \cite[Theorem 3.2.3]{CourThesis}.  But in the examples we have studied, $\A$ and $\X$ have finite colimits, and $L$, being a left adjoint, preserves all of these.

Next we take the symmetric monoidal double category ${}_L \lCsp(\X)$ and water it down, obtaining first a symmetric monoidal bicategory and then a symmetric monoidal category.   The definition of symmetric monoidal bicategory is nicely presented by Stay \cite{Stay}, who recalls how this definition was gradually discovered by a series of authors.   Shulman \cite{Shul2010} provides a convenient way to construct symmetric monoidal bicategories from symmetric monoidal double categories.  He defines a double category $\lD$ to be isofibrant if every vertical 1-isomorphism has a `companion' and a `conjoint' \cite{GP2}, and proves that if $\lD$ is symmetric monoidal and isofibrant, then $\bD$ becomes symmetric monoidal in a canonical way.

A \define{companion} of a vertical 1-morphism $f \maps a \to b$ is a horizontal 1-cell $\hat{f} \maps a \to b$ equipped with 2-morphisms	
\[
	\raisebox{-0.5\height}{
	\begin{tikzpicture}
		\node (A) at (0,1) {$a$};
		\node (B) at (1,1) {$b$};
		\node (A') at (0,0) {$b$};
		\node (B') at (1,0) {$b$};
		\path[->,font=\scriptsize,>=angle 90]
			(A) edge node[above]{$\hat{f}$} (B)
			(A) edge node[left]{$f$} (A')
			(B) edge node[right]{$1$} (B')
			(A') edge node[below]{$U_b$} (B');
		%
		\node () at (0.5,0.5) {\scriptsize{$\alpha \Downarrow$}};
	\end{tikzpicture}
	}
	\quad \text{ and } \quad
	\raisebox{-0.5\height}{
	\begin{tikzpicture}
		\node (A) at (0,1) {$a$};
		\node (B) at (1,1) {$a$};
		\node (A') at (0,0) {$a$};
		\node (B') at (1,0) {$b$};
		\path[->,font=\scriptsize,>=angle 90]
			(A) edge node[above]{$U_a$} (B)
			(A) edge node[left]{$1$} (A')
			(B) edge node[right]{$f$} (B')
			(A') edge node[below]{ $\hat{f}$} (B');
		%
		\node () at (0.5,0.5) {\scriptsize{$\beta \Downarrow$}};
	\end{tikzpicture}
	}
	\]
that obey these equations:
\begin{equation}
\label{eq:companion}
\raisebox{-0.5\height}{
\begin{tikzpicture}
	
		\node (A) at (0,2) {$a$};
		\node (B) at (1.1,2) {$a$};
		\node (A') at (0,1) {$a$};
		\node (B') at (1.1,1) {$b$};
		\node (A'') at (0,0) {$b$};
		\node (B'') at (1.1,0) {$b$};
		\path[->,font=\scriptsize,>=angle 90]
			(A) edge node[left]{$1$} (A')
			(A') edge node[left]{$f$} (A'')
			(B) edge node[right]{$f$} (B')
			(B') edge node[right]{$1$} (B'')
			(A) edge node[above]{$U_a$} (B)
			(A') edge  (B')
			(A'') edge node[below]{$U_b$} (B'');
		%
		\draw[line width=2mm,white] (0.5,.925) -- (0.5,1.075);
	      \node () at (0.5,1.5) {\scriptsize{$\beta \Downarrow$}};
		\node () at (0.5,0.5) {\scriptsize{$\alpha \Downarrow$}};
		\node () at (0.5,1) {\scriptsize $\hat{f}$};
	\end{tikzpicture}
	}
	\raisebox{-0.5\height}{=}
	\raisebox{-0.5\height}{
	\begin{tikzpicture}
		\node (A) at (0,1) {$a$};
		\node (B) at (1,1) {$a$};
		\node (A') at (0,0) {$b$};
		\node (B') at (1,0) {$b$};
		\path[->,font=\scriptsize,>=angle 90]
		(A) edge node[left]{$f$} (A')
		(B) edge node[right]{$f$} (B')
		(A) edge node[above]{$U_a$} (B)
		(A') edge node[below]{$U_b$} (B');
		%
		\node () at (0.5,0.5) {\scriptsize{$\Downarrow U_f$}};
	\end{tikzpicture}
	}
	\raisebox{-0.5\height}{\text{   and   }}
	\raisebox{-0.5\height}{
	\begin{tikzpicture}
		\node (A) at (0,1) {$a$};
		\node (A') at (0,0) {$a$};
		\node (B) at (1,1) {$a$};
		\node (B') at (1,0) {$b$};
		\node (C) at (2,1) {$b$};
		\node (C') at (2,0) {$b$};
		\node (A'') at (0,-1) {$a$};
		\node (C'') at (2,-1) {$b$};
		\path[->,font=\scriptsize,>=angle 90]
			(A) edge node[left]{$1$} (A')
			(B) edge node[left]{$f$} (B')
			(C) edge node[right]{$1$} (C')
			(A) edge node[above]{$U_a$} (B)
			(B) edge node[above]{$\hat{f}$} (C)
			(A') edge node[below]{$\hat{f}$} (B')
			(B') edge node[below]{$U_b$} (C')
			(A'') edge node[below]{$\hat{f}$} (C'')
			(A') edge node[left]{$1$} (A'')
			(C') edge node[right]{$1$} (C'');
		%
		\node () at (0.5,0.5) {\scriptsize{$\beta \Downarrow$}};
		\node () at (1.5,0.5) {\scriptsize{$\alpha \Downarrow$}};
		\node () at (0.8,-0.6) {\scriptsize{$\lambda_{\hat{f}} \Downarrow$}};
		
	\end{tikzpicture}
	}
	\raisebox{-0.5\height}{=}
	\raisebox{-0.5\height}{
	\begin{tikzpicture}
	     \node (A0) at (0,2) {$a$};
	     \node (B0) at (1,2) {$a$};
		\node (C0) at (2,2) {$b$};
		\node (A) at (0,1) {$a$};
		\node (C) at (2,1) {$b$};
		%
		\path[->,font=\scriptsize,>=angle 90]
			(A0) edge node[above]{$U_a$} (B0)
			(B0) edge node[above]{$\hat{f}$} (C0)
			(A0) edge node[left]{$1$} (A)
			(C0) edge node[right]{$1$} (C)
			(A) edge node[below]{$\hat{f}$} (C);
		%
		\node () at (0.8,1.5) {\scriptsize{$\rho_f \Downarrow$}};
	\end{tikzpicture}
	}
	\end{equation}
A \define{conjoint} of $f$ is a horizontal 1-cell $\check{f} \maps b \to a$ that is a companion of $f$ in the `horizontal opposite' of the double category in question.  Since ${}_L \lCsp(\X)$ is its own horizontal opposite, we only need to check the existence of companions.

\begin{cor}
\label{cor:LCsp(X) symmetric bicategory}
If $\A$ and $\X$ have finite colimits, $L \maps \A \to \X$ preserves them, and we choose finite colimits in both $\A$ and $\X$, then the bicategory ${}_L \bCsp(\X)$ of Corollary \ref{cor:LCsp(X) bicategory} becomes symmetric monoidal as follows:
\begin{itemize}
\item the tensor product of objects $a_1$ and $a_2$ is their chosen coproduct $a_1 + a_2$ in $\A$,
\item
the unit for the tensor product is the chosen initial object $0_A$ in $\A$,
\item the tensor product of 1-morphisms is given by
\[
\begin{tikzpicture}[scale=1.2]
\node (A) at (0,0) {$L(a)$};
\node (B) at (1,1) {$x$};
\node (C) at (2,0) {$L(b)$};
\node (D) at (2.5,0.5) {$\otimes$};
\node (E) at (3,0) {$L(a')$};
\node (F) at (4,1) {$x'$};
\node (G) at (5,0) {$L(b')$};
\node (H) at (5.65,0.5) {$=$};
\node (I) at (6.5,0) {$L(a + a')$};
\node (J) at (7.5,1) {$x + x'$};
\node (K) at (8.5,0) {$L(b + b')$};
\path[->,font=\scriptsize,>=angle 90]
(A) edge node[above,left]{$i$} (B)
(C)edge node[above,right]{$o$}(B)
(E) edge node[above,left]{$i'$} (F)
(G)edge node[above,right]{$o'$}(F)
(I) edge node [above,left] {$i + i'$} (J)
(K) edge node [above,right] {$o + o'$} (J);
\end{tikzpicture}
\]
\item the tensor product of 2-morphisms is given by
\[
\begin{tikzpicture}[scale=1.2]
\node (A) at (-0.5,0) {$L(a_1)$};
\node (B) at (0.5,1) {$x_1$};
\node (B') at (0.5,-1) {$x_1'$};
\node (F') at (4,-1) {$x_2'$};
\node (J') at (8.45,-1) {$x_1' + x_2'$};
\node (C) at (1.5,0) {$L(a_1')$};
\node (D) at (2.25,0) {$\otimes$};
\node (E) at (3,0) {$L(a_2)$};
\node (F) at (4,1) {$x_2$};
\node (G) at (5,0) {$L(a_2')$};
\node (H) at (5.7,0) {$=$};
\node (I) at (6.7,0) {$L(a_1 + a_2)$};
\node (J) at (8.45,1) {$x_1 + x_2$};
\node (K) at (10.2,0) {$L(a_1' + a_2')$};
\path[->,font=\scriptsize,>=angle 90]
(A) edge node[above,left]{$i_1$} (B)
(C)edge node[above,right]{$o_1$}(B)
(A) edge node[below,left]{$i_1'$} (B')
(C)edge node[below,right]{$o_1'$}(B')
(E) edge node[above,left]{$i_2$} (F)
(G)edge node[above,right]{$o_2$}(F)
(E) edge node[below,left]{$i_2'$} (F')
(G)edge node[below,right]{$o_2'$}(F')
(I) edge node [above,left] {$i_1 + i_2\;$} (J)
(K) edge node [above,right] {$\;o_1 + o_2$} (J)
(I) edge node [below,left] {$i_1' + i_2'\;$} (J')
(K) edge node [below,right] {$\;o_1' + o_2'$} (J')
(B) edge node [left] {$\alpha_1$} (B')
(F) edge node [left] {$\alpha_2$} (F')
(J) edge node [left] {$\alpha_1 + \alpha_2$} (J');
\end{tikzpicture}
\]
\item 
the associators, unitors, symmetries, and other structures of a symmetric monoidal
bicategory are constructed using the universal properties of binary coproducts and initial objects.
\end{itemize}
\end{cor}

\begin{proof}

A vertical 1-isomorphism in ${}_L \lCsp(\X)$ is a isomorphism $f \maps a \to b$ in $\A$.  We take its companion $\hat{f}$ to be the structured cospan  
\[   
\begin{tikzpicture}[scale=1.5]
\node (A) at (0,0) {$L(a)$};
\node (B) at (1,0) {$L(b)$};
\node (C) at (2,0) {$L(b).$};
\path[->,font=\scriptsize,>=angle 90]
(A) edge node[above]{$L(f)$} (B)
(C)edge node[above]{$1$}(B);
\end{tikzpicture}
\]
The unit horizontal 1-cells $U_a$ and $U_b$ are given respectively by
\[
\begin{tikzpicture}[scale=1.5]
\node (A) at (0,0) {$L(a)$};
\node (B) at (1,0) {$L(a)$};
\node (C) at (2,0) {$L(a)$};
\node (X) at (3,0) {and};
\node (A') at (4,0) {$L(b)$};
\node (B') at (5,0) {$L(b)$};
\node (C') at (6,0) {$L(b)$};
\path[->,font=\scriptsize,>=angle 90]
(A) edge node[above]{$1$} (B)
(C) edge node [above]{$1$} (B)
(A') edge node[above]{$1$} (B')
(C') edge node [above]{$1$} (B');
\end{tikzpicture}
\]
and the accompanying 2-morphisms $\alpha$ and $\beta$ are given by
\[
\begin{tikzpicture}[scale=1.5]
\node (E) at (3,0) {$L(a)$};
\node (F) at (5,0) {$L(b)$};
\node (G) at (4,0) {$L(b)$};
\node (E') at (3,-1) {$L(b)$};
\node (F') at (5,-1) {$L(b)$};
\node (G') at (4,-1) {$L(b)$};
\node (X) at (6,-0.5) {and};
\node (E'') at (7,0) {$L(a)$};
\node (F'') at (9,0) {$L(a)$};
\node (G'') at (8,0) {$L(a)$};
\node (E''') at (7,-1) {$L(a)$};
\node (F''') at (9,-1) {$L(b)$};
\node (G''') at (8,-1) {$L(b)$};
\path[->,font=\scriptsize,>=angle 90]
(F) edge node[above]{$1$} (G)
(E) edge node[left]{$L(f)$} (E')
(F) edge node[right]{$1$} (F')
(G) edge node[left]{$1$} (G')
(E) edge node[above]{$L(f)$} (G)
(E') edge node[below]{$1$} (G')
(F') edge node[below]{$1$} (G')
(F'') edge node[above]{$1$} (G'')
(E'') edge node[left]{$1$} (E''')
(F'') edge node[right]{$L(f)$} (F''')
(G'') edge node[left]{$L(f)$} (G''')
(E'') edge node[above]{$1$} (G'')
(E''') edge node[below]{$L(f)$} (G''')
(F''') edge node[below]{$1$} (G''');
\end{tikzpicture}
\]
respectively.  An easy calculation verifies Equation (\ref{eq:companion}).
\end{proof}

\begin{cor}
\label{cor:LCsp(X) symmetric category}
If $\A$ and $\X$ have finite colimits, $L \maps \A \to \X$ preserves them, and we choose binary coproducts and an initial object in $\A$, then the category ${}_L \Csp(\X)$ of Corollary \ref{_L Csp(X) category} becomes symmetric monoidal as follows:
\begin{itemize}
\item the tensor product of objects $a_1$ and $a_2$ is their chosen coproduct 
$a_1 + a_2$ in $\A$,
\item
the unit for the tensor product is the chosen initial object $0_\A$ in $\A$, 
\item the tensor product of morphisms is given by
\[
\begin{tikzpicture}[scale=1.2]
\node (A) at (0,0) {$L(a)$};
\node (B) at (1,1) {$x$};
\node (C) at (2,0) {$L(b)$};
\node (D) at (2.5,0.5) {$\otimes$};
\node (E) at (3,0) {$L(a')$};
\node (F) at (4,1) {$x'$};
\node (G) at (5,0) {$L(b')$};
\node (H) at (5.4,0.5) {$=$};
\node (I) at (6,0) {$L(a + a')$};
\node (J) at (7,1) {$x + x'$};
\node (K) at (8,0) {$L(b + b')$};
\path[->,font=\scriptsize,>=angle 90]
(A) edge node[above,left]{$i$} (B)
(C)edge node[above,right]{$o$}(B)
(E) edge node[above,left]{$i'$} (F)
(G)edge node[above,right]{$o'$}(F)
(I) edge node [above,left] {$i + i'$} (J)
(K) edge node [above,right] {$o + o'$} (J);
\end{tikzpicture}
\]
where in each case the cospan actually denotes an isomorphism class of cospans,
\item 
the associator, left and right unitors, and symmetry are constructed using the universal
properties of binary coproducts and initial objects.
\end{itemize}
\end{cor}

\begin{proof}
It can be checked by inspecting the definitions that
any symmetric monoidal bicategory $\B$ gives rise to a symmetric monoidal category $\mathsf{B}$ where:
\begin{itemize}
\item
the objects of $\mathsf{B}$ are those of $\B$,
\item
the morphisms of $\mathsf{B}$ are isomorphism classes of morphisms of $\B$,
\item
the unit object and the tensor product of objects are those of $\B$,
\item
the tensor product of morphisms, the associator,
the left and right unitor, and the symmetry of $\mathsf{B}$ arise from those of $\B$
by taking isomorphism classes.
\end{itemize}
Applying this `decategorification' construction to the symmetric monoidal bicategory \break
${}_L \bCsp(\X)$ gives the symmetric monoidal category ${}_L \Csp(\X)$.   
\end{proof}

The symmetric monoidal category  ${}_L \Csp(\X)$ is determined up to equality by the choice of binary coproducts and initial object in $\A$, but different choices of this data give isomorphic symmetric monoidal categories.

Readers interested in hypergraph categories may be pleased to learn that structured cospan categories tend to be of this type.  A `hypergraph category' is a symmetric monoidal category where each object has the structure of a special commutative Frobenius monoid in a way that is compatible with tensor products but not necessarily preserved
by morphisms \cite{Fong2015}.  Such categories are ubiquitous in network theory, where Frobenius structure allows us to split, join, start and terminate strings in string diagrams \cite{Fong2016}.    While the definition of hypergraph category may seem awkward at first, Fong and Spivak have clarified this concept using operads \cite{FongSpivak}.

\begin{thm}
\label{thm:hypergraph}
If $\A$ and $\X$ have finite colimits, $L \maps \A \to \X$ preserves them, and we choose binary coproducts and an initial object in $\A$, then the symmetric monoidal category ${}_L \Csp(\X)$ is a hypergraph category where each object $a \in \A$ is a special commutative Frobenius monoid as follows:
\begin{itemize}
\item
The multiplication is given by the structured cospan
\[
\begin{tikzpicture}[scale=1.5]
\node (A) at (0,0) {$L(a+a)$};
\node (B) at (1.5,0) {$L(a)$};
\node (C) at (2.5,0) {$L(a)$.};
\path[->,font=\scriptsize,>=angle 90]
(A) edge node[above]{$L(\nabla)$} (B)
(C) edge node[above]{$1$}(B);
\end{tikzpicture}
\]
where $\nabla \maps a+a \to a$ is the fold map.
\item
The unit is given by
\[
\begin{tikzpicture}[scale=1.5]
\node (A) at (0,0) {$L(0)$};
\node (B) at (1,0) {$L(a)$};
\node (C) at (2,0) {$L(a)$.};
\path[->,font=\scriptsize,>=angle 90]
(A) edge node[above]{$L(!)$} (B)
(C) edge node[above]{$1$}(B);
\end{tikzpicture}
\]
where $! \maps 0 \to a$ is the unique morphism.
\item 
The comultiplication is given by
\[
\begin{tikzpicture}[scale=1.5]
\node (A) at (0,0) {$L(a)$};
\node (B) at (1,0) {$L(a)$};
\node (C) at (2.5,0) {$L(a+a)$.};
\path[->,font=\scriptsize,>=angle 90]
(A) edge node[above]{$1$} (B)
(C) edge node[above]{$L(\nabla)$}(B);
\end{tikzpicture}
\]
\item
The counit is given by
\[
\begin{tikzpicture}[scale=1.5]
\node (A) at (0,0) {$L(a)$};
\node (B) at (1,0) {$L(a)$};
\node (C) at (2,0) {$L(0)$.};
\path[->,font=\scriptsize,>=angle 90]
(A) edge node[above]{$1$} (B)
(C) edge node[above]{$L(!)$}(B);
\end{tikzpicture}
\]
\end{itemize}
\end{thm}

\begin{proof}
Whenever $F \maps \C \to \D$ is a symmetric monoidal functor bijective on objects and $\C$ is a hypergraph category, there is a unique way to make $\D$ into a hypergraph category such that $F$ is a hypergraph functor.  To see this, first note that $F$ equips each object of $\D$ with the structure of a special commutative Frobenius monoid, coming from its structure in $\C$.  These Frobenius structures are compatible with tensor product because they were in $\C$ and $F$ is symmetric monoidal.  Thus, $\D$ becomes a hypergraph category.   By construction $F \maps \C \to \D$ preserves the Frobenius structures on objects, so $F$ is a hypergraph functor.  Moreover, the Frobenius structures on objects of $\D$ are uniquely determined by this requirement.

Let $\Csp(\A)$ 
be the symmetric monoidal category whose morphisms are isomorphism classes of
cospans in $\A$.   Since $L$ preserves finite colimits, there is a symmetric monoidal 
functor $F \maps \Csp(\A) \to {}_L\Csp(\X)$ given as follows:
\[
\begin{tikzpicture}[scale=1.2]
\node (A) at (0,0) {$a$};
\node (B) at (1,1) {$c$};
\node (C) at (2,0) {$b$};
\node (D) at (2.5,0.5) {$\mapsto$};
\node (E) at (3,0) {$L(a)$};
\node (F) at (4,1) {$L(c)$};
\node (G) at (5,0) {$L(b).$};
\path[->,font=\scriptsize,>=angle 90]
(A) edge node[above,left]{$i$} (B)
(C) edge node[above,right]{$o$}(B)
(E) edge node[above,left]{$L(i)$} (F)
(G) edge node[above,right]{$L(o)$}(F);
\end{tikzpicture}
\]
This is bijective on objects, and $\Csp(\A)$ is a hypergraph category \cite{Fong2015}, so
${}_L\Csp(\X)$ has a unique hypergraph category structure making $F$ into a hypergraph
functor.   This is given as in the statement of the theorem.
\end{proof}

\section{Maps between structured cospan double categories}

In this section we show how to construct maps between structured cospan categories, or 
bicategories, or double categories.   As before, it is best to start with double categories and work our way down.   A map between double categories is called a `double functor', and these are defined in Definition \ref{defn:double_functor}.   Suppose that we have structured cospan double categories coming from functors $L \maps \A \to \X$ and $L' \maps \A' \to X'$, where $\X$ and $\X'$ have chosen pushouts.   Then we get a double functor between these double categories from a diagram of this form:
\[
\begin{tikzpicture}[scale=1.5]
\node (A) at (0,0) {$\A$};
\node (B) at (1,0) {$\X$};
\node (C) at (0,-1) {$\A'$};
\node (D) at (1,-1) {$\X'$};
\node (F) at (0.5,-0.5) {$\alpha \NEarrow$};
\path[->,font=\scriptsize,>=angle 90]
(A) edge node [above] {$L$} (B)
(A) edge node [left]{$F_0$} (C)
(B) edge node [right]{$F_1$} (D)
(C) edge node [below] {$L'$} (D);
\end{tikzpicture}
\]
where $F_0$ is a functor, $F_1$ is a functor preserving pushouts, and $\alpha$ is a natural isomorphism.
We prove this in Theorem \ref{thm:double_functor}.  Furthermore, if all four categories involved have finite colimits and all four functors preserve these, then this double functor is symmetric monoidal---a concept defined in Definition \ref{defn:monoidal_double_functor}.   We prove this in Theorem \ref{thm:symmetric_monoidal_double_functor}.

\begin{defn}
Given a 2-category $\bC$ and two weak categories $\lD$ and $\lD'$ in $\bC$, a \define{weak functor}  $\lF \colon \lD \to \lD'$ in $\bC$ consists of:
\begin{itemize}
\item a morphism of objects $\lF_0 \colon \lD_0 \to \lD_0'$,
\item a morphism of arrows $\lF_1 \colon \lD_1 \to \lD_1'$,
\end{itemize}
such that:
\begin{itemize}
\item $\lF$ preserves the source and target morphisms:
$S' \circ \lF_1 = \lF_0 \circ S$ and $T' \circ \lF_1 = \lF_0 \circ T$,
\item
composition and the identity-assigning morphism are preserved up to 2-isomorphisms $\lF_\odot$
and $\lF_U$, respectively:
\[
\begin{tikzpicture}[scale=1.5]
\node (A) at (0,0) {$\lD_1 \times_{\lD_0} \lD_1$};
\node (B) at (1.5,0) {$\lD_1$};
\node (C) at (0,-1) {$\lD_1' \times_{\lD_0'} \lD_1'$};
\node (D) at (1.5,-1) {$\lD_1'$};
\node (E) at (0.75,-0.5) {$\lF_{\odot} \NEarrow$};
\node (A') at (2.5,0) {$\lD_0$};
\node (B') at (4,0) {$\lD_1$};
\node (C') at (2.5,-1) {$\lD_0'$};
\node (D') at (4,-1) {$\lD_1'$};
\node (E') at (3.25,-0.5) {$\lF_U \NEarrow$};
\path[->,font=\scriptsize,>=angle 90]
(A) edge node[above]{$\circ$} (B)
(B) edge node [right]{$\lF_1$} (D)
(C) edge node [above] {$\circ'$} (D)
(A)edge node[left]{$\lF_1 \times_{\lF_0} \lF_1$}(C)
(A') edge node[above]{$U$} (B')
(B') edge node [right]{$\lF_1$} (D')
(C') edge node [above] {$U'$} (D')
(A')edge node[left]{$\lF_0$}(C');
\end{tikzpicture}
\]
\item 
the 2-isomorphisms $\lF_\odot$ and $\lF_U$ satisfy the hexagon and square identities familiar from the definition of a monoidal functor.
\end{itemize}
\end{defn}
\noindent
A weak functor in $\Cat$ is the same as a double functor, and one can consult Definition \ref{defn:double_functor} to see the hexagon and square identities in this case.  We will also need weak functors in $\Rex$ and $\SMC$.

We begin by getting double functors between structured cospan double categories.

\begin{thm}
\label{thm:double_functor}
Suppose we have a square in $\Cat$:
\[
\begin{tikzpicture}[scale=1.5]
\node (A) at (0,0) {$\A$};
\node (B) at (1,0) {$\X$};
\node (C) at (0,-1) {$\A'$};
\node (D) at (1,-1) {$\X'$};
\node (E) at (0.5,-0.5) {$\alpha \NEarrow$};
\path[->,font=\scriptsize,>=angle 90]
(A) edge node[above]{$L$} (B)
(B) edge node [right]{$F_1$} (D)
(C) edge node [below] {$L'$} (D)
(A)edge node[left]{$F_0$}(C);
\end{tikzpicture}
\]
where $\X$ and $\X'$ have chosen pushouts, $F_1$ preserves pushouts and $\alpha$ is a natural isomorphism. 
Then there is a double functor 
$\lF \maps {}_L \lCsp(\X) \to {}_{L'} \lCsp(\X')$ such that:
\begin{itemize}
\item $\lF_0 = F_0$.
\item $\lF_1$ acts as follows on objects:
\[
\begin{tikzpicture}[scale=1.5]
\node (A) at (0,0) {$L(a)$};
\node (B) at (0.75,0) {$x$};
\node (C) at (1.5,0) {$L(b)$};
\path[->,font=\scriptsize,>=angle 90]
(A) edge node[above]{$i$} (B)
(C)edge node[above]{$o$}(B);
\end{tikzpicture}
\raisebox{0.5em}{$\quad \mapsto \quad$}
\begin{tikzpicture}[scale=1.5]
\node (A) at (0,0) {$L'(F_0(a))$};
\node (B) at (1.5,0) {$F_1(x)$};
\node (C) at (3,0) {$L'(F_0(b))$};
\path[->,font=\scriptsize,>=angle 90]
(A) edge node[above]{$ F_1(i) \alpha_a$} (B)
(C)edge node[above]{$ F_1(o)\alpha_b $}(B);
\end{tikzpicture}
\]
and as follows on morphisms:
\[
\begin{tikzpicture}[scale=1.5]
\node (A) at (0,0) {$L(a)$};
\node (B) at (0.75,0) {$x$};
\node (C) at (1.5,0) {$L(b)$};
\node (A') at (0,-1) {$L(a')$};
\node (B') at (0.75,-1) {$x'$};
\node (C') at (1.5,-1) {$L(b')$};
\path[->,font=\scriptsize,>=angle 90]
(A) edge node[above]{$i$} (B)
(C)edge node[above]{$o$}(B)
(A') edge node[above]{$i'$} (B')
(C')edge node[above]{$o'$}(B')
(C)edge node[right]{$L(g)$}(C')
(B)edge node[left]{$\gamma$}(B')
(A)edge node[left]{$L(f)$}(A');
\end{tikzpicture}
\raisebox{2.75em}{$\quad \mapsto \quad$}
\begin{tikzpicture}[scale=1.5]
\node (A) at (0,0) {$L'(F_0(a))$};
\node (B) at (1.5,0) {$F_1(x)$};
\node (C) at (3,0) {$L'(F_0(b))$};
\node (A') at (0,-1) {$L'(F_0(a'))$};
\node (B') at (1.5,-1) {$F_1(x')$};
\node (C') at (3,-1) {$L'(F_0(b'))$};
\path[->,font=\scriptsize,>=angle 90]
(A) edge node[above]{$ F_1(i) \alpha_a$} (B)
(C)edge node[above]{$ F_1(o)\alpha_b $}(B)
(A') edge node[above]{$ F_1(i') \alpha_{a'}$} (B')
(C')edge node[above]{$ F_1(o')\alpha_{b'}$}(B')
(A)edge node[left]{$ L'(F_0(f)) $}(A')
(B)edge node[left]{$ F_1(\gamma)$}(B')
(C)edge node[right]{$L'(F_0(g))$}(C');
\end{tikzpicture}
\]
\item Given composable structured cospans in ${}_L \lCsp(\X)$:
\[
\begin{tikzpicture}[scale=1.5]
\node (A) at (0,0) {$L(a)$};
\node (B) at (1,0) {$x$};
\node (C) at (2,0) {$L(b)$};
\node (A') at (3,0) {$L(b)$};
\node (B') at (4,0) {$y$};
\node (C') at (5,0) {$L(c)$};
\path[->,font=\scriptsize,>=angle 90]
(A) edge node[above]{$i$} (B)
(C) edge node [above]{$o$} (B)
(A') edge node[above]{$i'$} (B')
(C') edge node [above]{$o'$} (B');
\end{tikzpicture}
\]
the natural isomorphism $\lF_{\odot} \colon \lF_1(M) \odot \lF_1(N) \to \lF_1(M \odot N)$ is given by this map of cospans:
\[
\begin{tikzpicture}[scale=1.5]
\node (A) at (0,0) {$L'(F_0(a))$};
\node (B) at (3,0) {$F_1(x)+_{L'(F_0(b))}F_1(y)$};
\node (C) at (6,0) {$L'(F_0(c))$};
\node (A') at (0,-1) {$L'(F_0(a))$};
\node (B') at (3,-1) {$F_1(x+_{L(b)}y)$};
\node (C') at (6,-1) {$L'(F_0(c))$};
\path[->,font=\scriptsize,>=angle 90]
(A) edge node[above]{$\Psi j_{F_1(x)} F_1(i)\alpha_a$} (B)
(C) edge node [above]{$\Psi j_{F_1(y)} F_1(o')\alpha_a$} (B)
(A') edge node[above]{$F_1(\psi j_x i)\alpha_c$} (B')
(C') edge node [above]{$F_1(\psi j_y o')\alpha_c$} (B')
(A) edge node [left]{$1$} (A')
(B) edge node [left]{$\phi_{M,N}$} (B')
(C) edge node [right]{$1$} (C');
\end{tikzpicture}
\]
Here $j_x \colon x \to x+y$ is the natural map into a coproduct, and likewise for $j_y,j_{F_1(x)},j_{F_1(y)}$, $\psi \maps x+y \to x+_{L(b)}y$ is the natural map from a coproduct to a pushout and likewise for $\Psi$, and $\phi_{M,N} \maps F_1(x)+_{L'(F_0(b))}F_1(y) \to F_1(x+_{L(b)}y)$ is given by the composite 
\[  F_1(x)+_{L'(F_0(b))}F_1(y) \xrightarrow{\id +_{\alpha_b} \id} F_1(x) +_{F_1(L(b))} F_1(y) \xrightarrow{\kappa} F_1(x+_{L(b)}y)
\]  
where $\kappa$ is the natural isomorphism arising from $F_1$ preserving pushouts. 
\item Given an object $a \in A$, the natural isomorphism $\lF_U \colon U'(\lF_0(a)) \to \lF_1(U(a))$ is given by this map of cospans:
\[
\begin{tikzpicture}[scale=1.5]
\node (A) at (0,0) {$L'(F_0(a))$};
\node (B) at (1.5,0) {$L'(F_0(a))$};
\node (C) at (3,0) {$L'(F_0(a))$};
\node (A') at (0,-1) {$L'(F_0(a))$};
\node (B') at (1.5,-1) {$F_1(L(a))$};
\node (C') at (3,-1) {$L'(F_0(a))$};
\path[->,font=\scriptsize,>=angle 90]
(A) edge node[above]{$1$} (B)
(C) edge node [above]{$1$} (B)
(A') edge node[above]{$\alpha_a$} (B')
(C') edge node [above]{$\alpha_a$} (B')
(A) edge node [left]{$1$} (A')
(B) edge node [left]{$\alpha_a$} (B')
(C) edge node [right]{$1$} (C');
\end{tikzpicture}
\]
\end{itemize}
\end{thm}

\begin{proof}
The diagram in the definition of $\lF_\odot$ commutes as 
\[  F_1(\psi j_x i)\alpha_a = F_1(\psi) F_{1_{x,y}} j_{F_1(x)}F_1(i)\alpha_a =  \phi_{M,N} \Psi j_{F_1(x)} F_1(i)\alpha_a\]
where $F_{1_{x,y}} \maps F_1(x)+F_1(y) \to F_1(x+y)$ is the natural isomorphism arising from $F_1$ preserving binary coproducts. One can check that the natural isomorphisms $\lF_\odot$ and $\lF_U$ satisfy the left and right unit squares and laxator hexagon of a monoidal functor.
\end{proof}

\begin{thm}
\label{thm:symmetric_monoidal_double_functor}
Suppose we have a square commuting up to isomorphism in $\Rex$:
\[
\begin{tikzpicture}[scale=1.5]
\node (A) at (0,0) {$\A$};
\node (B) at (1,0) {$\X$};
\node (C) at (0,-1) {$\A'$};
\node (D) at (1,-1) {$\X'$};
\node (E) at (0.5,-0.5) {$\alpha \NEarrow$};
\path[->,font=\scriptsize,>=angle 90]
(A) edge node[above]{$L$} (B)
(B) edge node [right]{$F_1$} (D)
(C) edge node [above] {$L'$} (D)
(A)edge node[left]{$F_0$}(C);
\end{tikzpicture}
\]
Then the double functor $\lF \colon {}_L \lCsp(\X) \to {}_{L'} \lCsp(\X')$ is a weak functor between weak category objects in $\Rex$.   Moreover, if we make $ {}_L \lCsp(\X)$ and $ {}_{L'} \lCsp(\X')$ into symmetric monoidal double categories as in Theorem \ref{thm:LCsp(X) symmetric}, then $\lF \maps {_L \lCsp(\X)} \to {_{L'} \lCsp(\X')}$ can be given the structure of a symmetric monoidal double functor.
\end{thm}

\begin{proof}
This is a straightforward but lengthy verification.
\end{proof}

We can then water down this result, obtaining maps between symmetric monoidal bicategories or categories:

\begin{thm}
\label{thm:symmetric_monoidal_functor_bicategories}
A symmetric monoidal double functor $\lF \maps {}_L \lCsp(\X) \to {}_{L'} \lCsp(\X')$ induces
a symmetric monoidal functor $\mathbf{F} \maps {}_L \bCsp(\X) \to {}_{L'} \bCsp(\X')$.
\end{thm}

\begin{proof}
See Hansen and Shulman \cite{HS} for details of how this works, and a proof.
\end{proof}

\begin{thm}
\label{thm:symmetric_monoidal_functor_ategories}
A symmetric monoidal functor between bicategories $\mathbf{F} \maps {}_L \bCsp(\X) \to {}_{L'} \bCsp(\X')$ induces a symmetric monoidal functor between categories $F \maps {}_L \Csp(\X) \to {}_{L'} \Csp(\X')$.
\end{thm}

\begin{proof}
This is a straightforward decategorification process. 
\end{proof}

\section{Structured versus decorated cospans}
\label{sec:decorated}

We can illustrate some of the advantages of structured over decorated categories with an 
example that is fundamental in the study of networks: the double category with open graphs
as morphisms.   An `open graph' consists of a graph together with maps from two sets into its set of 
nodes:
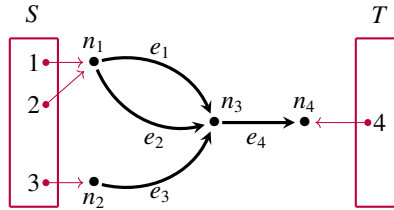
\[
\scalebox{0.8}{
\begin{tikzpicture}
	\begin{pgfonlayer}{nodelayer}
		\node [contact] (I) at (-2,1) {$\bullet$};
		\node [style = none] at (-2,1.3) {$n_1$};
		\node [contact] (T) at (-2,-1) {$\bullet$};
		\node [style = none] at (-2,-1.3) {$n_2$};
		\node [contact] (W) at (0,0) {$\bullet$};
		\node [style = none] at (0.3,0.3) {$n_3$};
		\node [contact] (Water) at (1.5,0) {$\bullet$};
		\node [style = none] at (1.5,0.3) {$n_4$};
		
		\node [style = none] at (-0.9,1.2) {$e_1$};
		\node [style = none] at (-1,-0.3) {$e_2$};
		\node [style = none] at (-0.9,-1.2) {$e_3$};
		\node [style = none] at (0.7,-0.3) {$e_4$};
		
		\node [style=none] (1) at (-3,1) {1};
		\node [style=none] (2'') at (-3,0.3) {2};
		\node [style=none] (2) at (-3,-1) {3};

		\node [style=none] (ATL) at (-3.4,1.4) {};
		\node [style=none] (ATR) at (-2.6,1.4) {};
		\node [style=none] (ABR) at (-2.6,-1.4) {};
		\node [style=none] (ABL) at (-3.4,-1.4) {};

		\node [style=none] (X) at (-3,1.8) {$S$};

		\node [style=inputdot] (inI) at (-2.8,1) {};
		\node [style=inputdot] (inS) at (-2.8,-1) {};
		\node [style=inputdot] (inI') at (-2.8,0.3) {};
		\node [style=none] (Z) at (2.75,1.8) {$T$};
		\node [style=none] (1'') at (2.75,0) {4};
		\node [style=none] (MTL) at (2.35,1.4) {};
		\node [style=none] (MTR) at (3.15,1.4) {};
		\node [style=none] (MBR) at (3.15,-1.4) {};
		\node [style=none] (MBL) at (2.35,-1.4) {};
		\node [style=inputdot] (outI') at (2.55,0) {};

	\end{pgfonlayer}
	\begin{pgfonlayer}{edgelayer}
		\draw [style=inarrow, bend right=40, looseness=1.00] (I) to (W);
		\draw [style=inarrow, bend left=40, looseness=1.00] (I) to (W);
		\draw [style=inarrow, bend right=40, looseness=1.00] (T) to (W);
		\draw [style=inarrow] (W) to (Water);
		\draw [style=simple] (ATL.center) to (ATR.center);
		\draw [style=simple] (ATR.center) to (ABR.center);
		\draw [style=simple] (ABR.center) to (ABL.center);
		\draw [style=simple] (ABL.center) to (ATL.center);
		\draw [style=simple] (MTL.center) to (MTR.center);
		\draw [style=simple] (MTR.center) to (MBR.center);
		\draw [style=simple] (MBR.center) to (MBL.center);
		\draw [style=simple] (MBL.center) to (MTL.center);
		\draw [style=inputarrow] (inI) to (I);
		\draw [style=inputarrow] (inI') to (I);
		\draw [style=inputarrow] (inS) to (T);
		\draw [style=inputarrow] (outI') to (Water);
	\end{pgfonlayer}
\end{tikzpicture}
}
\]
 As usual in category theory, by `graph' we mean a directed multigraph or quiver. In what follows we  restrict attention to finite graphs because these are the most important in applications.  

\begin{defn} 
A \define{graph} is a pair of functions $s,t\maps E \to N$ where $E$ and $N$ are finite sets.   We call elements of $E$ \define{edges} and elements of $N$ \define{nodes}.  We say that the edge $e \in E$ has \define{source} $s(e)$ and \define{target} $t(e)$, and say that $e$ is an edge \define{from} $s(e)$ \define{to} $t(e)$.  A \define{morphism} from the graph $s,t\maps E \to N$ to the graph $s',t' \maps E' \to N'$ is a pair of functions $f \maps E \to E', g \maps N \to N'$ such that these diagrams commute:
\[
\begin{tikzpicture}[scale=1.5]
\node (A) at (0,0) {$E$};
\node (A') at (0,-1) {$E'$};
\node (B) at (1,0) {$N$};
\node (B') at (1,-1) {$N'$};
\path[->,font=\scriptsize,>=angle 90]
(A) edge node[above]{$s$} (B)
(A') edge node[above]{$s'$} (B')
(A) edge node[left]{$f$} (A')
(B) edge node[right]{$g$} (B');

\node (C) at (2,0) {$E$};
\node (C') at (2,-1) {$E'$};
\node (D) at (3,0) {$N$};
\node (D') at (3,-1) {$N'$.};
\path[->,font=\scriptsize,>=angle 90]
(C) edge node[above]{$t$} (D)
(C') edge node[above]{$t'$} (D')
(C) edge node[left]{$f$} (C')
(D) edge node[right]{$g$} (D');
\end{tikzpicture}
\]
\end{defn}

\begin{defn}
\label{defn:Graph}
Let $\Graph$ be the category of graphs and morphisms between them, 
with composition defined by 
\[  (f, g) \circ (f',g') = (f \circ f' , g \circ g')  .\]
\end{defn}

There is a functor $U \maps \Graph \to \Fin\Set$ that takes a graph $s,t \maps E \to N$ to its underlying set of nodes $N$. This has a left adjoint $L \maps \Fin\Set \to \Graph$ sending any set to the graph with that set of nodes and no edges.  Both $\Fin\Set$ and $\Graph$ have finite colimits, and $L$, being a left adjoint, preserves them.  Thus Theorem \ref{thm:LCsp(X) symmetric} gives us a symmetric monoidal double category $_L \lCsp(\Graph)$ where:
\begin{itemize}
\item an object is a finite set,
\item a vertical 1-morphism is a function between finite sets,
\item a horizontal 1-cell from $S$ to $T$ is an \define{open graph}, meaning a
cospan in $\Graph$ of this form:
\[
\begin{tikzpicture}[scale=1.5]
\node (A) at (0,0) {$L(S)$};
\node (B) at (1,0) {$G$};
\node (C) at (2,0) {$L(T)$,};
\path[->,font=\scriptsize,>=angle 90]
(A) edge node[above,left]{$$} (B)
(C)edge node[above]{$$}(B);
\end{tikzpicture}
\]
\item a 2-morphism is a \define{map of open graphs}, meaning
a commutative diagram in $\Graph$ of this form:
\[
\begin{tikzpicture}[scale=1.5]
\node (E) at (3,0) {$L(S)$};
\node (F) at (5,0) {$L(T)$};
\node (G) at (4,0) {$G$};
\node (E') at (3,-1) {$L(S')$};
\node (F') at (5,-1) {$L(T')$};
\node (G') at (4,-1) {$G'$};
\path[->,font=\scriptsize,>=angle 90]
(F) edge node[above]{$$} (G)
(E) edge node[left]{$L(f)$} (E')
(F) edge node[right]{$L(g)$} (F')
(G) edge node[left]{$h$} (G')
(E) edge node[above]{$$} (G)
(E') edge node[below]{$$} (G')
(F') edge node[below]{$$} (G');
\end{tikzpicture}
\]
\end{itemize} 
Applying Corollary \ref{cor:LCsp(X) symmetric bicategory} we obtain a symmetric monoidal bicategory ${}_L \bCsp(\Graph)$ where the objects are finite sets, the morphisms are open graphs, and the
2-morphisms are commutative diagrams in $\Graph$ of this form:
\[
\begin{tikzpicture}[scale=1.5]
\node (E) at (3,-0.5) {$L(S)$};
\node (F) at (5,-0.5) {$L(T)$};
\node (G) at (4,0) {$G$};
\node (G') at (4,-1) {$G'$};
\path[->,font=\scriptsize,>=angle 90]
(F) edge node[above]{$o$} (G)
(G) edge node[left]{$h$} (G')
(E) edge node[above]{$i$} (G)
(E) edge node[below]{$i'$} (G')
(F) edge node[below]{$o'$} (G');
\end{tikzpicture}
\]

We can go further and apply Corollary \ref{cor:LCsp(X) symmetric category} to obtain a
symmetric monoidal category ${}_L \Csp(\Graph)$ where the objects are finite sets
and the morphisms are \emph{isomorphism classes} of open graphs.   An isomorphism of
open graphs is a diagram as above where $h$ is an isomorphism.  Below is a pair of isomorphic open graphs.
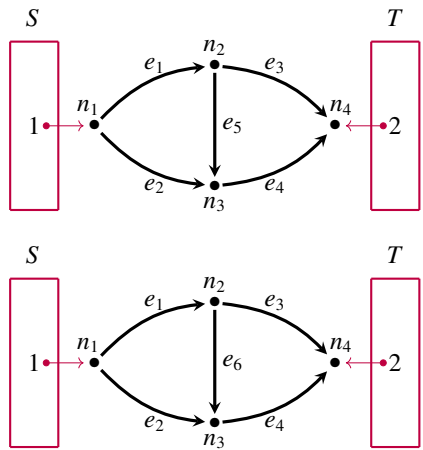
\[
\scalebox{0.8}{
\begin{tikzpicture}
	\begin{pgfonlayer}{nodelayer}
		\node [contact] (n1) at (-2,0) {$\bullet$};
		\node [style = none] at (-2.1,0.3) {$n_1$};
		\node [contact] (n2) at (0,1) {$\bullet$};
		\node [style = none] at (0,1.3) {$n_2$};
		\node [contact] (n3) at (0,-1) {$\bullet$};
		\node [style = none] at (0,-1.3) {$n_3$};
		\node [contact] (n4) at (2,0) {$\bullet$};
		\node [style = none] at (2.1,0.3) {$n_4$};
		
		\node [style = none] at (-1,1) {$e_1$};
		\node [style = none] at (-1,-1) {$e_2$};
		\node [style = none] at (1,1) {$e_3$};
		\node [style = none] at (1,-1) {$e_4$};
	    \node [style = none] at (0.3,0) {$e_5$};
		
		\node [style=none] (1) at (-3,0) {1};
		\node [style=none] (4) at (3,0) {2};
	
		\node [style=none] (ATL) at (-3.4,1.4) {};
		\node [style=none] (ATR) at (-2.6,1.4) {};
		\node [style=none] (ABR) at (-2.6,-1.4) {};
		\node [style=none] (ABL) at (-3.4,-1.4) {};

		\node [style=none] (X) at (-3,1.8) {$S$};
		\node [style=inputdot] (inI) at (-2.8,0) {};
		
		\node [style=none] (Z) at (3,1.8) {$T$};
	 \node [style=inputdot] (outI') at (2.8,0) {};

		\node [style=none] (MTL) at (2.6,1.4) {};
		\node [style=none] (MTR) at (3.4,1.4) {};
		\node [style=none] (MBR) at (3.4,-1.4) {};
		\node [style=none] (MBL) at (2.6,-1.4) {};
	
	\end{pgfonlayer}
	\begin{pgfonlayer}{edgelayer}
		\draw [style=inarrow, bend left=20, looseness=1.00] (n1) to (n2);
		\draw [style=inarrow, bend right=20, looseness=1.00] (n1) to (n3);
		\draw [style=inarrow, bend left=20, looseness=1.00] (n2) to (n4);
		\draw [style=inarrow, bend right=20, looseness=1.00] (n3) to (n4);
		\draw [style=inarrow] (n2) to (n3);
		\draw [style=simple] (ATL.center) to (ATR.center);
		\draw [style=simple] (ATR.center) to (ABR.center);
		\draw [style=simple] (ABR.center) to (ABL.center);
		\draw [style=simple] (ABL.center) to (ATL.center);
		\draw [style=simple] (MTL.center) to (MTR.center);
		\draw [style=simple] (MTR.center) to (MBR.center);
		\draw [style=simple] (MBR.center) to (MBL.center);
		\draw [style=simple] (MBL.center) to (MTL.center);
		\draw [style=inputarrow] (inI) to (n1);
		\draw [style=inputarrow] (outI') to (n4);
	\end{pgfonlayer}
\end{tikzpicture}
}
\]
\[
\scalebox{0.8}{
\begin{tikzpicture}
	\begin{pgfonlayer}{nodelayer}
		\node [contact] (n1) at (-2,0) {$\bullet$};
		\node [style = none] at (-2.1,0.3) {$n_1$};
		\node [contact] (n2) at (0,1) {$\bullet$};
		\node [style = none] at (0,1.3) {$n_2$};
		\node [contact] (n3) at (0,-1) {$\bullet$};
		\node [style = none] at (0,-1.3) {$n_3$};
		\node [contact] (n4) at (2,0) {$\bullet$};
		\node [style = none] at (2.1,0.3) {$n_4$};
		
		\node [style = none] at (-1,1) {$e_1$};
		\node [style = none] at (-1,-1) {$e_2$};
		\node [style = none] at (1,1) {$e_3$};
		\node [style = none] at (1,-1) {$e_4$};
	    \node [style = none] at (0.3,0) {$e_6$};
		
		\node [style=none] (1) at (-3,0) {1};
		\node [style=none] (4) at (3,0) {2};
	
		\node [style=none] (ATL) at (-3.4,1.4) {};
		\node [style=none] (ATR) at (-2.6,1.4) {};
		\node [style=none] (ABR) at (-2.6,-1.4) {};
		\node [style=none] (ABL) at (-3.4,-1.4) {};

		\node [style=none] (X) at (-3,1.8) {$S$};
		\node [style=inputdot] (inI) at (-2.8,0) {};
		
		\node [style=none] (Z) at (3,1.8) {$T$};
	 \node [style=inputdot] (outI') at (2.8,0) {};

		\node [style=none] (MTL) at (2.6,1.4) {};
		\node [style=none] (MTR) at (3.4,1.4) {};
		\node [style=none] (MBR) at (3.4,-1.4) {};
		\node [style=none] (MBL) at (2.6,-1.4) {};
	
	\end{pgfonlayer}
	\begin{pgfonlayer}{edgelayer}
		\draw [style=inarrow, bend left=20, looseness=1.00] (n1) to (n2);
		\draw [style=inarrow, bend right=20, looseness=1.00] (n1) to (n3);
		\draw [style=inarrow, bend left=20, looseness=1.00] (n2) to (n4);
		\draw [style=inarrow, bend right=20, looseness=1.00] (n3) to (n4);
		\draw [style=inarrow] (n2) to (n3);
		\draw [style=simple] (ATL.center) to (ATR.center);
		\draw [style=simple] (ATR.center) to (ABR.center);
		\draw [style=simple] (ABR.center) to (ABL.center);
		\draw [style=simple] (ABL.center) to (ATL.center);
		\draw [style=simple] (MTL.center) to (MTR.center);
		\draw [style=simple] (MTR.center) to (MBR.center);
		\draw [style=simple] (MBR.center) to (MBL.center);
		\draw [style=simple] (MBL.center) to (MTL.center);
		\draw [style=inputarrow] (inI) to (n1);
		\draw [style=inputarrow] (outI') to (n4);
	\end{pgfonlayer}
\end{tikzpicture}
}
\]
These differ only in that the edge $e_5$ has been renamed $e_6$.  We could also rename nodes, but we chose this example for a specific reason.   We can define a similar category of open graphs using the machinery of decorated cospans.  The morphisms in this other category are again equivalence classes of open graphs---but with a finer equivalence relation, for which the above open graphs are \emph{not} equivalent!  Indeed, this other notion of equivalence between open graphs only allows us to rename nodes, not edges.

Now let us compare the decorated cospan category of open graphs.  We shall go into some detail here, since the problems we meet afflict a number of attempted applications of decorated
cospans in the published literature \cite{BF,BFP,BP,Fong2015}.  We start with a functor $F \maps \Fin\Set \to \Set$ that assigns to any finite set $N$ the collection of all \define{graph structures} on $N$, meaning graphs whose set of nodes is $N$.   A small issue immediately presents itself: as described, $F(N)$ is actually a proper class.  We can get around this in various ways.  For example, we can replace $\Fin\Set$ by an equivalent small category, and define a finite graph to be a diagram $s, t \maps E \to N$ in this category.  Henceforth we consider this done.  

The functor $F$ acts on morphisms as follows: given any function $f \maps N \to N'$, we say that $F(f) \maps F(N) \to F(N')$ maps the graph structure $s,t \maps E \to N$ to the graph structure
\[     f \circ s, f \circ t \maps E \to N' .\]
Thus, we use $f$ to rename the nodes and let the edges `go along for the ride'.

To obtain a symmetric monoidal category $F\Cospan$ as described in Section \ref{sec:intro}, we need to make $F$ into a symmetric lax monoidal functor from $(\Fin\Set, +)$ to $(\Set, \times)$.    There is an obvious choice of laxator
\[    \phi_{N,N'} \maps F(N) \times F(N') \to F(N+N')  \]
since there is a natural graph structure on $N+N'$ built from graph structures $s,t \maps E \to N$ and $s',t' \maps E' \to N'$: namely, $s+s', t+t' \maps E + E' \to N + N'$.  However, as pointed out by an anonymous referee in a paper by Moeller and Vasilakopoulou \cite{MV}, this diagram in the definition of lax monoidal functor may fail to commute:
\[
\begin{tikzpicture}[scale=1.5]
\node (A) at (0,0.5) {$(F(N) \times F(N')) \times F(N'')$};
\node (A') at (4.5,0.5) {$F(N) \times (F(N') \times F(N''))$};
\node (B) at (0,-0.25) {$F(N + N') \times F(N'')$};
\node (C) at (4.5,-0.25) {$F(N) \times F(N' + N'')$};
\node (C') at (0,-1) {$F((N + N') + N'')$};
\node (D) at (4.5,-1) {$F(N + (N' + N''))$};
\path[->,font=\scriptsize,>=angle 90]
(A) edge node[above]{$ $} (A')
(A) edge node[left]{$\phi_{N,N'} \times 1$} (B)
(A') edge node[right]{$1 \times \phi_{N',N''}$} (C)
(B) edge node[left]{$\phi_{N + N',N''}$} (C')
(C) edge node [right] {$\phi_{N,N' + N''}$} (D)
(C') edge node [above] {$ $} (D);
\end{tikzpicture}
\]
where the horizontal arrows are the associator in $(\Set,\times)$ and $F$ of the associator in $(\Fin\Set,$ $+)$, respectively.   Suppose we start at upper left with a triple of graph structures $s,t \maps E \to N$, $s',t' \maps E' \to N'$ and $s'',t'' \maps E'' \to N''$.   If we follow the arrows first down and  then across, we obtain a graph structure on $N + (N' + N'')$ where the set of edges is 
$(E + E') + E''$.   If instead we follow the arrows first across and then down, we obtain a graph structure where the set of edges is $E + (E' + E'')$.   These graph structures are different
if $(E + E') + E'' \ne E + (E' + E'')$.  The problem is that $(\Fin\Set, +)$ may not be a strict monoidal category.  We say ``may not" because we have replaced the original $(\Fin\Set,+)$ by an equivalent small category.  

Of course we can use Mac Lane's coherence theorem to choose an equivalent monoidal category that is both small and strict.   One can then prove $F$ becomes lax monoidal with $\phi$ as its laxator---but still not \emph{symmetric} lax monoidal.  The problem is that this diagram fails to commute:
\[
\begin{tikzpicture}[scale=1.5]
\node (A) at (0,0) {$F(N) \times F(N')$};
\node (B) at (2,0) {$F(N') \times F(N)$};
\node (C) at (0,-1) {$F(N+N')$};
\node (D) at (2,-1) {$F(N'+N)$};
\path[->,font=\scriptsize,>=angle 90]
(A) edge node[left]{$\phi_{N,N'}$} (C)
(B) edge node[right]{$\phi_{N',N}$} (D)
(A) edge node[above]{$ $} (B)
(C) edge node[below]{$ $} (D);
\end{tikzpicture}
\]
where the horizontal arrows are the braiding in $(\Set, \times)$ and $F$ of the braiding in $(\Fin\Set, +)$, respectively.    Suppose we start at upper left with a pair of graph structures $s,t \maps E \to N$ and $s',t' \maps E' \to N'$.  If we follow the arrows first down and then across we obtain a graph structure on $N' + N$ where the set of edges is $E+E'$, but if we follow the arrows first across and then down we obtain a graph structure where the set of edges is $E'+E$.  These graph structures are different in general, and we cannot cure this problem with further strictification: $(\Fin\Set,+)$ is not equivalent as a symmetric monoidal category to one that where the braiding is the identity.

As a result, the theory of decorated cospans only gives a monoidal category $F\Cospan$ \cite[Thm.\ 2.1.3]{CourThesis}.  An object of $F\Cospan$ is a finite set, while a morphism is an equivalence class of $F$-decorated cospans
\[
\begin{tikzpicture}[scale=1.5]
\node (A) at (0,0) {$S$};
\node (B) at (1,0) {$N$};
\node (C) at (2,0) {$T$,};
\node (D) at (3,0) {$G \in F(N)$.};
\path[->,font=\scriptsize,>=angle 90]
(A) edge node[above]{$i$} (B)
(C)edge node[above]{$o$}(B);
\end{tikzpicture}
\]
Such an $F$-decorated cospan is a way of describing an open graph from $S$ to $T$.   However, two such $F$-decorated cospans, say the above one and this:
\[
\begin{tikzpicture}[scale=1.5]
\node (A) at (0,0) {$S$};
\node (B) at (1,0) {$N'$};
\node (C) at (2,0) {$T$,};
\node (D) at (3,0) {$G' \in F'(N)$,};
\path[->,font=\scriptsize,>=angle 90]
(A) edge node[above]{$i$} (B)
(C)edge node[above]{$o$}(B);
\end{tikzpicture}
\]
are equivalent iff there is a bijection $f \maps N \to N'$ making this diagram commute:
\[
\begin{tikzpicture}[scale=1.5]
\node (E) at (3,-0.5) {$S$};
\node (F) at (5,-0.5) {$T$};
\node (G) at (4,0) {$N$};
\node (G') at (4,-1) {$N'$};
\path[->,font=\scriptsize,>=angle 90]
(F) edge node[above]{$o$} (G)
(G) edge node[left]{$f$} (G')
(E) edge node[above]{$i$} (G)
(E) edge node[below]{$i'$} (G')
(F) edge node[below]{$o'$} (G');
\end{tikzpicture}
\]
and such that $F(f)(G) = G'$.  It follows that the graphs $G = \left( s,t \maps E \to N \right)$ and $G' = \left( s', t' \maps E' \to N' \right)$ are isomorphic, but in a specific way: we must have $E' = E$, $s' = f \circ s$, and $t' = f \circ  t$.   Thus, two open graphs with different edge sets cannot be equivalent!

In short, the decorated cospan category of open graphs resembles the structured cospan category, but it is merely monoidal, not symmetric monoidal, and it has many morphisms for each morphism in the structured cospan category, for no particularly useful reason.  This `redundancy' is eliminated by the functor $J \maps F\Cospan \to {}_L \Csp(\Graph)$ that is the identity on objects and identifies isomorphic open graphs.  

In attempted applications so far, one often uses a decorated cospan category as the `syntax' for open systems of a particular kind, with the `semantics' given by a monoidal functor out of this category \cite{Fong2016}.  Often this functor factors through a structured cospan category that eliminates the redundancy in the morphisms of the structured cospan category.   We give some examples in the next section.

On the other hand, there are also useful decorated cospan categories that do not suffer from the problems we have described.  Some appear not to be structured cospan categories.  An example is the category of open dynamical systems described in Section \ref{subsec:RxNet}.  Furthermore, the theory of decorated cospans plays an important role in the more general theory of decorated corelations \cite{Fong2018,FongSarazola}. So, it also interesting to see if we can improve the theory of decorated cospans a bit to eliminate the problems we have seen.

In the case of open graphs, one cheap solution is to use a different symmetric lax monoidal 
functor, say $F' \maps (\Fin\Set,+) \to (\Set,\times)$, that sends any finite set $N$ to the set 
of \emph{isomorphism classes} of graph structures on $N$.   Here given two graph structures 
$s,t\maps E \to N$ and $s',t' \maps E' \to N$  on $N$, we define a \textbf{morphism} 
from the first to the second to be a function $f \maps E \to E'$ such that these diagrams 
commute:
\[
\begin{tikzpicture}[scale=1.5]
\node (A) at (0,0) {$E$};
\node (A') at (0,-1) {$E'$};
\node (B) at (1,0) {$N$};
\node (B') at (1,-1) {$N$};
\path[->,font=\scriptsize,>=angle 90]
(A) edge node[above]{$s$} (B)
(A') edge node[above]{$s'$} (B')
(A) edge node[left]{$f$} (A')
(B) edge node[right]{$1$} (B');

\node (C) at (2,0) {$E$};
\node (C') at (2,-1) {$E'$};
\node (D) at (3,0) {$N$};
\node (D') at (3,-1) {$N$};
\path[->,font=\scriptsize,>=angle 90]
(C) edge node[above]{$t$} (D)
(C') edge node[above]{$t'$} (D')
(C) edge node[left]{$f$} (C')
(D) edge node[right]{$1$} (D');
\end{tikzpicture}
\]
We obtain a category of graph structures on $N$ in this way, allowing us to define
isomorphism classes of these.  One can check that using the theory of 
decorated cospans we obtain a symmetric monoidal category $F'\Cospan$ that is
equivalent to ${}_L \Csp(\Graph)$.  

However, working with isomorphism classes of graph structures does not give a double 
category of decorated cospans that is equivalent to ${}_L \lCsp(\Graph)$.  We should really 
work with the \emph{category} of graph structures, not isomorphism classes of graph 
structures!   A clue to a better approach is to note that the forgetful functor $U \maps \Graph 
\to \Fin\Set$ is an opfibration, and the category of graph structures on a finite set $N$ is the 
fiber of this opfibration over $N$.   Thus,  the inverse Grothendieck construction gives a 
pseudofunctor $\tilde{F} \maps \Fin\Set \to \Cat$ sending each finite set $N$ to the category 
of graph structures on $N$.  Moreover, $\tilde{F}$ is symmetric lax monoidal from $(\Fin\Set, +)$ to $(\Cat, \times)$.

In a forthcoming paper with Vasilakopoulou \cite{BCV}, we extend the theory of decorated cospans to handle this sort of data.  That is, given a category $\A$ with finite colimits and a symmetric lax monoidal pseudofunctor $\tilde{F} \maps (\A , +) \to (\Cat, \times)$, we construct a symmetric monoidal double category $\tilde{F} \lCospan$ with decorated cospans as horizontal 1-cells.    Any such pseudofunctor also gives an opfibration $R \maps \X \to \A$ where $\X = \int \tilde{F}$ is defined by the Grothendieck construction.
We show that if the symmetric lax monoidal pseudofunctor $\tilde{F} \maps (\mathsf{A},+) \to (\mathbf{Cat},\times)$ factors through $(\mathbf{Rex},\times)$, the resulting opfibration $R \maps \X \to \A$ is also a right adjoint.  From the accompanying left adjoint $L \maps \A \to \X$, we  construct a symmetric monoidal double category ${}_L \lCsp(\X)$ of structured cospans.  Finally, we prove that this structured cospan double category ${}_L \lCsp(\X)$ is equivalent to the decorated cospan double category $\tilde{F} \lCospan$. Thus, we reconcile the theory of structured cospans and the theory of decorated cospan categories by enhancing the latter.




\section{Applications}
\label{sec:applications}

Decorated cospans have already been used to study electrical circuits \cite{BF},  Markov processes \cite{BFP}, and chemical reaction networks \cite{BP}, while structured cospans
have been used to study electrical circuits \cite{BCR} and Petri nets \cite{BM}.    Here we revisit this work and show that structured cospans can take the place of decorated cospans in many of these applications.   For structured cospans in graph rewriting, see
Cicala's thesis \cite{Cicala}.

\subsection{Circuits}
\label{subsec:circuits}

Building on work with Fong \cite{BF}, Coya, Rebro and the first author have used structured
cospans to describe electrical circuits with inputs and outputs \cite{BCR}.  The key idea is to use graphs with labeled edges.  The edge labels can stand for resistors with any chosen resistance, capacitors with any chosen capacitance, inductors with any chosen inductance, or other circuit elements such as voltage sources, current sources, diodes, and so on.  To study such
circuits quite generally we start by fixing any set $\La$ to serve as edge labels.
\begin{defn}
Given a set $\La$ of \define{labels}, an \define{$\La$-graph} is a graph $s,t\maps E\to N$
 equipped with a function $\ell \maps E \to \La$.
 A \define{morphism} from the $\La$-graph 
 \[ \xymatrix{\La & E \ar@<2.5pt>[r]^{s} \ar@<-2.5pt>[r]_{t} \ar[l]_{\ell} & N} \]
 to the $\La$-graph 
\[ \xymatrix{\La & E' \ar@<2.5pt>[r]^{s'} \ar@<-2.5pt>[r]_{t'} \ar[l]_{\ell'} & N'} \]
is a pair of functions $f \maps E \to E', g \maps N \to N'$ such that these diagrams commute:
\[
\begin{tikzpicture}[scale=1.5]
\node (A) at (0,0) {$E$};
\node (A') at (0,-1) {$E'$};
\node (B) at (1,0) {$N$};
\node (B') at (1,-1) {$N'$};
\path[->,font=\scriptsize,>=angle 90]
(A) edge node[above]{$s$} (B)
(A') edge node[above]{$s'$} (B')
(A) edge node[left]{$f$} (A')
(B) edge node[right]{$g$} (B');

\node (C) at (2,0) {$E$};
\node (C') at (2,-1) {$E'$};
\node (D) at (3,0) {$N$};
\node (D') at (3,-1) {$N'$};
\path[->,font=\scriptsize,>=angle 90]
(C) edge node[above]{$t$} (D)
(C') edge node[above]{$t'$} (D')
(C) edge node[left]{$f$} (C')
(D) edge node[right]{$g$} (D');

\node (E) at (4,-0.5) {$\La$};
\node (G) at (5,0) {$E$};
\node (G') at (5,-1) {$E'$.};
\path[->,font=\scriptsize,>=angle 90]
(G) edge node[above]{$\ell$} (E)
(G) edge node[right]{$f$} (G')
(G') edge node[below]{$\ell'$} (E);
\end{tikzpicture}
\]
We say such a morphism is \define{determined by its action on nodes} if $E' = E$ and
$f=1_E$.
\end{defn}

\begin{defn}
\label{defn:Lgraph}
We define $\Graph_\La$ to be the category of $\La$-graphs and morphisms between them, with composition given by
\[  (f, g) \circ (f',g') = (f \circ f' , g \circ g')  .\]
\end{defn}

When $\La = 1$, an $\La$-graph reduces to a graph and $\Graph_\La$ reduces
to the category $\Graph$ discussed in Section \ref{sec:decorated}.  We now generalize the key
ideas of that section from graphs to $\La$-graphs.   Everything works the same way, 
but following previous work \cite{BCR} we call an open $\La$-graph an `$\La$-circuit'.

There is a functor $U \maps \Graph_\La \to \Fin\Set$ that takes an $\La$-graph to its underlying set of nodes. This has a left adjoint $L \maps \Fin\Set \to \Graph_\La$ sending
any set to the $\La$-graph with that set of nodes and no edges.
Both $\Fin\Set$ and $\Graph_\La$ have colimits, and $L$ preserves them.  Thus Theorem  \ref{thm:LCsp(X) symmetric} gives us a symmetric monoidal double category ${}_L \lCsp(\Graph_\La)$.   Alternatively, we can use Corollary \ref{cor:LCsp(X) symmetric category} to create a symmetric monoidal category ${}_L \Csp(\Graph_\La)$ where:
\begin{itemize}
\item an object is a finite set,
\item a morphism is an isomorphism class of $\La$-circuits, where an
\define{$\La$-circuit} is a cospan in $\Graph_\La$ of this form:
\[
\begin{tikzpicture}[scale=1.5]
\node (A) at (0,0) {$L(S)$};
\node (B) at (1,0) {$G$};
\node (C) at (2,0) {$L(T)$,};
\path[->,font=\scriptsize,>=angle 90]
(A) edge node[above,left]{$$} (B)
(C)edge node[above]{$$}(B);
\end{tikzpicture}
\]
and an \define{isomorphism} of $\La$-circuits is a commutative diagram in $\Graph_\La$ 
of this form:
\[
\begin{tikzpicture}[scale=1.5]
\node (E) at (3,-0.5) {$L(S)$};
\node (F) at (5,-0.5) {$L(T)$};
\node (G) at (4,0) {$G$};
\node (G') at (4,-1) {$G'$};
\path[->,font=\scriptsize,>=angle 90]
(F) edge node[above]{$o$} (G)
(G) edge node[left]{$h$} (G')
(E) edge node[above]{$i$} (G)
(E) edge node[below]{$i'$} (G')
(F) edge node[below]{$o'$} (G');
\end{tikzpicture}
\]
where $h$ is an isomorphism.
\end{itemize} 
This category has a nice universal property, found by Rosebrugh, Sabadini and
Walters \cite{RSW}.  To state this, it is convenient to use the language of props.
Recall that a \define{prop} is a symmetric strict monoidal category whose objects are natural
numbers, with tensor product of objects given by addition.  An \define{algebra} of a prop $\T$ in a symmetric strict monoidal category $\C$ is a symmetric strict monoidal functor $A \maps \T \to \C$.   A \define{morphism} from the algebra $A \maps \T \to \C$ to the algebra
$A' \maps \T \to \C$ is a monoidal natural transformation $\alpha \maps A \To A'$.

\begin{lem} As a symmetric monoidal category, ${}_L \Csp(\Graph_\La)$ is equivalent to a prop $\Circ_\La$.    
\end{lem}

\begin{proof}
This is \cite[Proposition 4.3]{BCR}.
\end{proof}
\begin{prop} 
\label{prop:LCirc_algebra}
An algebra of $\Circ_\La$ in a symmetric strict monoidal category $\C$ is a 
special commutative Frobenius monoid in $\C$ whose underlying object $x$ is equipped
with an endomorphism $\ell \maps x \to x$ for each element $\ell \in \La$. 
A morphism of algebras of $\Circ_\La$ in $\C$ is a morphism of special commutative Frobenius monoids that also preserves all these endomorphisms.
\end{prop}

\begin{proof}
This was proved by Rosebrugh, Sabadini and Walters \cite{RSW}, and appears
in the above form in \cite[Proposition 7.2]{BCR}.  
\end{proof}

In applications to circuits, the morphisms $\ell \maps x \to x$ describe different
circuit elements, while the special commutative Frobenius monoid structure is used to split and join wires.  This framework is used to study a wide variety of electrical circuits in a paper with Coya and Rebro \cite{BCR}, so the reader can turn there for details.   To illustrate the ideas let us consider circuits of resistors, where a label in $\La = (0,\infty)$ serves to indicate the resistance of a resistor.   In this case a typical morphism from $1$ to $3$ in 
$\Circ_\La$ looks like this:
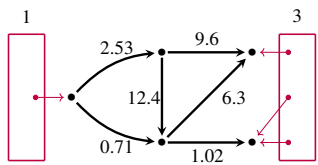
\[
\scalebox{0.6}{
\begin{tikzpicture}
	\begin{pgfonlayer}{nodelayer}
		\node [contact] (n1) at (-2,0) {$\bullet$};
		\node [style = none] at (-2.1,0.3) {$$};
		\node [contact] (n2) at (0,1) {$\bullet$};
		\node [style = none] at (0,1.3) {$$};
		\node [contact] (n3) at (0,-1) {$\bullet$};
		\node [style = none] at (0,-1.3) {$$};
		\node [contact] (n4) at (2,1) {$\bullet$};
		\node [style = none] at (2.1,0.3) {$$};
		\node [contact] (n5) at (2,-1) {$\bullet$};
		\node [style = none] at (2.1,0.3) {$$};
		
		\node [style = none] at (-1,1.1) {$2.53$};
		\node [style = none] at (-1,-1.1) {$0.71$};
		\node [style = none] at (1,1.3) {$9.6$};
		\node [style = none] at (1,-1.3) {$1.02$};
	     \node [style = none] at (-0.4,0) {$12.4$};
	     \node [style = none] at (1.6,0) {$6.3$};
		
		\node [style=none] (1) at (-3,0) {};
		\node [style=none] (4) at (3,0) {};
	
		\node [style=none] (ATL) at (-3.4,1.4) {};
		\node [style=none] (ATR) at (-2.6,1.4) {};
		\node [style=none] (ABR) at (-2.6,-1.4) {};
		\node [style=none] (ABL) at (-3.4,-1.4) {};

		\node [style=none] (X) at (-3,1.8) {$1$};
		\node [style=inputdot] (inI) at (-2.8,0) {};
		
		\node [style=none] (Z) at (3,1.8) {$3$};
	 \node [style=inputdot] (outI') at (2.8,1) {};
	 \node [style=inputdot] (outI'') at (2.8,0) {};
	 \node [style=inputdot] (outI''') at (2.8,-1) {};

		\node [style=none] (MTL) at (2.6,1.4) {};
		\node [style=none] (MTR) at (3.4,1.4) {};
		\node [style=none] (MBR) at (3.4,-1.4) {};
		\node [style=none] (MBL) at (2.6,-1.4) {};
	
	\end{pgfonlayer}
	\begin{pgfonlayer}{edgelayer}
		\draw [style=inarrow, bend left=20, looseness=1.00] (n1) to (n2);
		\draw [style=inarrow, bend right=20, looseness=1.00] (n1) to (n3);
		\draw [style=inarrow, bend left=0, looseness=1.00] (n2) to (n4);
		\draw [style=inarrow, bend right=0, looseness=1.00] (n3) to (n4);
		\draw [style=inarrow, bend right=0, looseness=1.00] (n3) to (n5);
		\draw [style=inarrow] (n2) to (n3);
		\draw [style=simple] (ATL.center) to (ATR.center);
		\draw [style=simple] (ATR.center) to (ABR.center);
		\draw [style=simple] (ABR.center) to (ABL.center);
		\draw [style=simple] (ABL.center) to (ATL.center);
		\draw [style=simple] (MTL.center) to (MTR.center);
		\draw [style=simple] (MTR.center) to (MBR.center);
		\draw [style=simple] (MBR.center) to (MBL.center);
		\draw [style=simple] (MBL.center) to (MTL.center);
		\draw [style=inputarrow] (inI) to (n1);
		\draw [style=inputarrow] (outI') to (n4);
		\draw [style=inputarrow] (outI'') to (n5);
		\draw [style=inputarrow] (outI''') to (n5);
	\end{pgfonlayer}
\end{tikzpicture}
}
\]
The edges here represent wires, and the positive real numbers labeling them describe
the resistance of the resistor on each wire.  The points in the boxes represent `terminals': that is, points where we allow ourselves to attach a wire from another circuit.   The points in the left box are called `inputs' and the points in the right box are called `outputs'.  In electrical engineering we associate two real numbers to each terminal, called `potential' and `current'.   Any circuit of resistors imposes a specific relation between the potentials and currents at its inputs and those at its outputs.  All these relations, for all circuits of resistors, can be described using a single functor as follows.

There is a symmetric monoidal category $\Fin\Rel_\R$ where the objects are finite-dimensional real vector spaces and a morphism from $V$ to $W$ is a \define{linear relation} from $V$ to $W$: that is, a relation $L \subseteq V \times W$ that is a linear subspace of $V \times W$.   Composition in $\Fin\Rel_\R$ is the usual composition of relations, and the symmetric monoidal structure is provided by direct sum.  

There is a symmetric monoidal functor
\[   \blacksquare \maps \Circ_\La \to \Fin\Rel_\R  \]
sending any finite set $S$ to the vector space $\R^S \oplus \R^S$ and
sending any circuit of resistors to the relation it imposes between the potentials and currents at its inputs and those at its outputs \cite[Section 9]{BCR}.   We can construct this using Proposition \ref{prop:LCirc_algebra}, by choosing a special commutative Frobenius monoid in $\Fin\Rel_\R$ whose underlying object is equipped with an endomorphism for each resistance $R \in (0,\infty)$.   The object $\R^2 \in \Fin\Rel_\R$ is a special commutative Frobenius monoid in a standard way \cite[Section\ 8]{BCR}, so we choose this one.   To define $\blacksquare \maps \Circ_\La \to \Fin\Rel_\R$ it then suffices to choose for each $R \in (0,\infty)$ a linear relation from $\R^2$ to itself.   We use this:
\[ \left\{(\phi_1,I_1,\phi_2,I_2) : \; I_1 = I_2 , \; \phi_2-\phi_1 = R I_1\right\} 
 \subseteq \R^2 \oplus \R^2.\]
This expresses two laws of electrical engineering.   Kirchhoff's current law says that the 
current flowing into a wire equals the current flowing out: $I_1 = I_2$.   Ohm's law says that the voltage across a wire with a resistor on it, $\phi_2 - \phi_1$, is equal to the current flowing through the wire times the resistance $R$ of that resistor.

Earlier work with Fong studied circuits using decorated rather than structured cospans \cite{BF}, and it fell afoul of the problems discussed in Section \ref{sec:decorated}.  We make no attempt to explain the results here, but we can quickly explain a corrected version of this decorated cospan category of circuits.   For any set $\La$, define an \define{$\La$-graph structure} on a finite set $N$ to be an $\La$-graph whose set of nodes is $N$.    If we use a small strict monoidal version of $(\Fin\Set,+)$, there is a lax monoidal functor
\[     F_\La \maps (\Fin\Set, +) \to (\Set, \times ) \]
assigning to each finite set $N$ the collection of all $\La$-graph structures on $N$.
The theory of decorated cospans \cite[Thm.\  2.1.3]{CourThesis} thus gives a monoidal category 
$F_{\!\La} \Cospan$ where:
\begin{itemize}
\item an object is a finite set,
\item a morphism is an equivalence class of $\La$-circuits
\[
\begin{tikzpicture}[scale=1.5]
\node (A) at (0,0) {$L(S)$};
\node (B) at (1,0) {$G$};
\node (C) at (2,0) {$L(T)$};
\path[->,font=\scriptsize,>=angle 90]
(A) edge node[above,left]{$$} (B)
(C)edge node[above]{$$}(B);
\end{tikzpicture}
\]
where two are equivalent if there is a commutative diagram in $\Graph_\La$ of this form:
\[
\begin{tikzpicture}[scale=1.5]
\node (E) at (3,-0.5) {$L(S)$};
\node (F) at (5,-0.5) {$L(T)$};
\node (G) at (4,0) {$G$};
\node (G') at (4,-1) {$G'$};
\path[->,font=\scriptsize,>=angle 90]
(F) edge node[above]{$o$} (G)
(G) edge node[left]{$h$} (G')
(E) edge node[above]{$i$} (G)
(E) edge node[below]{$i'$} (G')
(F) edge node[below]{$o'$} (G');
\end{tikzpicture}
\]
with $h$ an isomorphism that is determined by its action on nodes in the sense of Definition \ref{defn:Lgraph}.
\end{itemize}

The restriction that $h$ be determined by its action on nodes means that isomorphic 
$\La$-circuits can give different morphisms in $F_{\!\La} \Cospan$.  However, there is a
functor
\[    J \maps F_{\!\La} \Cospan \to \Circ_\La \]
that eliminates this redundancy: it is the identity on objects, and it maps each open
circuit to its isomorphism class.  Furthermore, $\Circ_\La$ is symmetric monoidal, while
$F_{\!\La} \Cospan$ is merely monoidal, due to the problem discussed in Section \ref{sec:decorated}.

\subsection{Petri nets}
\label{subsec:Petri}

Petri nets are widely used by computer scientists as a simple model of distributed, 
concurrent computation \cite{GiraultValk,Peterson}.   From the viewpoint of a category
theorist, a Petri net is a convenient way to present a simple sort of symmetric monoidal
category: namely, a \emph{commutative} monoidal category---a commutative monoid object in $\Cat$---that is free on some objects and morphisms \cite{MM}.   Recently Master and the first author studied `open' Petri nets using structured cospans \cite{BM}.  By composing and tensoring open Petri nets, we can build complicated Petri nets out of smaller pieces.   As we shall see, the semantics of open Petri nets is a nice illustration of our main method of describing maps between structured cospan categories, Theorem \ref{thm:symmetric_monoidal_double_functor}.

To define Petri nets \cite{MM} we start with the monad for commutative monoids, $\N \maps \Set \to \Set$.  Concretely, $\N[X]$ is the set of formal finite linear combinations of elements of $X$ with natural number coefficients.   The set $X$ naturally includes in $\N[X]$, and for any function $f \maps X \to Y$, there is a unique monoid homomorphism $\N[f] \maps \N[X] \to \N[Y]$ extending $f$.  

\begin{defn}
\label{defn:PetriNet}
We define a \define{Petri net} to be a pair of functions of the following form:
\[\xymatrix{ T \ar@<-.5ex>[r]_-t \ar@<.5ex>[r]^-s & \N[S]. } \]
We call $T$ the set of \define{transitions}, $S$ the set of \define{places}, $s$ the \define{source} function and $t$ the \define{target} function. 
A \define{morphism} from the Petri net $s,t\maps T \to \N[S]$ to the Petri net $s',t'\maps T' \to \N[S']$ is a pair of functions $f \maps T \to T', g \maps S \to S'$ such that the following diagrams commute:
	\[
	\xymatrix{ 
		T \ar[d]_f  \ar[r]^-{s} & \N[S] \ar[d]^-{\N[g]} \\	
		T' \ar[r]^-{s'} & \N[S'] 
	}
	\qquad
	\xymatrix{ 
		T \ar[d]_f  \ar[r]^-{t} & \N[S] \ar[d]^-{\N[g]} \\	
		T' \ar[r]^-{t'} & \N[S'] . 
	}
	\]
Let $\Petri$ be the category of Petri nets and Petri net morphisms, with composition 
defined by 
\[  (f, g) \circ (f',g') = (f \circ f' , g \circ g')  .\]
\end{defn}

It is commmon to draw a Petri net as a bipartite graph with the places as circles and the transitions as squares:
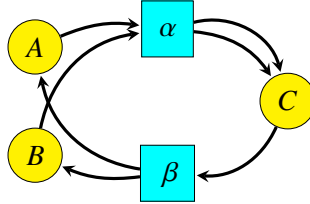
\[
\scalebox{0.8}{
\begin{tikzpicture}
	\begin{pgfonlayer}{nodelayer}
		\node [style=species] (A) at (-2.2, 0.9) {$A$};
		\node [style=species] (B) at (-2.2, -0.9) {$B$};
             \node [style=transition] (a) at (0, 1.2) {$\;\alpha\phantom{\big|}$}; 
		\node [style = species] (C) at (2, 0) {$C$};
	     \node [style = transition] (b) at (0, -1.2) {$\,\;\beta\phantom{\big|}$};
		
	\end{pgfonlayer}
	\begin{pgfonlayer}{edgelayer}
		\draw [style=inarrow, bend left=15, looseness=1.00] (A) to (a);
		\draw [style=inarrow, bend left=35, looseness=1.00] (B) to (a);
		\draw [style=inarrow, bend left=25, looseness=1.00] (a) to (C);
		\draw [style=inarrow, bend left=45, looseness=1.00] (a) to (C);
	     \draw [style=inarrow, bend left=35, looseness=1.00] (C) to (b);
	     	\draw [style=inarrow, bend left=35, looseness=1.00] (b) to (A);
		\draw [style=inarrow, bend left=15, looseness=1.00] (b) to (B);

	\end{pgfonlayer}
\end{tikzpicture}
}
\]
However, we must bear in mind that the edges in this graph are merely a device
for describing the source and target of each transition: there is not really a set of edges from 
a place to a transition or a transition to a place, but merely a \emph{number}.   For example, $\alpha$ above is a transition with $s(\alpha) = A+B$ and $t(\alpha) = 2C$.

Any Petri net has an underlying set of places.  Indeed there is a functor
$R \maps \Petri \to \Set$ that acts as follows on Petri nets and Petri net morphisms:
\[ 	
	\xymatrix{ 
		T \ar[d]_f  \ar@<-.5ex>[r]_-{t} \ar@<.5ex>[r]^-{s} & \N[S] \ar[d]^{\N[g] \quad \mapsto} &  S \ar[d]^{g} \\
		T' \ar@<-.5ex>[r]_-{t'} \ar@<.5ex>[r]^-{s'} & \N[S'] & \; S'.
	}
\]
To build a structured cospan category we need the left adjoint of $R$, and we need
$\Petri$ to have finite colimits.

\begin{lem}
\label{lem:L_Petri}
The functor $R$ has a left adjoint $L \maps \Set \to \Petri$ defined on sets and functions as follows:
\[
	\xymatrix{
		X \ar[d]_g^{ \quad \mapsto} & \emptyset \ar[d] \ar@<-.5ex>[r] \ar@<.5ex>[r] & \N[X] 		\ar[d]^{\N[g]} \\
		Y & \emptyset  \ar@<-.5ex>[r] \ar@<.5ex>[r] & \N[Y]
	}
\]
where the unlabeled maps are the unique maps of that type. 
\end{lem}

\begin{proof}
This is \cite[Lemma 11]{BM}.
\end{proof}

\begin{lem}
\label{lem:Petri_colimit}
The category $\Petri$ has small colimits.
\end{lem}

\begin{proof}
This is \cite[Lemma 15]{BM}.
\end{proof}

Thanks to these lemmas, Theorem \ref{thm:LCsp(X) symmetric} gives us a symmetric monoidal double category ${}_L\lCsp(\Petri)$, or $\Open(\Petri)$ for short, in which:
\begin{itemize}
\item an object is a set,
\item a vertical 1-morphism is a function,
\item a horizontal 1-cell from $X$ to $Y$ is an \define{open Petri net}, meaning a
cospan in $\Petri$ of this form:
\[
\begin{tikzpicture}[scale=1.5]
\node (A) at (0,0) {$L(X)$};
\node (B) at (1,0) {$P$};
\node (C) at (2,0) {$L(Y)$,};
\path[->,font=\scriptsize,>=angle 90]
(A) edge node[above,left]{$$} (B)
(C)edge node[above]{$$}(B);
\end{tikzpicture}
\]
\item a 2-morphism is a \define{map of open Petri nets}, meaning
a commutative diagram in $\Petri$ of this form:
\[
\begin{tikzpicture}[scale=1.5]
\node (E) at (3,0) {$L(X)$};
\node (F) at (5,0) {$L(Y)$};
\node (G) at (4,0) {$P$};
\node (E') at (3,-1) {$L(X')$};
\node (F') at (5,-1) {$L(Y')$.};
\node (G') at (4,-1) {$P'$};
\path[->,font=\scriptsize,>=angle 90]
(F) edge node[above]{$$} (G)
(E) edge node[left]{$L(f)$} (E')
(F) edge node[right]{$L(g)$} (F')
(G) edge node[left]{$h$} (G')
(E) edge node[above]{$$} (G)
(E') edge node[below]{$$} (G')
(F') edge node[below]{$$} (G');
\end{tikzpicture}
\]
\end{itemize} 
We can draw an open Petri net as a Petri net with maps from sets $X$ and $Y$ into
its set of places:
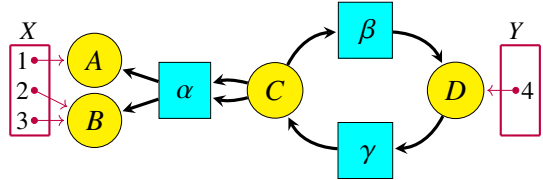
\[
\scalebox{0.8}{
\begin{tikzpicture}
	\begin{pgfonlayer}{nodelayer}
		\node [style=species] (A) at (-4, 0.5) {$A$};
		\node [style=species] (B) at (-4, -0.5) {$B$};;
           \node [style=transition] (a) at (-2.5, 0) {$\;\alpha\phantom{\big|}$}; 
		\node [style = species] (E) at (-1, 0) {$C$};
		\node [style = species] (F) at (2,0) {$D$};

	     \node [style = transition] (b) at (.5, 1) {$\,\;\beta\phantom{\big|}$};
		\node [style = transition] (c) at (.5, -1) {$\,\;\gamma\phantom{\big|}$};
		
		\node [style=empty] (X) at (-5.1, 1) {$X$};
		\node [style=none] (Xtr) at (-4.75, 0.75) {};
		\node [style=none] (Xbr) at (-4.75, -0.75) {};
		\node [style=none] (Xtl) at (-5.4, 0.75) {};
           \node [style=none] (Xbl) at (-5.4, -0.75) {};
	
		\node [style=inputdot] (1) at (-5, 0.5) {};
		\node [style=empty] at (-5.2, 0.5) {$1$};
		\node [style=inputdot] (2) at (-5, 0) {};
		\node [style=empty] at (-5.2, 0) {$2$};
		\node [style=inputdot] (3) at (-5, -0.5) {};
		\node [style=empty] at (-5.2, -0.5) {$3$};	
		
      	\node [style=empty] (Z) at (3, 1) {$Y$};
		\node [style=none] (Ztr) at (2.75, 0.75) {};
		\node [style=none] (Ztl) at (3.4, 0.75) {};
		\node [style=none] (Zbl) at (3.4, -0.75) {};
		\node [style=none] (Zbr) at (2.75, -0.75) {};

		\node [style=inputdot] (6) at (3, 0) {};
		\node [style=empty] at (3.2, 0) {$4$};	
		
	\end{pgfonlayer}
	\begin{pgfonlayer}{edgelayer}
		\draw [style=inarrow] (a) to (A);
		\draw [style=inarrow] (a) to (B);
	     \draw [style=inarrow, bend right=15, looseness=1.00] (E) to (a);
	     \draw [style=inarrow, bend left =15, looseness=1.00] (E) to (a);	
	     	\draw [style=inarrow, bend left=30, looseness=1.00] (E) to (b);
		\draw [style=inarrow, bend left=30, looseness=1.00] (b) to (F);
		\draw [style=inarrow, bend left=30, looseness=1.00] (c) to (E);
		\draw [style=inarrow, bend left=30, looseness=1.00] (F) to (c);	
		\draw [style=inputarrow] (1) to (A);
		\draw [style=inputarrow] (2) to (B);
		\draw [style=inputarrow] (3) to (B);
		\draw [style=inputarrow] (6) to (F);
		\draw [style=simple] (Xtl.center) to (Xtr.center);
		\draw [style=simple] (Xtr.center) to (Xbr.center);
		\draw [style=simple] (Xbr.center) to (Xbl.center);
		\draw [style=simple] (Xbl.center) to (Xtl.center);
		\draw [style=simple] (Ztl.center) to (Ztr.center);
		\draw [style=simple] (Ztr.center) to (Zbr.center);
		\draw [style=simple] (Zbr.center) to (Zbl.center);
		\draw [style=simple] (Zbl.center) to (Ztl.center);
	\end{pgfonlayer}
\end{tikzpicture}
}
\]
We explained composition and tensoring of open Petri nets using pictures
in Section \ref{sec:intro}.

Now we construct a structured cospan category $\Open(\CMC)$ of `open commutative monoidal categories' and a map
\[          \Open(F) \maps  \Open(\Petri) \to \Open(\CMC) .\]
\begin{defn}
A \define{commutative monoidal category} is a symmetric strict monoidal category where
all the braidings $a \otimes b \to b \otimes a$ are identities.  A \define{morphism} of commutative monoidal categories is a symmetric strict monoidal
functor.
\end{defn}

\begin{defn}
Let $\CMC$ be the category of commutative monoidal categories and morphisms between
them. 
\end{defn}

Any commutative monoidal category has an underlying set of objects.  Let $R' \maps \CMC \to \Set$ be the functor sending any commutative monoidal category to its underlying set of objects and any morphism to its underlying function on objects.   To build a structured cospan category of open commutative monoidal categories we use a left adjoint of $R'$, and we need $\CMC$ to have finite colimits.   

\begin{lem}
\label{lem:L'_CMC}
The functor $R'$ has a left adjoint $L' \maps \Set \to \CMC$ sending any set $S$ to the commutative monoidal category with $\N[S]$ as its commutative monoid of objects and with only identity morphisms.
\end{lem}

\begin{proof}
This is \cite[Lemma 9]{BM}.
\end{proof}

\begin{lem}
\label{lem:CMC_cocomplete}
The category $\CMC$ has small colimits.
\end{lem}

\begin{proof}
This can be shown in various ways; see \cite[Theorem 16]{BM} for two.
\end{proof}

Thanks to these lemmas, Theorem \ref{thm:LCsp(X) symmetric} gives us a symmetric monoidal double category ${}_{L'} \lCsp(\CMC)$, or $\Open(\CMC)$ for short, in which:
\begin{itemize}
\item an object is a set,
\item a vertical 1-morphism is a function,
\item a horizontal 1-cell from $X$ to $Y$ is an \define{open commutative monoidal category}, meaning a cospan in $\CMC$ of this form:
\[
\begin{tikzpicture}[scale=1.5]
\node (A) at (0,0) {$L'(X)$};
\node (B) at (1,0) {$\C$};
\node (C) at (2,0) {$L'(Y)$,};
\path[->,font=\scriptsize,>=angle 90]
(A) edge node[above,left]{$$} (B)
(C)edge node[above]{$$}(B);
\end{tikzpicture}
\]
\item a 2-morphism is a \define{map of open commutative monoidal categories}, meaning
a commutative diagram in $\CMC$ of this form:
\[
\begin{tikzpicture}[scale=1.5]
\node (E) at (3,0) {$L'(X)$};
\node (F) at (5,0) {$L'(Y)$};
\node (G) at (4,0) {$\C$};
\node (E') at (3,-1) {$L'(X')$};
\node (F') at (5,-1) {$L(Y')$.};
\node (G') at (4,-1) {$\C'$};
\path[->,font=\scriptsize,>=angle 90]
(F) edge node[above]{$$} (G)
(E) edge node[left]{$L'(f)$} (E')
(F) edge node[right]{$L'(g)$} (F')
(G) edge node[left]{$h$} (G')
(E) edge node[above]{$$} (G)
(E') edge node[below]{$$} (G')
(F') edge node[below]{$$} (G');
\end{tikzpicture}
\]
\end{itemize}
We can turn a Petri net $P = (s,t \maps T \to \N[S])$ into a commutative monoidal category $FP$ as follows.  We take the commutative monoid of objects $\Ob(FP)$ to be the free commutative monoid on $S$.    We construct the commutative monoid of morphisms $\Mor(FP)$ as follows.  First we generate morphisms recursively:
\begin{itemize}
\item for every transition $\tau \in T$ we include a morphism $\tau \maps s(\tau) \to t(\tau)$;
\item for any object $a$ we include a morphism $1_a \maps a \to a$;
\item for any morphisms $f \maps a \to b$ and $g \maps a' \to b'$ we include a morphism denoted $f+g \maps a +a' \to b +b'$ to serve as their tensor product;
\item for any morphisms $f \maps a \to b$ and $g \maps b \to c$ we include a morphism $g\circ f \maps a \to c$ to serve as their composite.
\end{itemize}
Then we mod out by an equivalence relation on morphisms that imposes the laws of a commutative monoidal category, obtaining the commutative monoid $\Mor(FP)$.		
		
Let $F \maps \Petri \to \CMC$ be the functor that makes the following assignments on Petri nets and morphisms:
	\[
	\xymatrix{ 
		T \ar[d]_f  \ar@<-.5ex>[r]_{t} \ar@<.5ex>[r]^{s} & \N[S] \ar[d]^{\N[g] \quad \mapsto} &  FP \ar[d]^{F(f,g)} \\
		T' \ar@<-.5ex>[r]_{t'} \ar@<.5ex>[r]^{s'} & \N[S'] & FP'.
	}
	\]
Here $F(f,g) \maps FP \to FP'$ is defined on objects by $\N [g]$. On morphisms, $F(f,g)$ is the unique map extending $f$ that preserves identities, composition, and the tensor product.

\begin{lem}
\label{lem:F_Petri}
The functor
\[ F \maps \Petri \to \CMC \] 
is a left adjoint.
\end{lem}

\begin{proof}
This is a special case of \cite[Theorem 5.1]{Master}.
\end{proof}

We thus obtain a triangle of left adjoint functors, which commutes up to natural isomorphism:
\[
\begin{tikzpicture}[scale=1.5]
\node (A) at (0,0) {$\Set$};
\node (B) at (1.3,0) {$\Petri$};
\node (D) at (1.3,-1.3) {$\CMC$};
\node (F) at (0.8,-0.4) {$\alpha \! \NEarrow$};
\path[->,font=\scriptsize,>=angle 90]
(A) edge node [above] {$L$} (B)
(B) edge node [right]{$F$} (D)
(A) edge node [below] {$L'$} (D);
\end{tikzpicture}
\]
As a result we obtain:

\begin{thm}
\label{thm:functoriality}
There is a symmetric monoidal double functor 
\[          \Open(F) \maps \Open(\Petri) \to \Open(\CMC) \]
that is the identity on objects and vertical 1-morphisms and makes the following assignments on horizontal 1-cells and 2-morphisms:
  \[ \xymatrix{	LX \ar[r]^i \ar[d]_{Lf} & P \ar[d]_{h} & LY \ar[l]_o \ar[d]^{Lg \qquad {\Huge{\mapsto}} \quad } & &
  L'X \ar[r]^{F(i) \alpha_X} \ar[d]_{L'f} & FP \ar[d]_{Fh} & L'Y \ar[l]_{F(o) \alpha_Y} \ar[d]^{L'g } 
   \\
			  	LX' \ar[r]^{i'} & P'  & LY' \ar[l]_{o'} & &
			  	L'X' \ar[r]^{F(i') \alpha_{X^\prime}}  & FP'  & L'Y'. \ar[l]_{F(o') \alpha_{Y^\prime}} }
\]
\end{thm}
\begin{proof}
The triangle above is a degenerate case of the square studied in Theorem \ref{thm:double_functor}:
\[
\begin{tikzpicture}[scale=1.5]
\node (A) at (0,0) {$\Set$};
\node (B) at (1,0) {$\Petri$};
\node (C) at (0,-1) {$\Set$};
\node (D) at (1,-1) {$\CMC$};
\node (E) at (0.5,-0.5) {$\alpha \NEarrow$};
\path[->,font=\scriptsize,>=angle 90]
(A) edge node[above]{$L$} (B)
(B) edge node [right]{$F$} (D)
(C) edge node [above] {$L'$} (D)
(A) edge node[left]{$1$}(C);
\end{tikzpicture}
\]
and applying that theorem we obtain the desired result.
\end{proof}

In the language of computer science, the commutative monoidal category $FP$ provides an `operational semantics' for the Petri net $P$: morphisms in this category are processes allowed by the Petri net. The above theorem says that this semantics is compositional.  That is, if we write $P$ as a composite (or tensor product) of smaller open Petri nets, $FP$ will be the composite (or tensor product) of the corresponding open commutative monoidal categories. 

\subsection{Petri nets with rates}
\label{subsec:RxNet}

Chemists often describe collections of chemical reactions using `reaction networks'.  
They have a standard formalism for obtaining a dynamical system from any reaction network where each reaction is labeled by a positive real number called its `rate constant' \cite{HJ}.  Reaction networks equipped with rate constants are equivalent to Petri nets where every transition is labeled by a positive real number.   These are sometimes called `stochastic' 
Petri nets, and they are used not only in chemistry but also biology and other fields \cite{Haas,MBCDF}.   

Pollard and the first author studied `open' reaction networks using decorated cospans \cite{BP}.  Here we show how to translate some of that work into the language of structured cospans.   We need a finiteness condition in many applications, so we include that from the start.

\begin{defn}
A \define{Petri net with rates} is a Petri net $s,t \maps T \to \N[S]$ where $S$ and $T$ are finite sets, together with a function $r \maps T \to (0,\infty)$.  We call $r(\tau)$ the \define{rate constant} of the transition $\tau \in T$. A \define{morphism} from the Petri net with rates
\[\xymatrix{(0,\infty) & T  \ar[l]_-r \ar@<-.5ex>[r]_-t \ar@<.5ex>[r]^-s & \N[S] } \]
to the Petri net with rates
\[\xymatrix{(0,\infty) & T' \ar[l]_-{r'} \ar@<-.5ex>[r]_-{t'} \ar@<.5ex>[r]^-{s'} & \N[S'] } \]
 is a morphism  $f \maps T \to T', g \maps S \to S'$  of the underlying Petri nets such that the following diagram also commutes:
 \[
\begin{tikzpicture}[scale=1.5]
\node (E) at (3,-0.5) {$(0,\infty)$};
\node (G) at (4,0) {$T$};
\node (G') at (4,-1) {$T'$};
\path[->,font=\scriptsize,>=angle 90]
(G) edge node[right]{$f$} (G')
(G) edge node[above]{$r$} (E)
(G') edge node[below]{$r'$} (E);
\end{tikzpicture}
\]
Let $\Petri_r$ be the category of Petri nets with rates and morphisms between them,
with composition defined by 
\[  (f, g) \circ (f',g') = (f \circ f' , g \circ g')  .\]
\end{defn}

There is a functor $R \maps \Petri_r \to \Set$ that sends any Petri net with rates to its underlying set of places
\[ 	
	\xymatrix{ 
		(0,\infty)  \ar[d]_1 & T  \ar[l]_-r \ar[d]_f  \ar@<-.5ex>[r]_-{t} \ar@<.5ex>[r]^-{s} & \N[S] \ar[d]^{\N[g] \quad \mapsto} &  S \ar[d]^{g} \\
		(0,\infty) & T'  \ar[l]_-{r'} \ar@<-.5ex>[r]_-{t'} \ar@<.5ex>[r]^-{s'} & \N[S'] & \; S'.
	}
\]
To build a structured cospan category we use the left adjoint of $R$, and we need
$\Petri_r$ to have finite colimits.

\begin{lem}
\label{lem:L_Petri_r}
The functor $R$ has a left adjoint $L \maps \Set \to \Petri_r$ defined on sets and functions as follows:
\[
	\xymatrix{
		X \ar[d]_f^{ \quad \mapsto} & (0,\infty)  \ar[d]_1 & \emptyset  \ar[l]  \ar[d] \ar@<-.5ex>[r] \ar@<.5ex>[r] & \N[X] 		\ar[d]^{\N[f]} \\
		Y & (0,\infty) & \emptyset \ar[l] \ar@<-.5ex>[r] \ar@<.5ex>[r] & \N[Y]
	}
\]
where the unlabeled maps are the unique maps of that type. 
\end{lem}

\begin{proof}
This is easily checked from the definitions.
\end{proof}

\begin{lem}
\label{lem:Petri_r_colimit}
The category $\Petri_r$ has finite colimits.
\end{lem}

\begin{proof}
Note that $\Petri_r$ is equivalent to the comma category  $f/g$ where $f \maps \Fin\Set \to \Fin\Set$ is the identity and $g \maps \Fin\Set \to \Fin\Set$ is $(0,\infty) \times \N[-]^2$.   Whenever $A$ and $B$ are have finite colimits, $f \maps A \to C$ preserves finite colimits and $g \maps B \to C$ is any functor, then $f/g$ has finite colimits \cite[Section 5.2, Theorem 3]{BR}.  
\end{proof}

As a consequence of these lemmas, Corollary \ref{cor:LCsp(X) symmetric category} gives a symmetric monoidal category $_L \Csp(\Petri_r)$, or $\mathsf{Open}(\Petri_r)$ for
short, in which:
\begin{itemize}
\item an object is a finite set,
\item a morphism is an isomorphism class of open Petri nets with rates, where an
\define{open Petri net with rates} is a cospan in $\Petri_r$ of this form:
\[
\begin{tikzpicture}[scale=1.5]
\node (A) at (0,0) {$L(X)$};
\node (B) at (1,0) {$P$};
\node (C) at (2,0) {$L(Y)$,};
\path[->,font=\scriptsize,>=angle 90]
(A) edge node[above,left]{$$} (B)
(C)edge node[above]{$$}(B);
\end{tikzpicture}
\]
and an \define{isomorphism} of such is a commutative diagram in $\Petri_r$ of this form:
\[
\begin{tikzpicture}[scale=1.5]
\node (E) at (3,-0.5) {$L(X)$};
\node (F) at (5,-0.5) {$L(Y)$};
\node (G) at (4,0) {$P$};
\node (G') at (4,-1) {$P'$};
\path[->,font=\scriptsize,>=angle 90]
(F) edge node[above]{$o$} (G)
(G) edge node[left]{$h$} (G')
(E) edge node[above]{$i$} (G)
(E) edge node[below]{$i'$} (G')
(F) edge node[below]{$o'$} (G');
\end{tikzpicture}
\]
where $h$ is an isomorphism.
\end{itemize} 

Pollard and the first author \cite{BP} used decorated cospans to construct a symmetric monoidal category $\RxNet$ equivalent to $\mathsf{Open}(\Petri_r)$.  They avoided the `redundancy problem' using a trick explained in Section \ref{sec:decorated}.   Namely, they used a symmetric lax monoidal functor $F' \maps (\Fin\Set,+) \to (\Set,\times)$ sending any finite set $S$ to the set of \emph{isomorphism classes} of Petri nets with rates having $S$ as their set of places.

Pollard and the first author then constructed a symmetric monoidal functor from $\RxNet$ to a category $\Dynam$ of `open dynamical systems', and a further symmetric monoidal functor from $\Dynam$ assigning to each open dynamical system the relation between its inputs and outputs that holds in steady state.   Thanks to the equivalence between  $\RxNet$ and $\mathsf{Open}(\Petri_r)$, these functors can also be construed as functors out of the structured cospan category $\mathsf{Open}(\Petri_r)$.  Thus, structured cospans can be used to study both the dynamics and the steady states of open systems of chemical reactions.

\appendix
\section{Double Categories}
\label{appendix}

What follows is a brief review of double categories. A more detailed exposition can be found in the work of Grandis and Par\'e \cite{GP1,GP2}, and for monoidal double categories the work of Hansen and Shulman \cite{HS,Shul2007,Shul2010}.  We use `double category' to mean what earlier authors called a `pseudo double category'.

\begin{defn}
\label{defn:double_category}
A \textbf{double category} is a weak category in $\Cat$. More explicitly, a double category $\lD$ consists of:
\begin{itemize}
\item a \define{category of objects} $\lD_0$ and a \define{category of arrows} $\lD_1$,
\item  \define{source} and \define{target} functors
\[  S,T \colon \lD_1 \to \lD_0 ,\]
an \define{identity-assigning} functor
\[  U\colon \lD_0 \to \lD_1 ,\]
and a \define{composition} functor
\[ \odot \colon \lD_1 \times_{\lD_0} \lD_1 \to \lD_1 \]
where the pullback is taken over $\lD_1 \xrightarrow[]{T} \lD_0 \xleftarrow[]{S} \lD_1$,
such that
\[  S(U_{A})=A=T(U_{A}) , \quad
	S(M \odot N)=SN, \quad
   T(M \odot N)=TM, \]
\item natural isomorphisms called the \define{associator}
\[ \alpha_{N,N',N''} \maps (N \odot N') \odot N'' \toiso N \odot (N' \odot N'') , \]
the \define{left unitor}
\[		\lambda_N \maps U_{T(N)} \odot N \toiso N, \]
and the \define{right unitor}
\[  \rho_N \maps N \odot U_{S(N)} \toiso N \]
such that $S(\alpha), S(\lambda), S(\rho), T(\alpha), T(\lambda)$ and $T(\rho)$ are all identities, and such that the standard coherence axioms hold: the pentagon identity for the 
associator and the triangle identity for the left and right unitor. 
\end{itemize}
If $\alpha$, $\lambda$ and $\rho$ are identities, we call $\lD$ a \define{strict} double category.
\end{defn}

Objects of $\lD_0$ are called \define{objects} and morphisms in $\lD_0$ are called \define{vertical 1-morphisms}.  Objects of $\lD_1$ are called \define{horizontal 1-cells} of $\lD$ and morphisms in $\lD_1$ are called \define{2-morphisms}.   A morphism $\alpha \maps M \to N$ in $\lD_1$ can be drawn as a square:
\[
\begin{tikzpicture}[scale=1]
\node (D) at (-4,0.5) {$a$};
\node (E) at (-2,0.5) {$b$};
\node (F) at (-4,-1) {$c$};
\node (A) at (-2,-1) {$d$};
\node (B) at (-3,-0.25) {$\Downarrow \alpha$};
\path[->,font=\scriptsize,>=angle 90]
(D) edge node [above]{$M$}(E)
(E) edge node [right]{$g$}(A)
(D) edge node [left]{$f$}(F)
(F) edge node [above]{$N$} (A);
\end{tikzpicture}
\]
where $f = S\alpha$ and $g = T\alpha$.  If $f$ and $g$ are identities we call $\alpha$ a \textbf{globular 2-morphism}.  These give rise to a bicategory:

\begin{defn}
\label{defn:horizontal}
Let $\lD$ be a double category. Then the \define{horizontal bicategory} of $\lD$, denoted $H(\lD)$, is the bicategory consisting of objects, horizontal 1-cells and globular 2-morphisms of $\lD$.
\end{defn}

We have maps between double categories, and also transformations between maps:

\begin{defn}
\label{defn:double_functor}
Let $\lA$ and $\lB$ be double categories. A \define{double functor} $\lF \maps \lA \to \lB$ consists of:
\begin{itemize}
\item functors $\lF_0 \maps \lA_0 \to \lB_0$ and $\lF_1 \maps \lA_1 \to \lB_1$ obeying the following
equations: 
\[S \circ \lF_1 = \lF_0 \circ S, \qquad T \circ \lF_1 = \lF_0 \circ T,\]
\item natural isomorphisms called the \define{composition comparison}: 
\[   \phi(N,N') \maps \lF_1(N) \odot \lF_1(N') \toiso \lF_1(N \odot N') \]
and the \define{unit comparison}:
\[  \phi_{A} \maps U_{\lF_0 (A)} \toiso \lF_1(U_A) \]
whose components are globular 2-morphisms, 
\end{itemize}
such that the following diagrams commmute:
\begin{itemize} 
\item a diagram expressing compatibility with the associator: 
\[\xymatrix{ 	(\lF_1(N) \odot \lF_1(N')) \odot \lF_1(N'') \ar[d]_{\phi (N,N') \odot 1} \ar[r]^{\alpha} & \lF_1(N) \odot (\lF_1(N') \odot \lF_1(N'')) \ar[d]^{1 \odot \phi(N',N'')} \\
			\lF_1(N \odot N') \odot \lF_1(N'') \ar[d]_{\phi(N \odot N', N'')} & \lF_1(N) \odot \lF_1(N' \odot N'') \ar[d]^{\phi(N, N'\odot N'')}\\
\lF_1((N \odot N') \odot N'') \ar[r]^{\lF_1(\alpha)} & \lF_1(N \odot (N' \odot N'')) }	\]
\item two diagrams expressing compatibility with the left and right unitors:
	\[
	\begin{tikzpicture}[scale=1.5]
	\node (A) at (1,1) {$\lF_1(N) \odot U_{\lF_0(A)}$};
	\node (A') at (1,0) {$\lF_1(N) \odot \lF_1(U_{A})$};
	\node (C) at (3.5,1) {$\lF_1(N)$};
	\node (C') at (3.5,0) {$\lF_1(N \odot U_A)$};
	\path[->,font=\scriptsize,>=angle 90]
	(A) edge node[left]{$1 \odot \phi_{A}$} (A')
	(C') edge node[right]{$\lF_1(\rho_N)$} (C)
	(A) edge node[above]{$\rho_{\lF_1(N)}$} (C)
	(A') edge node[above]{$\phi(N,U_{A})$} (C');
	\end{tikzpicture}
	\]
	\[
	\begin{tikzpicture}[scale=1.5]
	\node (B) at (5.5,1) {$U_{\lF_0(B)} \odot \lF_1(N)$};
	\node (B') at (5.5,0) {$\lF_1(U_{B}) \odot \lF_1(N)$};
	\node (D) at (8,1) {$\lF_1(N)$};
	\node (D') at (8,0) {$\lF_1(U_{B} \odot N).$};
		\path[->,font=\scriptsize,>=angle 90]
		(B) edge node[left]{$\phi_{B} \odot 1$} (B')
	(B') edge node[above]{$\phi(U_{B},N)$} (D')
	(B) edge node[above]{$\lambda_{\lF_1(N)}$} (D)
	(D') edge node[right]{$F_1(\lambda_{N})$} (D);
	\end{tikzpicture}
	\]
\end{itemize}
If the 2-morphisms $\phi(N,N')$ and $\phi_A$ are identities for all $N,N' \in \lA_1$ and 
$A \in \lA_0$, we say $\lF \maps \lA \to \lB$ is a \define{strict} double functor.  If on the other hand we drop the requirement that these 2-morphisms be invertible, we call $\lF$ a \define{lax} double
functor.
\end{defn}
	
\begin{defn}
Let $\lF \maps \lA \to \lB$ and $\lG \maps \lA \to \lB$ be lax double functors. A \define{transformation} $\beta \maps \lF \Rightarrow \lG$ consists of natural transformations $\beta_0 \maps \lF_0 \Rightarrow \lG_0$ and $\beta_1 \maps \lF_1 \Rightarrow \lG_1$ (both usually written as $\beta$) such that 
		\begin{itemize}
			\item $S( \beta_M) = \beta_{SM}$ and $T(\beta_M) = \beta_{TM}$ for any $M \in \lA_1$, 
			\item $\beta$ preserves the composition comparison, and
			\item $\beta$ preserves the unit comparison.
		\end{itemize}
\end{defn}
Grandis and Par\'e define a 2-category $\Dbl$ of double categories, double functors, and transformations \cite{GP2}.  This has finite products.  In any 2-category with finite products we can define a pseudomonoid \cite{DayStreet}.
	
\begin{defn}
\label{defn:monoidal_double_category}
A \define{monoidal double category} is a pseudomonoid in $\Dbl$. Explicitly, a monoidal double category is a double category equipped with double functors $\otimes \maps \lD \times \lD \to \lD$ and $I \maps 1 \to \lD$ where $1$ is the terminal double category, along with invertible transformations called the \define{associator}:
\[  \alpha \maps \otimes \, \circ \; (1_{\lD} \times \otimes ) \Rightarrow \otimes \; \circ \; (\otimes \times 1_{\lD}) ,\]
\define{left unitor}:
\[ \ell \maps \otimes \, \circ \; (1_{\lD} \times I) \Rightarrow 1_{\lD} ,\]
and \define{right unitor}:
\[ r \maps \otimes \,\circ\; (I \times 1_{\lD}) \Rightarrow 1_{\lD} \]
satisfying the pentagon axiom and triangle axioms.
\end{defn}

This definition neatly packages a large quantity of information.   
In detail, a monoidal double category $\lD$ is a double category with:
\begin{itemize}
\item monoidal structures on both categories $\lD_0$ and $\lD_1$ (with tensor product denoted $\otimes$, associator $a$, left unitor $\ell$ and right unitor $r$), and
\item the structure of a double functor on $\otimes$:
that is, invertible globular 2-morphisms
\[ \chi \maps (M_1\otimes N_1)\odot (M_2\otimes N_2) \toiso
(M_1\odot M_2)\otimes (N_1\odot N_2)\]
\[ \mu \maps U_{A\otimes B} \stackrel{\sim}{\longrightarrow} (U_A \otimes U_B)\]
making these diagrams commute:
\[\xymatrix{
			((M_1\otimes N_1)\odot (M_2\otimes N_2)) \odot (M_3\otimes N_3) \ar[r]^-{\chi \odot 1} \ar[d]_{\alpha}
			& ((M_1\odot M_2)\otimes (N_1\odot N_2)) \odot (M_3\otimes N_3) \ar[d]^{\chi}\\
			(M_1\otimes N_1)\odot ((M_2\otimes N_2) \odot (M_3\otimes N_3)) \ar[d]_{1 \odot \chi} &
			((M_1\odot M_2)\odot M_3) \otimes ((N_1\odot N_2)\odot N_3) \ar[d]^{\alpha \otimes \alpha}\\
			(M_1\otimes N_1) \odot ((M_2\odot M_3) \otimes (N_2\odot N_3))\ar[r]^-{\chi} &
			(M_1\odot (M_2\odot M_3)) \otimes (N_1\odot (N_2\odot N_3))}
\]
\[\xymatrix{(M\otimes N) \odot U_{C\otimes D} \ar[r]^-{1 \odot \mu} \ar[d]_{\rho} &
			(M\otimes N)\odot (U_C\otimes U_D) \ar[d]^{\chi}\\
			M\otimes N\ar@{<-}[r]^-{\rho \otimes \rho} & (M\odot U_C) \otimes (N\odot U_D)}
\]
\[\xymatrix{U_{A\otimes B}\odot (M\otimes N)  \ar[r]^-{\mu \odot 1} \ar[d]_{\lambda} &
			(U_A\otimes U_B)\odot (M\otimes N) \ar[d]^{\chi}\\
			M\otimes N\ar@{<-}[r]^-{\lambda \otimes \lambda} & (U_A \odot M) \otimes (U_B\odot N)}
			\]
\end{itemize}
We also demand the following properties:
\begin{itemize}
\item If $I$ is the monoidal unit of $\lD_0$ then $U_I$ the monoidal unit of $\lD_1$.
\item The functors $S$ and $T$ are strict monoidal.
\item The associator and left and right unitors for the tensor product in $\lD$ are transformations between double functors.  In other words, the following six diagrams commute:
  \[\xymatrix{
    ((M_1\otimes N_1)\otimes P_1) \odot ((M_2\otimes N_2)\otimes P_2) \ar[r]^{a \odot a}\ar[d]_{\chi} &
    (M_1\otimes (N_1\otimes P_1)) \odot (M_2\otimes (N_2\otimes P_2)) \ar[d]^{\chi}\\
    ((M_1\otimes N_1) \odot (M_2\otimes N_2)) \otimes (P_1\odot P_2) \ar[d]_{\chi \otimes 1} &
    (M_1\odot M_2) \otimes ((N_1\otimes P_1)\odot (N_2\otimes P_2))\ar[d]^{1 \otimes \chi} \\
    ((M_1\odot M_2) \otimes (N_1\odot N_2)) \otimes (P_1\odot P_2) \ar[r]^{a} &
    (M_1\odot M_2) \otimes ((N_1\odot N_2)\otimes (P_1\odot P_2))}\]
  \[\xymatrix{
    U_{(A\otimes B)\otimes C} \ar[r]^{U_{a}} \ar[d]_{\mu} & U_{A\otimes (B\otimes C)} \ar[d]^{\mu}\\
    U_{A\otimes B} \otimes U_C \ar[d]_{\mu \otimes 1} & U_A\otimes U_{B\otimes C}\ar[d]^{1 \otimes \mu}\\
    (U_A\otimes U_B)\otimes U_C \ar[r]^{a} & U_A\otimes (U_B\otimes U_C) }\]
  \[\vcenter{\xymatrix{
      (U_I\otimes M)\odot (U_I\otimes N)\ar[r]^{\chi} \ar[d]_{\ell \odot \ell} &
      (U_I \odot U_I) \otimes (M\odot N) \ar[d]^{\lambda \otimes 1}\\
      M\odot N \ar@{<-}[r]^{\ell} &
      U_I\otimes (M\odot N) }}\]
  \[\vcenter{\xymatrix{U_{I\otimes A} \ar[r]^{\mu}\ar[dr]_{U_{\ell}} & U_I\otimes U_A \ar[d]^{\ell}\\
      & U_A}}\]
       \[\vcenter{\xymatrix{
      (M\otimes U_I)\odot (N\otimes U_I)\ar[r]^{\chi}\ar[d]_{r \odot r} &
      (M\odot N)\otimes (U_I \odot U_I) \ar[d]^{1 \otimes \rho}\\
      M\odot N \ar@{<-}[r]^{r} &
      (M\odot N)\otimes U_I }}\]
  \[\vcenter{\xymatrix{U_{A\otimes I} \ar[r]^{\mu} \ar[dr]_{U_r} & U_A\otimes U_I \ar[d]^{r}\\
       & U_A.}}\]
	\end{itemize}

\begin{defn}
\label{defn:symmetric_monoidal_double_category}
A \define{braided monoidal double category} is a braided pseudomonoid in $\Dbl$.  Explicitly,
it is a monoidal double category equipped with an invertible transformation
\[ \beta \maps \otimes \Rightarrow \otimes \circ \tau \]
called the \define{braiding}, where $\tau \maps \lD \times \lD \to \lD \times \lD$ is the twist double functor sending pairs in the object and arrow categories to the same pairs in the opposite order. The braiding is required to satisfy the usual two hexagon identities \cite{MacLane}.  If the braiding is self-inverse we say that $\lD$ is a \define{symmetric monoidal double category}. 
\end{defn}

\begin{defn}
\label{defn:monoidal_double_functor}
A \define{monoidal lax double functor} $\lF \colon \lC \to \lC'$ between monoidal double categories $\lC$ and $\lC'$ is a lax double functor $\lF \maps \lC \to \lC'$ such that
	\begin{itemize}
		\item $\lF_0$ and $\lF_1$ are monoidal functors,
		\item $S'\lF_1= \lF_0S$ and $T'\lF_1 = \lF_0T$ as monoidal functors, and
		\item the composition and unit comparisons $\phi(N_1,N_2) \maps \lF_1(N_1) \odot \lF_1(N_2) \to \lF_1(N_1\odot N_2)$ and $\phi_A \maps U_{\lF_0 (A)} \to \lF_1(U_A)$ are monoidal natural transformations.
	\end{itemize}
The monoidal lax double functor is \define{braided} if $\lF_0$ and $\lF_1$ are braided monoidal functors and \define{symmetric} if they are symmetric monoidal functors. 
\end{defn}

We also have transformations between double functors:

\begin{defn}
A \define{double transformation} $\Phi \maps \lF \Rightarrow \mathbb{G}$ between two double functors $\lF \maps \lX \to \lX'$ and $\mathbb{G} \maps \lX \to \lX'$ consists of two natural transformations $\Phi_0 \maps \lF_0 \Rightarrow \mathbb{G}_0$ and $\Phi_1 \maps \lF_1 \Rightarrow \mathbb{G}_1$ such that for all horizontal 1-cells $M$ we have that $S({\Phi_1}_{M})={\Phi_0}_{S(M)}$ and $T({\Phi_1}_{M})={\Phi_0}_{T(M)}$ and for composable horizontal 1-cells $M$ and $N$, we have 
\[
\begin{tikzpicture}[scale=1.5]
\node (A) at (0,1) {$\lF(x)$};
\node (E) at (1.5,1) {$\lF(y)$};
\node (C) at (3,1) {$\lF(z)$};
\node (A') at (0,-1) {$\mathbb{G}(x)$};
\node (C') at (3,-1) {$\mathbb{G}(z)$};
\node (F) at (0,0) {$\lF(x)$};
\node (G) at (3,0) {$\lF(z)$};
\node (H) at (4,0) {$=$};
\node (I) at (6.5,1) {$\lF(y)$};
\node (J) at (5,0) {$\mathbb{G}(x)$};
\node (K) at (6.5,0) {$\mathbb{G}(y)$};
\node (L) at (8,0) {$\mathbb{G}(z)$};
\node (M) at (1.5,0.5) {$\Downarrow \lF_{M,N}$};
\node (N) at (1.5,-0.5) {$\Downarrow {\Phi_1}_{M \odot N}$};
\node (B) at (5,1) {$\lF(x)$};
\node (B') at (5,-1) {$\mathbb{G}(x)$};
\node (D) at (8,1) {$\lF(z)$};
\node (D') at (8,-1) {$\mathbb{G}(z)$};
\node (O) at (5.75,0.5) {$\Downarrow {\Phi_1}_M$};
\node (P) at (7.25,0.5) {$\Downarrow {\Phi_1}_N$};
\node (Q) at (6.5,-0.5) {$\Downarrow {\mathbb{G}}_{M,N}$};
\path[->,font=\scriptsize,>=angle 90]
(A) edge node[left]{$1$} (F)
(F) edge node [left]{${\Phi_0}_x$} (A')
(C) edge node[right]{$1$} (G)
(G) edge node [right]{${\Phi_0}_z$} (C')
(A) edge node[above]{$\lF(M)$} (E)
(F) edge node [above]{$\lF(M \odot N)$} (G)
(E) edge node[above] {$\lF(N)$} (C)
(A') edge node[above]{$\mathbb{G}(M \odot N)$} (C')
(B) edge node[left]{${\Phi_0}_x$} (J)
(J) edge node [left]{$1$} (B')
(I) edge node [left]{${\Phi_0}_y$} (K)
(K) edge node [above]{$\mathbb{G}(N)$} (L)
(J) edge node [above]{$\mathbb{G}(M)$} (K)
(B') edge node[above]{$\mathbb{G}(M \odot N)$} (D')
(B) edge node[above]{$\lF(M)$} (I)
(I) edge node [above] {$\lF(N)$} (D)
(D) edge node[right]{${\Phi_0}_z$} (L)
(L) edge node[right]{$1$}(D');
\end{tikzpicture}
\]
\[
\begin{tikzpicture}[scale=1.5]
\node (A) at (0,1) {$\lF(x)$};
\node (C) at (3,1) {$\lF(x)$};
\node (A') at (0,-1) {$\mathbb{G}(x)$};
\node (C') at (3,-1) {$\mathbb{G}(x)$};
\node (F) at (0,0) {$\lF(x)$};
\node (G) at (3,0) {$\lF(x)$};
\node (H) at (4,0) {$=$};
\node (J) at (5,0) {$\mathbb{G}(x)$};
\node (L) at (8,0) {$\mathbb{G}(x)$};
\node (M) at (1.5,0.5) {$\Downarrow \lF_U$};
\node (N) at (1.5,-0.5) {$\Downarrow {\Phi_1}_{U_x}$};
\node (B) at (5,1) {$\lF(x)$};
\node (B') at (5,-1) {$\mathbb{G}(x)$};
\node (D) at (8,1) {$\lF(x)$};
\node (D') at (8,-1) {$\mathbb{G}(x)$};
\node (O) at (6.5,0.5) {$\Downarrow U_{{\Phi_0}_x}$};
\node (Q) at (6.5,-0.5) {$\Downarrow \mathbb{G}_U$};
\path[->,font=\scriptsize,>=angle 90]
(A) edge node[left]{$1$} (F)
(F) edge node [left]{${\Phi_0}_x$} (A')
(C) edge node[right]{$1$} (G)
(G) edge node [right]{${\Phi_0}_x$} (C')
(A) edge node[above]{$U_{\lF(x)}$} (C)
(F) edge node [above]{$\lF(U_x)$} (G)
(A') edge node[above]{$\mathbb{G}(U_x)$} (C')
(B) edge node[left]{${\Phi_0}_x$} (J)
(J) edge node [left]{$1$} (B')
(J) edge node [above]{$U_{\mathbb{G}(x)}$} (L)
(B') edge node[above]{$\mathbb{G}(U_x)$} (D')
(B) edge node[above]{$U_{\lF(x)}$} (D)
(D) edge node[right]{${\Phi_0}_x$} (L)
(L) edge node[right]{$1$}(D');
\end{tikzpicture}
\]
We call $\Phi_0$ the \define{object component} and $\Phi_1$ the \define{arrow component} of the double transformation $\Phi$.
\end{defn}

One can also define monoidal, braided monoidal and symmetric monoidal double transformations, but since we do not use these, we refer the reader to Hansen and Shulman for the details \cite[Definition 2.15]{HS}.

\end{document}